\documentclass[12pt]{amsart}
\usepackage{color}
\usepackage{stmaryrd}
\usepackage{amssymb,latexsym,amsmath, amsfonts}
\usepackage{bbm}

\usepackage{color}
\usepackage{slashed}

\newcommand{\mL}{\mathcal L}

\newcommand{\mM}{\mathcal M}
\newcommand{\Lw}{\mathcal{L}_{\varphi}}

\newcommand{\nw}[1]{\| {#1} \|_{\Lw}}

\newcommand{\N}{{\mathbb N}}
\newcommand{\Z}{{\mathbb Z}}
\newcommand{\R}{{\mathbb R}}
\newcommand{\C}{{\mathbb C}}

\newcommand{\e}{\mathrm{e}}

\newtheorem{thm}{Theorem}[section]
\newtheorem{lem}[thm]{Lemma}
\newtheorem{cor}[thm]{Corollary}
\newtheorem{prop}[thm]{Proposition}

\theoremstyle{definition}
\newtheorem{rem}[thm]{Remark}
\newtheorem{dfn}[thm]{Definition}
\newtheorem{example}[thm]{Example}

\title[Dixmier traces on weak operator ideals]{Dixmier traces and residues on weak operator ideals}
\author{Magnus Goffeng, Alexandr Usachev}
\address{
Email: goffeng@chalmers.se, usachev@chalmers.se\newline
\indent Department of Mathematical Sciences\newline 
\indent Chalmers University of Technology and \newline 
\indent University of Gothenburg\newline 
\indent SE-412 96 Gothenburg\newline 
\indent Sweden\newline}
\date{\today}

\begin{document}

\begin{abstract}
We develop the theory of modulated operators in general principal ideals of compact operators. For Laplacian modulated operators we establish Connes' trace formula in its local Euclidean model and a global version thereof. It expresses Dixmier traces in terms of the vector-valued Wodzicki residue. We demonstrate the applicability of our main results in the context of log-classical pseudo-differential operators, studied by Lesch, and a class of operators naturally appearing in noncommutative geometry.

\end{abstract}

\maketitle

\section{Introduction}

Computing the spectral asymptotics of physically and geometrically significant operators is a classical problem in spectral theory. In situations where this is not possible, because of technical limitations or substantial reasons such as a lack of classicality (in the sense of asymptotic expansions), it is nevertheless of interest to obtain quantitative information about the spectrum by the computation of an ``averaged" spectral asymptotics. One of the tools to proceed is via the employment of Dixmier traces.\\

\subsection*{Background} Throughout this section $A$ is an operator acting on a separable Hilbert space.
Let $\Phi:(0,\infty)\to (0,\infty)$ be a concave function such that $\lim_{t\to\infty} \Phi(t)=+\infty$ and let  
\begin{equation*}
\mM_\Phi := \left\{A \ \text{is compact and } \|A\|_{\mM_\Phi} := \sup_{n\ge0} \frac1{\Phi(n+1)} \sum_{k=0}^n \mu(k,A) < \infty \right\}
\end{equation*}
be the corresponding \emph{Lorentz ideal}. Here $\mu(k,A)$ denotes the singular values of the operator $A$. If
\begin{equation*}
\lim_{t\to \infty} \frac{\Phi(2t)}{\Phi(t)}=1,
\end{equation*} then it was shown by J.~Dixmier~\cite{D} that for a suitable choice of extended limiting procedure $\omega$, the functional
\begin{equation*}
{\rm Tr}_\omega(A) = \omega\left(\frac1{\Phi(n+1)} \sum_{k=0}^n \mu(k,A)\right), \ 0\le A \in \mM_\Phi,
\end{equation*}
extends by linearity to a trace on $\mM_\Phi$. Such traces are called \emph{Dixmier traces}. Dixmier traces provide an ``averaged" spectral asymptotics in the sense that ${\rm Tr}_\omega(A)=c$ whenever $\lim_{k\to \infty}\frac{\mu(k,A)}{\varphi(k)}=c$ for $\varphi=\Phi'$. There is no general method to compute the sequence of singular values of an operator, and the computation of Dixmier traces is in general an equally difficult problem. There are several approaches to this problem in concrete situations that produce manageable results describing the ``averaged" spectral asymptotics explicitly. One of them is to find formulas relating Dixmier traces to the $\zeta$-function of an operator or to its heat expansion. Such formulas first appeared in~\cite{CM95} for $\mM_\Phi$ with $\Phi(t)=\log(1+t)$ and then were further investigated in~\cite{CRSS, SZ, CS12, SUZ3, SUZ4}. Recently, such formulae were proved in the case of general ideals $\mM_\Phi$~\cite{GS} (with some mild and natural assumptions on $\Phi$).
 
Another approach, which is more relevant to the present paper, originated in Connes' paper~\cite{C_AF}.
To state this result we recall the notion of weak-ideals. For  a decreasing function $\varphi : [0,\infty) \to (0,\infty)$ such that $\sup_{t>0} \frac{\varphi(t)}{\varphi(2t)} <\infty$ we define the ideal of compact operators
\begin{equation}
\Lw := \left\{ A \ \text{compact and  }   \nw{A} := \sup_{n\ge0} \frac{\mu(n,A)}{\varphi(n)} < \infty \right\}.
\end{equation}

When $\varphi(t)=\frac1{1+t}$, the set $\Lw$ is called the weak trace-class ideal and denoted by $\mathcal L_{1,\infty}$.
It is easy to see that $\Lw \subset \mM_\Phi$ with $\varphi = \Phi'$. The Dixmier traces on $\Lw$ are the restrictions of Dixmier traces from the corresponding $\mM_\Phi$. 

\begin{thm}[Connes' trace formula, Theorem 1 of \cite{C_AF}]
Every compactly supported classical pseudo-differential operator $A : C_c^\infty(\mathbb R^d) \to C_c^\infty(\mathbb R^d)$ of order $-d$ extends to a compact linear operator belonging to $\mathcal L_{1,\infty}(L_2(\mathbb R^d))$  and 
$${\rm Tr}_\omega(A) = \frac{1}{d (2\pi)^d} {\rm Res}_W(A),$$
where ${\rm Res}_W(A)$ is the Wodzicki residue of $A$ and ${\rm Tr}_\omega$ is any Dixmier trace.
\end{thm}

This theorem reduces the computation of Dixmier traces to that of the Wodzicki residue, which in turn is defined from the principal symbol of an operator and is therefore more accessible to computations. The spectral asymptotics of positive pseudo-differential operators was well known prior to Connes' trace formula by means of a Weyl law, see for instance \cite[Chapter 29.1]{hormaIV}, but Connes' trace formula conceptualized Dixmier traces as an ``integral" in noncommutative geometry. There were several attempts to extend and generalise Connes' trace formula (see e.g.~\cite{BN}, ~\cite{Fathizadeh} and ~\cite{grubshd} for the treatment of anisotropic pseudo-differential operators, the perturbed Laplacian on the noncommutative two tori and manifolds with boundaries, respectively).

\subsection*{Recent developments}
A recent extension of the above theorem appeared in~\cite{KLPS} (see also~\cite{LSZ}). It expresses the Dixmier traces of so-called Laplacian modulated operators (for rigorous definitions, see Section~\ref{subsec:stronglymod} below) in terms of a vector-valued Wodzicki residue, which again depends on the symbol of an operator. We state the main result of~\cite{KLPS} here. Let $\ell^\infty(\N)$ denote the Banach space of bounded sequences with supremum norm.

\begin{thm}[Theorem 1.2, \cite{KLPS}]
\label{CTTv1}
Let $A$ be a compactly supported Laplacian modulated operator with Hilbert-Schmidt symbol $p_A\in L^2(\R^d\times \R^d)$. We have $A \in \mathcal L_{1,\infty}(L_2(\mathbb R^d))$.  Moreover,
for a Dixmier trace $\mathrm{Tr}_\omega$,
$$
\mathrm{Tr}_\omega(A) = \frac1{d(2\pi)^d} \omega({\rm Res}(A))
$$
where $\omega\in \ell^\infty(\N)^*$ is a state and
$${\rm Res}(A) := \left\{\frac1{\log(2+n)}\int_{\mathbb R^d} \int_{|\xi|\le n^{1/d}} p_A(x,\xi) d\xi\,dx\right\}_{n\in \N}\in \ell^\infty(\N),$$ 
is the vector-valued residue of $A$.
\end{thm}

We note that the above result compared to the original Connes theorem does not assume the classicality of the symbol nor any smoothness of it. Indeed, the class of Laplacian modulated operators is larger than the class of classical pseudo-differential operators (of suitable order).
More recently, this result was extended to allow for the computations of \textit{all} positive traces~\cite{SSUZ}. The common value of all positive traces is in~\cite{SSUZ} described in terms of the classical notion of almost convergence (introduced by G.~G.~Lorentz in 1948) of the vector-valued Wodzicki residue ${\rm Res}(A)\in \ell^\infty(\N)$.

\subsection*{Motivation} 
Although Theorem~\ref{CTTv1} treats a wide class of operators, it only considers the case when the operator $A$ falls into the weak trace-class ideal $\mathcal L_{1,\infty}$. There is a large number of examples where a physically or geometrically important operator does not belong to $\mathcal L_{1,\infty}$, but rather to $\Lw$ for some more general function $\varphi$. Here we list two of such examples:

(i) The class of $\log$-classical pseudo-differential operators was considered in~\cite{Lesch}. Whereas the symbol of a classical pseudo-differential operator has an asymptotic expansion in terms of homogeneous functions, that of a $\log$-classical pseudo-differential operator has an expansion in terms of homogeneous functions multiplied by the power of the logarithm. Such operators belong to more general ideals $\Lw$. We discuss this example in details in Section~\ref{sec:CTTonmanifolds};

(ii) Recently, several $\theta$-summable spectral triples have been constructed on purely infinite $C^*$-algebras arising from crossed product constructions~\cite{DGMW16,GM16,MS16}. These constructions take their starting point on a closed manifold $M$, or in some cases more exotic geometries, and a (semi)group action forces the spectral growth to be sub-polynomial, even logarithmic. This is dictated by Connes' tracial obstruction to finite summability. An interesting feature of these spectral triples $(\mathcal{A},\mathcal{H},D)$ is that for the $a$ in the subalgebra of smooth functions $C^\infty(M)\subseteq \mathcal{A}$, the commutators $[D,a]$ are not just bounded but belong to a general weak ideal. This phenomena potentially opens up for possibly new and exciting noncommutative geometric situations. The results of this paper develops some of the tools needed to deal with these situations.

We have collected some more examples in Subsection~\ref{subsec:ex} below.

\subsection*{Main results} 
The purpose of this paper is to develop a theory allowing the computation of Dixmier traces for a wide class of operators from $\Lw$ in term of their Hilbert-Schmidt symbol.

One of the most important technical notions studied in this paper is that of modulated operators in general weak ideals. Modulated operators in the weak trace ideal $\mathcal{L}_{1,\infty}$ were extensively studied by Lord-Sukochev-Zanin (see \cite{LSZ}). The proof of Connes' trace formula in \cite{LSZ} (see Theorem~\ref{CTTv1} above) relies heavily on properties of modulated operators in $\mathcal{L}_{1,\infty}$. We generalize this notion to $\varphi$-modulated operators in general weak ideals $\Lw$. Much of the theory needs to be carefully redone, and relies heavily on the theory of functions of regular variation \cite{RegVar}. The theory of $\varphi$-modulated operators is more delicate than the corresponding theory in $\mathcal{L}_{1,\infty}$, often due to the function $\varphi(t)=1/t$ having coincidental properties such as being both its own inverse and its own reciprocal. In general weak ideals, we separate the notion of being modulated into two separate cases: weakly and strongly, see Definition \ref{weaklymod} and Definition \ref{mod}, respectively. These two notions are compared in Lemma \ref{mod_w-mod}. Being weakly or strongly modulated is a regularity condition, as seen from the work \cite{GG}, and to compute Dixmier traces we require an additional spectral condition in general weak ideals; we call this condition having $\varphi$-reasonable decay (see Definition \ref{phidec}).

The main result of the paper is the following generalisation of Connes' trace formula and its version for closed manifolds. For Euclidean spaces, it is stated as Theorem \ref{CTT_R} in the bulk of the text and  for closed manifolds as Theorem~\ref{CTT}.

\begin{thm}
\label{main}
Let $W$ be a $d$-dimensional inner product space and let $\varphi$ be a decreasing function smoothly varying of index $-1$. Assume that $G\in \Lw(L_2(W))$ is a compactly supported weakly $\varphi$-Laplacian modulated operator with $L_2$-symbol $p_G$ of $\varphi$-reasonable decay. Then, as $n\to \infty$,
$$\int_{W} \int_{\langle \xi\rangle \le n^{1/d}} p_G(x,\xi)\,\mathrm{d}\xi \mathrm{d}x=O(\Phi(n)),$$
and for every extended limit $\omega$ on $\ell_\infty$ we have
$${\rm Tr}_\omega(G) = \omega \left(\frac1{\Phi (n+1)} \int_{W} \int_{\langle \xi\rangle\le n^{1/d}} p_G(x,\xi)\,\mathrm{d}\xi \mathrm{d}x\right).$$
The same formula holds also when $M$ is a closed Riemannian manifold and $G\in \Lw(L^2(M))$ satisfies the conditions above.
\end{thm}

Note that Theorem~\ref{main} compared to Theorem~\ref{CTTv1} requires additionally that the symbol has $\varphi$-reasonable decay. For $\varphi(t)=O(t^{-1})$, $\varphi$-reasonable decay is automatic for strongly modulated operators. In general, this is unclear. Seeing that this is the case, we have chosen to assume the weakest possible modulation property in parallel to having $\varphi$-reasonable decay. 

As a demonstration of our main results we prove in Theorem~\ref{cor_logcl} that Dixmier traces of log-classical pseudo-differential operators considered in~\cite{Lesch} are proportional the higher Wodzicki residues. It is of general interest to find operators with exotic $\zeta$-functions. Lesch~\cite{Lesch} proved that if $A$ is a log-classical pseudo-differential operator and $P$ is an elliptic, invertible and positive classical pseudo-differential operator, then the $\zeta$-function ${\rm Tr}(AP^{-s})$ admits a meromorphic extension to the complex plane and has a finite order pole at $s=0$. The residue at this pole computes the higher Wodzicki residue. Another way to look at the $\zeta$-function is to consider the expression ${\rm Tr}(A^{1+s})$ for a positive operator $A$~\cite{CRSS, SZ, GS}. We show that for log-classical pseudo-differential operators the higher Wodzicki residue can be computed using this function, too. However, in general, this $\zeta$-function does not admit a meromorphic extension to a neighbourhood of $s=1$ (see Remark \ref{remfromDima}).

\subsection*{Organisation of the paper} $\,$

1. In Section~\ref{sec:prelsection} we collect some well known facts about weak ideals and provide examples of operators therein. We also give some preliminary results on weak ideals and recall the theory of functions of regular variation.

2. In Section~\ref{subsec:mod} we introduce the notion of operators in $\mathbb{B}(\mathcal{H})$ being weakly $\varphi$-modulated with respect to some auxiliary operator $V$ (see Definition \ref{weaklymod}). This notion extends the concept of weakly modulated operators from~\cite{GG} that implicitly appeared in~\cite[Lemma 11.2.9]{LSZ}. The main result of this section (Theorem~\ref{weaklymodform}) provides the formula which computes Dixmier traces of weakly $\varphi$-modulated operators through the expectation values.

3. In Section~\ref{subsec:stronglymod} we develop the theory of modulated operators in the general context. We introduce the notion of operators in $\mathbb{B}(\mathcal{H})$ that are \textit{strongly} $\varphi$-modulated with respect to some auxiliary operator $V$ (see Definition \ref{mod}). This notion extends the concept of modulated operators from~\cite{LSZ}. In this section we study properties of such operators and prove that under additional assumption on $\varphi$ any strongly $\varphi$-modulated operator is, in fact, weakly $\varphi$-modulated (see Definition \ref{someavass} and Lemma \ref{mod_w-mod}).

4. In Section~\ref{subsec:symbprop} we restrict our considerations to the particular case of operators in $\mathbb{B}(L_2(W))$, where $W$ is a $d$-dimensional real inner product space and study Laplacian modulated operators, that is those which are $\varphi$-modulated with respect to $\varphi[(1-\Delta)^{d/2}]$. We characterise such operators in terms of their Hilbert-Schmidt symbol. These results lay the ground for our first main result - our generalization of Connes' trace formula proved in Section~\ref{cttsection} for Euclidean spaces.

5. In Section~\ref{psipseudos} we develop a pseudo-differential calculus that widely generalises Lesch's log-classical pseudo-differential operators. With the help of these techniques we study more general operators -- locally $\varphi$-Laplacian modulated operators on manifolds -- in Section~\ref{sec:localizing}. Finally, in Section~\ref{sec:CTTonmanifolds} we prove our second main result --  our generalization of Connes' trace formula on closed manifolds (see Theorem~\ref{CTT}). As a consequence of this result we show that Dixmier traces of Lesch's log-classical pseudo-differential operators can be computed by means of the higher Wodzicki's residues introduced in~\cite{Lesch}.

6. In Section~\ref{sec:examples} we discuss the applications our results to noncommutative geometry. In particular, we express Dixmier traces of the product of commutators in terms of the integral of their symbols.

\subsection*{Acknowledgements} 

Both authors were supported by the Swedish Research Council Grant 2015-00137 and Marie Sklodowska Curie Actions, Cofund, Project INCA 600398. We thank Bogdan Nica for discussions leading up to Example \ref{fromBogdan}, Grigori Rozenblum for pointing us to the Example \ref{fromSimon}, Dima Zanin for providing us with the example in Remark \ref{remfromDima}, and Bram Mesland for discussions concerning logarithmic dampening in Section~\ref{sec:examples}. The authors also thank Robin Deeley and Heiko Gimperlein for useful feedback.

\section{Preliminaries}
\label{sec:prelsection}

\subsection{Weak operator ideals and Lorentz ideals}
Throughout the paper we denote by $\mathbb B(\mathcal H)$ (resp., $\mathbb{K}(\mathcal{H})$) the algebra of all bounded (resp., compact) linear operators on a separable Hilbert space $\mathcal{H}$.

The ideals of operators of relevance in this paper are the weak ideals associated with a decreasing function $\varphi : [0,\infty) \to (0,\infty)$. For a compact operator $A$ on a separable Hilbert space $\mathcal{H}$, we let $(\mu(n,A))_{n\in \N}$ denote its sequence of singular values and let $n_A(s):= {\rm Tr} (E_{|A|}(s,\infty)),$ where $E_{|A|}$ is a spectral projection of $|A|$. It can be shown that
$$\mu(n,A) = \inf\{ s>0 \ : \ n_A(s)\le n\} \ \text{and} \ \mu(n,A)\le s \Leftrightarrow n_A(s)\le n.$$

We define the set of compact operators
\begin{equation}
\label{weak}
\Lw := \left\{ A \in \mathbb{K}(\mathcal{H}) :  \nw{A} := \sup_{n\ge0} \frac{\mu(n,A)}{\varphi(n)} < \infty \right\}.
\end{equation}
The dependence on the Hilbert space $\mathcal{H}$ will be implicit; we write it out only when we wish to emphasize which Hilbert space we use. Occasionally, we consider the space 
$$\mathcal{L}_\varphi(\mathcal{H},\mathcal{H}')=\{A\in \mathbb{K}(\mathcal{H},\mathcal{H}'): (A^*A)^{1/2}\in \Lw(\mathcal{H}')\}=\{A\in \mathbb{K}(\mathcal{H},\mathcal{H}'): (AA^*)^{1/2}\in \Lw(\mathcal{H})\},$$
for two possibly different separable Hilbert spaces $\mathcal{H}$ and $\mathcal{H}'$.

We will throughout the paper implicitly assume that $\varphi : [0,\infty) \to (0,\infty)$ is a positive decreasing function such that 
\begin{equation}
\label{quasi}
\sup_{t>0} \frac{\varphi(t)}{\varphi(2t)} <\infty.
\end{equation}
The assumption that $\varphi$ is decreasing and satisfies Equation \eqref{quasi} guarantees that $\Lw$ is an ideal in the $C^*$-algebra of bounded operators $\mathbb{B}(\mathcal{H})$. We call $\Lw$ the \emph{weak ideal} associated with $\varphi$. The condition~\eqref{quasi} guarantees that the ideal $\Lw$ is a quasi-Banach ideal in the quasi-norm $\nw{\cdot}$. When referring to a weak ideal, we mean one of the form $\Lw$ with $\varphi$ satisfying \eqref{quasi}.

There is a larger ideal of operators associated with the function $\varphi$. We set 
$$\Phi(t):= \int_0^t \varphi(s) \mathrm{d}s +c_{\varphi},$$ 
for some constant $c_\varphi\geq 0$. Typically the constant $c_{\varphi}$ is chosen to be zero. Since the von Neumann algebra $\mathbb{B}(\mathcal{H})$ is atomic, adding this constant does not change the setting, but it slightly simplifies the notations.

The \emph{Lorentz ideal} defined from $\Phi$ is the set 
\begin{equation}
\label{Lor}
\mM_\Phi := \left\{ A \in \mathbb{K}(\mathcal{H}) :  \|A\|_{\mM_\Phi} := \sup_{n\ge0} \frac1{\Phi(n+1)} \sum_{k=0}^n \mu(k,A) < \infty \right\}.
\end{equation}
The set $\mM_\Phi$ forms an ideal in the $C^*$-algebra of bounded operators $\mathbb{B}(\mathcal{H})$. It forms a Banach space in the norm $\|\cdot\|_{\mM_\Phi}$. The ideal $\mM_\Phi$ can be defined as in Equation \eqref{Lor} for an arbitrary increasing concave function $\Phi$. 

Below we list some typical examples of functions $\varphi$ and $\Phi$.

\begin{example}
\label{standardex}
When $\varphi(t)=\frac1{\e+t}$ the ideal $\Lw$ is the well studied weak-trace class ideal $\mL_{1,\infty}$. In this case, $\Phi(t)=\log(\e+t)$. The corresponding Lorentz ideal $\mM_\Phi$ is denoted by $\mM_{1,\infty}$.
\end{example}

\begin{example}
\label{logkont}
Another interesting example comes from the function $\varphi(t)=\frac{\log^k(\e+t)}{\e+t}$, $k\in \Z$, $k\neq-1$. We use the notation $\log^k(\e+t)=(\log(\e+t))^k$. In this case, 
$$\Phi(t)=\frac{\log^{k+1}(\e+t)}{k+1}.$$
\end{example}

\begin{example}
\label{oneontlog}
The function $\varphi(t)=\frac{1}{(\e+t)\log(\e+t)}$ will also be of interest (corresponding to the excluded case $k=-1$ in Example \ref{logkont}). In this case, $\Phi(t)=\log(\log(\e+t))$.
\end{example}

\begin{example}
\label{exp}
The function $\varphi=\Phi'$, where $\Phi(t) = \e^{\log^\beta (\e+t)}$ with $0<\beta<1$. This function appeared in~\cite{SUZ2}.
\end{example}

\begin{example}
\label{oneonlog}
The function $\varphi(t)=\frac{1}{\log(\e+t)}$ is of interest in noncommutative geometry. In this case, $\Phi(t)$ is a logarithmic integral, behaving like $\frac{t}{\log(t)} (1+o(1))$ as $t\to \infty$ (in fact, as the prime counting function).
\end{example}

For $q>0$ we denote by $\Lw^{(q)}$ the \emph{$q$-convexification} of $\Lw$, that is
$\Lw^{(q)}$ consists of all operators $A$ such that $|A|^q \in \Lw$. In particular, 
$$\Lw^{(q)}=\mathcal{L}_{\sqrt[q]{\varphi}}=\left\{ A \in \mathbb{K}(\mathcal{H}) :  \mu(n,A)^q=O(\varphi(n))\right\}.$$
We note that the quasi-norms $\|A\|_{\Lw^{(q)}}:=\sqrt[q]{\||A|^q\|_{\Lw}}$ admit a \emph{quasi-H\"older inequality} 
\begin{equation}
\label{holdineq}
\|A_1A_2\|_{\Lw^{(q)}}\leq C\|A_1\|_{\Lw^{(q_0)}}\|A_2\|_{\Lw^{(q_1)}},\quad\mbox{for}\quad q^{-1}=q_0^{-1}+q_1^{-1}.
\end{equation}
As such, we have $\Lw^{(q_0)}\Lw^{(q_1)}\subseteq \Lw^{(q)}$ if $q^{-1}=q_0^{-1}+q_1^{-1}$. We abuse the notation by setting $\Lw^{(\infty)}$ to be the bounded operators.

\begin{rem}
We will not use the convexification $\mM_\Phi^{(q)}$ in this paper but we nevertheless warn the reader of a substantial amount of confusion in the literature regarding the distinction between  $\mM_\Phi^{(q)}$ and $\Lw^{(q)}$. This mainly concerns the special case $\varphi(t)=\frac{1}{\e+t}$, in this case, one writes $\mM_\Phi^{(q)}=\mM_{q,\infty}$ and $\Lw^{(q)}=\mathcal{L}_{q,\infty}$. It is not uncommon that $\mM_{1,\infty}$ is denoted by $\mathcal{L}_{1,\infty}$. Sometimes, untrue claims such as  $\mM_{q,\infty}=\mathcal{L}_{q,\infty}$ can even be found.
\end{rem}

\subsection{Examples of operators in weak ideals and Lorentz ideals}
\label{subsec:ex}

Before proceeding, we give some naturally occurring examples of operators that fall into $\Lw$ and $\mathcal M_{\Phi}$.

\begin{example}
\label{fromSimon}
It was shown in~\cite[Theorem 1.5]{Simon_NonCl} that if $A$ is the Dirichlet Laplacian for the region $D:=\{(x,y)\in \R^2: |xy|\le1\}$, then
$$\lim_{t\to\infty} \frac{N_A(t)}{t \log t}=\frac1\pi,$$
where $N_A(t)$ is the number of eigenvalues of $A$ less than  $t$. Since $n_{(1+A)^{-1}}(s)\sim N_A(1/s),$ it follows that
$$\lim_{s\to 0} \frac{n_{(1+A)^{-1}}(s)}{\frac1s \log \frac1s}=\frac1\pi.$$
Since the asymptotic inverse of the function $s\mapsto \frac1s \log \frac1s$ is $\varphi(t) = \frac{\log (\e+t)}{\e+t}$, it follows that
$$\lim_{n\to\infty} \frac{\mu(n,(1+A)^{-1})}{\varphi(n)}=\frac1\pi.$$ 
Hence, the operator $(1+A)^{-1} \in \Lw(L_2(D))$.

A related example in~\cite{Simon_NonCl} is the self-adjoint operator $\Delta_a=-\Delta+|xy|^a$ on $L_2(\R^2)$. Here $a>0$ is a parameter. It was shown in~\cite[Theorem 1.4]{Simon_NonCl} that 
$$\lim_{t\to\infty} \frac{N_{\Delta_a}(t)}{t^{1+a^{-1}} \log t}=\frac1\pi.$$
By a similar argument as above, and using that 
\begin{equation}
\label{varphiaex}
\varphi_a(t)=\left(\frac{\log(\e+t)}{\e+t}\right)^{\frac{a}{a+1}}, 
\end{equation} 
is an asymptotic inverse to $s\mapsto s^{-1-a^{-1}} \log \frac1s$, we have that $(1+ \Delta_a)^{-1} \in \mathcal{L}_{\varphi_a}(L_2(\R^2))$ and 
$$\lim_{n\to\infty} \frac{\mu(n,(1+\Delta_a)^{-1})}{\varphi_a(n)}=\frac1\pi.$$ 
\end{example}

\begin{example}
Let $\varphi$ be a function satisfying \eqref{quasi} and $(M,g)$ a closed Riemannian manifold. The Weyl law guarantees the eigenvalue asymptotics $\lambda_k((1+\Delta_g)^{d/2})=c_gk+o(k)$ for a suitable constant $c_g>0$. Condition \eqref{quasi} implies that $\varphi((1-\Delta_g)^{d/2})\in \Lw$.
\end{example}

\begin{example}
\label{fromBogdan}
If $(M,g)$ is a Riemannian manifold of dimension $d$, and $\rho_g$ denotes the geodesic distance, we can define an operator $T:C^\infty_c(M)\to C^\infty(M)$ by
\begin{equation}
\label{loglap}
Tf(x):=\int_M \frac{f(x)-f(y)}{\rho_g(x,y)^d}\ \mathrm{d}V_g,
\end{equation}
where $\mathrm{d}V_g$ denotes the volume form associated with $g$. It can be verified that $T$ is a pseudo-differential operator of H\"ormander type $(1,0)$ and any order $>0$. A short computation with the Fourier transform shows that the symbol of $T$ is proportional to $\log|\xi|_g$ modulo symbols of order $0$. As such, whenever $M$ is closed there is a non-zero constant $c$ such that $T-c\log((1-\Delta_g)^{d/2})$ is bounded in the $L_2$-norm, where $\Delta_g$ denotes the (negative) scalar Laplacian on $M$. We conclude that $T$ is essentially self-adjoint on $L_2(M)$ and the Weyl law on the closed manifold $M$ guarantees that $(i\pm \overline{T})^{-1}\in \Lw(L_2(M))$ for $\varphi(t)=\frac{1}{\log(\e+t)}$ being as in Example \ref{oneonlog}. 

Operators of the type in Equation \eqref{loglap} have as of lately made its appearances in noncommutative geometry. They can be defined on more general metric measure spaces (see \cite{GM16}). Such operators are geometric yet display better compatibility properties with respect to (semi-) group actions in examples (see \cite{DGMW16,GM16,MS16}). 
\end{example}

\begin{example}
\label{fromGS}
In~\cite[Example 4.9]{GS} the following example was considered. For a certain positive operator $P$ on the Podl\`es sphere $S^2_q$ -- a $q$-analogue of the Laplacian -- and the flat (positive) Laplacian $\Delta_{\mathbb T^2}$ on the 2-torus, the operator $A:=(1+P\otimes 1 + 1\otimes\Delta_{\mathbb T^2})^{-1}\in \mathcal M_{\Phi}(L_2(S^2_q\times \mathbb{T}^2))$, with 
\begin{equation}
\label{phifrompodles}
\Phi(t)=\log^3(\e+t^{1/3}).
\end{equation}
We show that, in fact, $A$ belongs to the corresponding weak ideal. It was shown in~\cite[Example 4.9]{GS} that
$${\rm Tr} (e^{-tA^{-1}}) \sim c\cdot \frac{\log^2t}{t}, \ t\to0.$$
Writing the heat trace as the Laplace transform
$$ {\rm Tr} (e^{-tA^{-1}}) = \int_0^\infty e^{-tz} d N_{A^{-1}}(z)$$
(where $N_{A^{-1}}(t)$ is the number of eigenvalues of $A^{-1}$ less than  $t$)
and using the classical Karamata Theorem (see e.g.~\cite[Chapter IV, Theorem 8.1]{Korevaar}), we obtain 
$$N_{A^{-1}}(t) \sim c\cdot t \log^2t, \ t\to\infty.$$
Since $n_{A}(s)\sim N_{A^{-1}}(1/s),$ and the asymptotic inverse of the function $s\mapsto \frac1s \log^2 \frac1s$ is $\varphi(t) = \frac{\log^2 (\e+t)}{\e+t}$, it follows that $\mu(n, A) \sim c\cdot \frac{\log^2 (\e+t)}{\e+t}$ as $t\to\infty.$ Thus, $A\in \Lw(L_2(S^2_q\times \mathbb{T}^2)).$
\end{example}

\begin{example}
Consider the strictly positive operator 
$$G:= (1+|x|^2)^{-d/4}(1-\Delta)^{-d/2}(1+|x|^2)^{-d/4},$$
on $L_2(\R^d)$. Let $\|\cdot\|_{\mathcal{L}_p}$ denote the $p$:th Schatten class norm. By using the Hausdorff-Young inequalities (see \cite[Theorem 4.1]{Simon}) one shows that for a constant $C>0$, 
$$\|G\|_{\mathcal{L}_p(L_2(\R^d))}\leq C(p-1)^{-2} \le \Phi(e^{(p-1)^{-1}}), \quad p>1,$$
where $\Phi(t)=\log^2(\e+t)$.
It follows from~\cite[Proposition 2.13 and formula (28)]{GS}, that $G\in \mathcal{M}_\Phi$ . Alternatively, one applies \cite[Proposition 2.15]{GS} directly to see that $G\in \mathcal{M}_\Phi$. This result can be made more precise using the results of \cite{bacorsg}. It follows from \cite[theorem 4.5]{bacorsg} that $\frac{N_{G^{-1}}(t)}{t\log(t)}$ has a limit $c$ as $t\to \infty$. Using the method from Example \ref{fromSimon}, we see that with $\varphi(t) = \frac{\log (\e+t)}{\e+t}$, it holds that
$$\lim_{n\to\infty} \frac{\mu(n,G)}{\varphi(n)}=c,\quad\mbox{and}\quad G\in \Lw(L_2(\R^d)).$$

The operator $G$ is an elliptic operator of order $(-d,-d)$ in the $SG$-calculus $SG^{*,*}(\R^d)$. This pseudo-differential calculus is also known as the scattering calculus \cite[Chapter 6]{MelScatt}, see also \cite{bacorsg}. Similar results holds also for more general SG-manifolds. Using \cite[Theorem 2.9]{bacorsg}, we see that $G^{p/d}\in SG^{-p,-p}(\R^d)$ for any $p>0$. Combining this fact with that $SG^{0,0}(\R^d)$ acts as bounded operators on $L_2(\R^d)$ and $G\in \Lw$, we conclude the following inclusion:
$$SG^{-p,-p}(\R^d)\subseteq \Lw^{(d/p)}(L_2(\R^d)), \quad p>0.$$
\end{example} 

\begin{example}
Other examples come from singular manifolds. The following example is based on the spectral properties of operators studied \cite{HartLeschVert}. The Laplacian on a manifold with cuspidal singularities in the metric decomposes into an operator with discrete spectrum, and one with continuous spectrum. The discrete part of the operator in turn decomposes as a countable direct sum (over different $\mu$ and $V$) of operators of the following form
$$H=-\frac{\mathrm{d}}{\mathrm{d} x}(x^2 \frac{\mathrm{d}}{\mathrm{d} x})+x^2\mu^2-\frac{1}{4}+V(x), \quad x>a.$$
Here $\mu$ is a spectral parameter and $V\in x^\gamma L^1[a,\infty)$ (where $\gamma<2$) is a potential. We equipp the operator $H$ with self-adjoint Robin type boundary conditions at $x=a$. See more in \cite[Introduction]{HartLeschVert}. It follows from \cite[Section 1.1.3]{HartLeschVert} that $\lim_{t\to\infty} \frac{N_{H}(t)}{t^{1/2} \log t}=\frac{1}{2\pi}$. By an argument similar to that in Example~\ref{fromSimon} we obtain $(i\pm H)^{-1}\in \Lw(L_2([a,\infty)))$, where $\varphi(t)=\left(\frac{\log(\e+t)}{\e+t}\right)^2$. 
\end{example}

\subsection{Dixmier traces}

A positive operator $A\in \Lw$ will have eigenvalues $\lambda_n(A)=O(\varphi(n))$. A standard problem in spectral theory is to determine the spectral asymptotics of $A$, i.e. if there is a $c$ such that $\lambda_n(A)=c\varphi(n)+o(\varphi(n))$ and to determine $c$. The computation of $c$ can be done using Dixmier traces -- a gadget detecting averaged spectral asymptotics. Their natural domain of definition is the larger Lorentz ideal. We refer the reader to the general definition of Dixmier traces due to J.~Dixmier \cite{D} (see also~\cite{C_book, LSZ}); we only consider the construction under assumptions relevant later in the paper. 

A functional $\omega\in \ell_\infty(\N)^*$ is called an \emph{extended limit} if it is a state vanishing on $c_0(\N)$. Equivalently, $\omega$ is a positive unital extension of the limit functional from the closed subspace of convergent sequence to $\ell_\infty(\N)$. We say that an extended limit $\omega$ is \emph{dilation invariant} if 
$$\omega(c_1,c_2,c_3,\ldots)=\omega((c_n)_{n\in \N})=\omega((c_{\lfloor n/2 \rfloor})_{n\in \N})=\omega(c_1,c_1,c_2,c_2,c_3,c_3,\ldots),$$ 
for any $(c_n)_{n\in \N}\in \ell_\infty(\N)$.

\begin{dfn}
\label{DT}
Let $\Phi$ be a concave function on $(0,\infty)$ such that $\lim_{t\to\infty} \Phi(t)=+\infty$ and
\begin{equation}
\label{c1}
\lim_{t\to \infty} \frac{\Phi(2t)}{\Phi(t)}=1.
\end{equation}
For every dilation invariant extended limit $\omega$ on the sequence space $\ell_\infty(\N)$, the extension by linearity of the functional
\begin{equation}
\label{defofdix}
{\rm Tr}_\omega(A) = \omega\left(\frac1{\Phi(n+1)} \sum_{k=0}^n \mu(k,A)\right), \ 0\le A \in \mM_\Phi,
\end{equation}
to $\mM_\Phi$ is called a \emph{Dixmier trace}.

A Dixmier trace on $\Lw$ is a restriction of a Dixmier trace (on the corresponding $\mM_\Phi$) to $\Lw$. In other words, a Dixmier trace on $\Lw$ is a functional given by the same formula as in Equation \eqref{defofdix} for positive operators.
\end{dfn}

\subsection{Functions of regular variation}
Now we introduce several classes of functions and describe their properties.

\begin{dfn}[see e.g.~\cite{RegVar}]\label{SV}
A positive measurable function $f$ defined on some interval $(a,\infty)$ is said to be 
\begin{itemize}
\item[(i)] \emph{regularly varying of index $\rho$} if 
\begin{equation*}
\lim_{t\to \infty} \frac{f(\lambda t)}{f(t)}=\lambda^\rho, \ \forall \ \lambda>0;
\end{equation*}
\item[(ii)] \emph{smoothly regularly varying of index $\rho$} if $f\in C^\infty$ and for any $n\in \N$,
\begin{equation}
\label{liminitsrrho}
\lim_{t\to \infty} \frac{t^n f^{(n)}(t)}{f(t)}=\rho(\rho-1)...(\rho-n+1).
\end{equation}
\end{itemize}
The classes of \emph{regularly varying} and \emph{smoothly regularly varying} functions are denoted by $R_\rho$ and $SR_\rho$, respectively. The class $R_0$ is called the class of slowly varying functions.
\end{dfn}

\begin{prop}\label{P1}
Let $f$ be a positive measurable monotone function defined on some interval $(a,\infty)$. One has $f\in R_\rho$ if and only if 
\begin{equation}
\label{eq20}
\lim_{t\to \infty} \frac{f(\lambda_0 t)}{f(t)}=\lambda_0^\rho
\end{equation}
for some $\lambda_0>0.$
\end{prop}

\begin{proof}
It is clear that~\eqref{eq20} holds for $f\in R_\rho$, and it only remains to prove the converse. Without loss of generality suppose that~\eqref{eq20} holds for $\lambda_0=2$. 
We first consider the case $\rho=0$. For every $n\in \Z$, $n\neq0$ we have
$$\lim_{t\to \infty} \frac{f(2^n t)}{f(t)}=1 .$$

Since $f$ is monotone, it follows that for every $\lambda>0$ and some $n,m\in\Z$ we have
$$\frac{f(2^n t)}{f(t)} \le \frac{f(\lambda t)}{f(t)} \le \frac{f(2^m t)}{f(t)} $$
which implies the statement.

For a general $\rho$ and $f$ satisfying~\eqref{eq20}, then $g(t)=t^{-\rho}f(t)$ satisfies~\eqref{eq20} with $\rho=0$. Hence,
$$\lim_{t\to \infty} \frac{f(\lambda t)}{f(t)}= \lambda^{\rho}\lim_{t\to \infty} \frac{g(\lambda t)}{g(t)}=\lambda^{\rho},$$
by the first part of the proof.
\end{proof}

\begin{thm}\cite[Theorem 1.8.2]{RegVar}
\label{SRV}
If $f\in R_\rho$, then there exists $g\in SR_\rho$ such that $f\sim g$ as $t\to\infty.$
\end{thm}

\begin{rem} 
For the existence of Dixmier traces on $\mM_\Phi$ (and $\Lw$, too) we only need that the decreasing function $\varphi : [0,\infty) \to (0,\infty)$ satisfies that
\begin{equation}
\label{limitinpf}
\lim_{t\to \infty} \frac{\varphi(2t)}{\varphi(t)}=\frac12.
\end{equation}
 Indeed, this implies~\eqref{c1}. By Proposition~\ref{P1}, Equation \eqref{limitinpf} is equivalent to $\varphi\in R_{-1}.$ Whereas Theorem~\ref{SRV} tells that assuming $\varphi\in SR_{-1}$ alter neither the class of ideals $\Lw$ nor Dixmier traces under consideration.
 
Throughout most of the paper we assume that $\varphi\in R_{-1}$. Assuming that $\varphi$ has regular variation is crucial in several proofs. That we assume the regular variation to be of index $-1$ is only to ensure the existence of Dixmier trace.
\end{rem}

\begin{rem}
In Section~\ref{psipseudos}, we use the smooth regular variation of $\varphi$ to construct a well defined pseudo-differential calculus. In fact, the minimal assumption that can be used in Section~\ref{psipseudos} is that we in Equation \eqref{liminitsrrho} have uniform bounds rather than a limit. 
\end{rem}

\begin{example}
The functions $\varphi(t)=\frac{\log^k(\e+t)}{\e+t}$, $k\in \Z$, from Examples \ref{standardex}, \ref{logkont} and  \ref{oneontlog} have smooth regular variation of index $-1$, and so does the function from Example \ref{exp}. The function $\varphi(t)=\frac{1}{\log(\e+t)}$ (from Example \ref{oneonlog}) is slowly varying, i.e. it has regular variation of index $0$. The function $\varphi_a$ from Equation \eqref{varphiaex} or $\Phi'$ with $\Phi$ as in Example \ref{fromGS} have smooth regular variation of index $\neq -1$. However, suitable powers of these functions will however asymptotically behave like $\log^a(t)/t$ and have smooth regular variation of index $-1$.
\end{example}

The following result shows that on weak ideals $\Lw$ \textit{all} Dixmier traces can be constructed without the additional assumption that the extended limit $\omega$ is dilation invariant. It is a corollary of~\cite[Theorem 17]{Sed_Suk}.

\begin{thm}
If $\varphi\in R_{-1}$ is decreasing, then
for every extended limit $\omega$ on $\ell_\infty(\N)$, the functional
\begin{equation*}
{\rm Tr}_\omega(A) = \omega\left(\frac1{\Phi(n+1)} \sum_{k=0}^n \mu(k,A)\right), \ 0\le A \in \Lw,
\end{equation*}
extends by linearity to a Dixmier trace $\Lw$.
\end{thm}

\begin{prop}
\label{schbelong}
If $\varphi\in R_{-1}$ is decreasing, then for any $p>1$, 
$$\Lw\subseteq \mathcal{L}_p.$$
Here $\mathcal{L}_p$ denotes the ideal of $p$-th Schatten class operators.
\end{prop}

\begin{proof}
Using \cite[Lemma 2.2]{GS} the proposition follows if $\varphi\in L_p(\R_+)$ for any $p>1$. It follows from $\varphi\in R_{-1}$ that for any $\epsilon>0$, there is a $t_0$ such that $\varphi(2t)\leq \varphi(t)(2-\epsilon)^{-1}$ for $t>t_0$. We conclude that for any $\epsilon$, there is a $C_\epsilon>0$ such that $\varphi(t)\leq C_\epsilon(2-\epsilon)^{-k}$ for $t\in [2^k,2^{k-1})$. For $p>1$, take $\epsilon<2-2^{1/p}$ and estimate 
\begin{align*}
\int_0^\infty |\varphi|^p(t)\mathrm{d}t=\int_0^1 |\varphi|^p(t)\mathrm{d}t&+\sum_{k=0}^\infty \int_{2^k}^{2^{k+1}} |\varphi|^p(t)\mathrm{d}t\leq\\
&\leq  \int_0^1 |\varphi|^p(t)\mathrm{d}t+ C_\epsilon\sum_{k=0}^\infty\left( \frac{2}{(2-\epsilon)^p}\right)^k<\infty.
\end{align*}
\end{proof}

Throughout the paper, we will make use of various averaging properties on the function $\varphi$.

\begin{prop}
\label{regvarresult} 
If $\varphi\in R_{-1}$, then

(i) for every $\alpha\ge \beta-1$ we have
$$\lim_{t\to \infty} \frac{t^{\alpha+1} \varphi^\beta(t)}{\int_0^t s^\alpha \varphi^\beta(s) \mathrm{d}s}=\alpha-\beta+1 .$$

(ii) for every $\alpha < \beta-1$ we have
$$\lim_{t\to \infty} \frac{t^{\alpha+1} \varphi^\beta(t)}{\int_t^\infty s^\alpha \varphi^\beta(s) \mathrm{d}s}=-\alpha+\beta-1 .$$
\end{prop}

\begin{proof}
It follows from $\varphi\in R_{-1}$ that 
$$\lim_{t\to \infty} \frac{\varphi^\beta(2t)}{\varphi^\beta(t)}=2^{-\beta}.$$
This means that the function $\varphi^\beta$ varies regularly with index $-\beta$ (see~\cite[first definition in Section 1.4.1]{RegVar}).
The assertions follow from~\cite[Theorem 1.5.11]{RegVar}. 
\end{proof}

We shall prove the discrete counterpart of the preceding result. First, we recall the general result from~\cite{Vuilleumier}.

\begin{thm}\label{Vuilleumier}
Let a matrix $A=\{a_{nk}\}$ be such that for some $\eta>0$ the following hold:
$$\sum_{k=n}^\infty |a_{nk}| k^\eta = O(n^\eta), \ n\to \infty; \ \ \sum_{k=1}^n |a_{nk}| k^{-\eta} = O(n^{-\eta}), \ n\to \infty.$$ Let $L$ be a slowly varying sequence, that is 
$\lim_{n\to\infty} \frac{L_{\lfloor \lambda n\rfloor}}{L_{n}}=1.$
If $$\sum_{k=1}^\infty a_{nk} \to A, \ n\to \infty,$$
then 
$$\frac{\sum_{k=1}^\infty a_{nk} L_k}{L_n} \to A, \ n\to \infty.$$
\end{thm}

We shall need the following particular forms of Theorem~\ref{Vuilleumier}.

\begin{lem}\label{regvarresultdiscrete}
If $\varphi\in R_{-1}$, then

(i) for every $\alpha> \beta$ 
we have
$$\lim_{n\to \infty} \frac{\sum_{k=1}^n 2^{k\alpha} \varphi^\beta(2^k)}{2^{n\alpha}\varphi^\beta(2^n)} =\frac{2^{\alpha-\beta}}{2^{\alpha-\beta}-1} .$$

(ii) for every $\alpha< \beta$ 
we have
$$\lim_{n\to \infty} \frac{\sum_{k=n}^\infty 2^{k\alpha} \varphi^\beta(2^k)}{2^{n\alpha}\varphi^\beta(2^n)} =\frac1{1-2^{\alpha-\beta}} .$$
\end{lem}

\begin{proof}

Set $L_k := [k\varphi(k)]^\beta$. The fact that $\varphi\in R_{-1}$ implies that the sequence $L$ is slowly varying.

(i) Set $$a_{nk} = \begin{cases}
\frac{k^{\alpha-\beta}}{n^{\alpha-\beta}}, \text{for} \ k=2^i, \ 1\le k\le n,\\
0, \text{otherwise}.
\end{cases}$$

We have
$$\sum_{k=1}^\infty a_{nk} = n^{-\alpha+\beta}\sum_{i=1}^{\lfloor \log_2 n \rfloor} 2^{i(\alpha-\beta)} \to \frac{2^{\alpha-\beta}}{2^{\alpha-\beta}-1}, n \to \infty.$$
The other conditions of Theorem~\ref{Vuilleumier} are verified in a similar way.

It follows from Theorem~\ref{Vuilleumier} that, in particular,
$$\frac{\sum_{k=1}^\infty a_{2^n,k} L_k}{L_{2^n}} =\frac{\sum_{k=1}^n 2^{k\alpha} \varphi^\beta(2^k)}{2^{n\alpha}\varphi^\beta(2^n)}\to \frac{2^{\alpha-\beta}}{2^{\alpha-\beta}-1}, \ n\to \infty.$$

(ii) The proof is similar to that of (ii) upon setting
$$a_{nk} = \begin{cases}
\frac{k^{\alpha-\beta}}{n^{\alpha-\beta}}, \text{for} \ k=2^i, \ k\ge n,\\
0, \text{otherwise}.
\end{cases}$$
\end{proof}

\section{Weakly modulated operators and formulas for their Dixmier traces}
\label{subsec:mod}

In this section we give a quite general formula for Dixmier traces of an operator with additional regularity properties relative to an auxiliary operator $V$, see Theorem \ref{weaklymodform}. This formula is in general unwieldely, but can, under some additional assumptions, for operators on manifolds produce closed expressions in terms of the operator's $L_2$-symbol.

To compute Dixmier traces of an operator $G\in \Lw$, we make use of an auxiliary operator $V$. We say that $V$ is \emph{strictly positive} if $V$ is positive and $V\mathcal{H}\subseteq \mathcal{H}$ is dense. For a strictly positive operator $V$ and $p>0$, $V^{-1/p}$ can be defined as a densely defined unbounded operator with domain $V^{1/p}\mathcal{H}$. The operator $V^{-1/p}$ is self-adjoint because it is symmetric and $(i\pm V^{-1/p})^{-1}=V^{1/p}(iV^{1/p}\pm 1)^{-1}$ exists. Recall our convention that $\Lw^{(\infty)}$ denotes the bounded operators. The following definition extends the notion of weakly modulated operators~\cite[Definition 2.12]{GG} from $\mL_{1,\infty}$ to more general ideals $\Lw$.

\begin{dfn}
\label{weaklymod}
Let $V \in \mathbb{B}(\mathcal{H})$ be strictly positive. We say that $G\in \mathbb{B}(\mathcal{H})$ is \emph{weakly $\varphi$-modulated} with respect to $V$ if there is a $p\geq 1$ such that the densely defined operator $GV^{-1/p}$ extends to a bounded operator with $GV^{-1/p}\in \Lw^{(\frac{p}{p-1})}(\mathcal{H})$.
\end{dfn}

We will later introduce further variations on being $\varphi$-modulated.

\begin{rem}
It follows from the definition that the set of operators being weakly $\varphi$-modulated with respect to $V$ is closed under left multiplication in $\mathbb{B}(\mathcal{H})$. It is unclear to the authors if the set of weakly $\varphi$-modulated operators (with respect to a fixed $V$) forms a vector space, except for when fixing $p$.
\end{rem}

For a strictly positive $V\in \mathbb{B}(\mathcal{H})$ and $s\in \R$, we define the Hilbert space 
\begin{align*}
\mathcal{H}^s_V:=&V^s\mathcal{H}, \quad\mbox{with the inner product}\quad \langle f_1,f_2\rangle_{\mathcal{H}^s_V}:=\langle V^{-s}f_1,V^{-s}f_2\rangle_{\mathcal{H}}.
\end{align*}

\begin{prop}
\label{weakregandsob}
Let $V \in \mathbb{B}(\mathcal{H})$ be strictly positive. The operator $G\in \mathbb{B}(\mathcal{H})$ is weakly $\varphi$-modulated with respect to $V$ if and only if there is an $s\in (0,1]$ such that the densely defined operator $G:\mathcal{H}^{-s}_V\dashrightarrow \mathcal{H}$ extends to a bounded operator with $G\in \Lw^{(\frac{1}{1-s})}(\mathcal{H}^{-s}_V,\mathcal{H})$.
\end{prop}

Proposition \ref{weakregandsob} should be compared to \cite[Lemma 2.21]{GG}.

\begin{proof}
By construction, $G\in \Lw^{(\frac{1}{1-s})}(\mathcal{H}^{-s}_V,\mathcal{H})$ if and only if $GV^{-s}\in \Lw^{(\frac{1}{1-s})}(\mathcal{H})$. The proposition follows immediately from setting $p=1/s$.
\end{proof}

The following result relates the eigenvalues and expectation values of a weakly modulated operator. The proof follows that of~\cite[Lemma 11.2.10]{LSZ} which concerns the case $\varphi(t)=\frac{1}{\e+t}$, when $\Lw=\mathcal{L}_{1,\infty}$. We sketch the proof for the convenience of the reader, focusing mainly on the differences to $\frac{1}{\e+t}$.

\begin{lem}
\label{l_exp}
Assume that $\varphi\in R_{-1}$ is decreasing. Let $V \in \Lw$ be strictly positive and let $\{e_n\}_{n\in \N}$ be an eigenbasis for $V$ ordered so that $Ve_n = \mu(n,V)e_n$, $n\ge0$. If $G\in  \mathbb{B}(\mathcal{H})$ is weakly $\varphi$-modulated with respect to $V$, then
$G \in \Lw$ and
\begin{equation}
\label{lemma11210}
\sum_{k=0}^n \lambda(k,\Re G) - \sum_{k=0}^n \langle(\Re G)e_k, e_k\rangle = o(\Phi(n)), \ n \to \infty.
\end{equation}
Here $\Re G= \frac{G^*+G}2$ denotes the real part of $G$.
\end{lem}

\begin{proof}
Using standard properties of singular values, we obtain
$$\mu(2n, G)= \mu(2n, GV^{-1/p}V^{1/p})\le \mu(n, GV^{-1/p}) \mu(n, V^{1/p}).$$
Since $GV^{-1/p}\in \Lw^{(\frac{p}{p-1})}$ and $V \in \Lw$, it follows that
$$\mu(2n, G)= O(\varphi^\frac{p-1}p (n)) O(\varphi^\frac1p (n))=O(\varphi(n)), \ n\ge0.$$
The property $\mu(2n, G)=O(\varphi(n))$ and $\varphi\in R_{-1}$ imply that $\mu(2n, G)= O(\varphi(2n)), \ n\ge0.$ We conclude that $G \in \Lw$.

The remainder of the proof concerns the property in Equation \eqref{lemma11210}. Let $f_n$ be a basis such that $(\Re G)f_n=  \lambda(n, \Re G)f_n$. Let $p_n$ (resp., $q_n$) be the projection on the linear span of $e_k$ (resp., $f_k$), $0\le k\le n$. Let $r_n= p_n \vee q_n$.

Following the proof of~\cite[Lemma 11.2.10]{LSZ} we obtain
$$|{\rm Tr}((\Re G)(r_n-q_n))| \le 2(n+1) \mu(n, \Re G)$$
and
$$|{\rm Tr}(G^*(r_n-p_n))| \le  \mu(n, V)^{1/p} \sum_{k=0}^{2n+1} \mu(k,GV^{-1/p}). $$

Since $G \in \Lw$ (and so $\Re G \in \Lw$), it follows from $\varphi\in R_{-1}$ that
$$|{\rm Tr}((\Re G)(r_n-q_n))| = O((n+1) \varphi(n)) = o(\Phi(n)),  \ n \to \infty,$$
by Proposition~\ref{regvarresult}(i).

Also, since $V \in \Lw$ and $GV^{-1/p}\in \Lw^{(\frac{p}{p-1})}$, it follows that
$$|{\rm Tr}(G^*(r_n-p_n))| =  O(\varphi^\frac1p (n)) O\left(\sum_{k=0}^{2n+1} \varphi^\frac{p-1}p (k)\right).$$

While $\varphi\in R_{-1}$, the sequence $\{\varphi^\frac{p-1}p(k)\}_{k=0}^\infty$ varies regularly with index $\frac{1-p}p$. It follows from~\cite[Theorem 6]{Bojanic_Seneta} that
$$\lim_{n\to \infty} \frac{\sum_{k=0}^n \varphi^\frac{p-1}p(k)}{n\varphi^\frac{p-1}p(n)} =\frac{1}{1-\frac{p-1}p}=p .$$

Hence,
$$|{\rm Tr}(G^*(r_n-p_n))| =  O\left(\varphi^\frac1p (n)(2n+1) \varphi^\frac{p-1}p (2n+1)\right) = O(n \varphi(n))= o(\Phi(n)),  \ n \to \infty, $$
due to $\varphi\in R_{-1}$ and Proposition~\ref{regvarresult}(i).

Combining the above estimates with the following estimate ensures the lemma.
\begin{align*}
&\left|\sum_{k=0}^n \lambda(k,\Re G) - \sum_{k=0}^n \langle(\Re G)e_k, e_k\rangle\right|=|{\rm Tr}(\mathrm{Re}(G)(p_n-q_n))|\\
&\qquad\qquad\leq |{\rm Tr}(\mathrm{Re}(G)(r_n-q_n))|+\frac{1}{2}|{\rm Tr}(G(r_n-p_n))|+\frac{1}{2}|{\rm Tr}(\mathrm{Re}(G)(r_n-p_n))|.
\end{align*}
\end{proof}

The following result provides a formula to compute Dixmier traces in terms of expectation values.

\begin{thm}
\label{weaklymodform}
Assume that $\varphi\in R_{-1}$ is decreasing. Let $V \in \Lw$ be strictly positive and let $(e_n)_{n\in \N}$ be an eigenbasis for $V$ ordered so that $Ve_n = \mu(n,V)e_n$, $n\ge0$. If $G\in \mathbb{B}(\mathcal{H})$ is weakly $\varphi$-modulated with respect to $V$, then for every extended limit $\omega$ on $\ell_\infty$ we have
$${\rm Tr}_\omega(G) = \omega \left(\frac1{\Phi(n+1)} \sum_{k=0}^n \langle Ge_k, e_k\rangle\right).$$
\end{thm}

The proof of Theorem \ref{weaklymodform} follows from linearity of Dixmier traces and Lemma \ref{l_exp}; it is the same as that of~\cite[Theorem 2.18]{GG} and is therefore omitted. 

\begin{rem}
The reader should beware of the fact that for any $G\in \Lw$ there is a $V\in \Lw$ with ordered eigenbasis $(e_n)_{n\in \N}$ such that for all extended limits $\omega$,
$$\omega \left(\frac1{\Phi(n+1)} \sum_{k=0}^n \langle Ge_k, e_k\rangle\right)=0.$$
This statement depends only on the existence of the ON-basis, and the existence proof can be found in \cite[Corollary 7.5.3]{LSZ}. As such, some compatibility between $G$ and $V$ is in general required to have a formula as in Theorem \ref{weaklymodform}. In particular, this compatibility is guaranteed when $G$ is weakly $\varphi$-modulated with respect to $V$.
\end{rem}

The formula in Theorem \ref{weaklymodform} works in quite large generality, but is difficult to compute. We now turn to a stronger condition which guarantees improved formulas. 

\section{Strongly modulated operators}
\label{subsec:stronglymod}

The following definition extends the notion of modulated operators~\cite[Definition 11.2.1]{LSZ} from $\mL_{1,\infty}$ to more general ideals $\Lw$. The reader will note that we have added a number of new adjectives to the terminology of \cite{LSZ}, all in the purpose of clearer terminology in the more convoluted world of general weak ideals. To shorten notation, we let $\|\cdot\|_p$ denote the $p$:th Schatten class norm.

\begin{dfn}
\label{mod}
Let $V \in \mathbb{B}(\mathcal{H})$ be a positive operator and $\varphi\in R_{-1}$ be a decreasing function. An operator $G\in \mathbb{B}(\mathcal{H})$ is said to be \emph{strongly $\varphi$-modulated} with respect to $V$ if
$$\|G\|_{\varphi,V}:=\sup_{t>0} \frac{\| G(1+tV)^{-1}\|_{2}}{\sqrt{\varphi(t)}}  < \infty.$$
\end{dfn}

When $\varphi(t)=\frac1{t+1}$, and $\Lw=\mathcal{L}_{1,\infty}$, a strongly $\varphi$-modulated operator with respect to $V$ is called a $V$-modulated operator in the terminology of \cite{LSZ}, see~\cite[Definition 11.2.1]{LSZ}.

The $\varphi$-modulated operators with respect to $V$ can be characterized in terms of the spectral measure of $V$. We define the \emph{spectrally $\varphi$-modulated} norm with respect to $V$ as:
$$\|G\|_{\varphi,V, spec}:=\sup_{t>0} \frac{\|GE_V[0,t^{-1}]\|_2}{\sqrt{\varphi(t)}}.$$
If $G$ satisfies $\|G\|_{\varphi,V, spec}<\infty$, we say that $G$ is spectrally $\varphi$-modulated with respect to $V$. The proof of the following equivalence between $\|\cdot\|_{\varphi,V, spec}$ and $\|\cdot\|_{\varphi,V}$ is a modified version of~\cite[Lemma 11.2.5]{LSZ}.

\begin{lem}
\label{spec_char}
Let $V$ be a positive operator and $\varphi\in R_{-1}$ be a decreasing function. An operator $G \in \mathcal L_2$ is spectrally $\varphi$-modulated with respect to $V$ if and only if it is strongly $\varphi$-modulated with respect to $V$. In fact, there is a constant $C=C(V,\varphi)>0$ such that 
$$\frac{1}{C}\|G\|_{\varphi,V, spec}\leq \|G\|_{\varphi,V}\leq C\|G\|_{\varphi,V, spec}.$$
\end{lem}

\begin{proof} Let $G$ be spectrally $\varphi$-modulated with respect to $V$.
Without loss of generality we assume that $V\le 1$. Let $t\in [2^k, 2^{k+1})$ for some $k\ge 0$.
We estimate
\begin{align*}
\| G(1+tV)^{-1}\|_{2} &\le \|GE_V[0,2^{-k}]\|_2+\sum_{j=1}^{k-1} \|GE_V(2^{-j-1},2^{-j}](1+tV)^{-1}\|_{2}\\
&\le \|G\|_{\varphi,V, spec}\sqrt{\varphi(2^k)}+\sum_{j=1}^{k-1} (1+t2^{-j-1})^{-1}\|GE_V(2^{-j-1},2^{-j}]\|_{2}\\
&\le  \|G\|_{\varphi,V, spec}\left(\sqrt{\varphi(2^k)}+\sum_{j=1}^{k-1} 2^{-k+j+1} \sqrt{\varphi(2^j)}\right).
\end{align*}
Since the function $\varphi$ is bounded and belongs to $R_{-1}$, Lemma \ref{regvarresultdiscrete} implies that
\begin{align*}
\sum_{j=1}^{k-1} 2^{-k+j+1} \sqrt{\varphi(2^j)}&= 2^{-k+1} \sum_{j=1}^{k-1} 2^{j} \sqrt{\varphi(2^j)}\\
&= O(1) 2^{-k+1} 2^{k-1} \sqrt{\varphi(2^{k-1})} = O(\sqrt{\varphi(t)}), \ t\to \infty,
\end{align*}
since $\varphi\in R_{-1}$.

Summing up, for a suitable constant $C$, we can estimate
$$\| G(1+tV)^{-1}\|_{2} \leq C\|G\|_{\varphi,V, spec}\sqrt{\varphi(t)}, \ t>0.$$
That is, $G$ is strongly $\varphi$-modulated with respect to $V$ if $G$ is spectrally $\varphi$-modulated with respect to $V$.

The converse implication is a straightforward repetition of ~\cite[Lemma 11.2.5]{LSZ} and therefore omitted.
\end{proof}

The following lemma is an analogue of~\cite[Lemma 11.2.6]{LSZ}.

\begin{lem}
\label{VAmod}
Let $V\in \mathbb{B}(\mathcal{H})$ be a positive operator, $A\in\mathbb{B}(\mathcal{H})$ and let $\varphi\in R_{-1}$ be decreasing. If an operator $G \in \mathcal L_2$ is strongly $\varphi$-modulated with respect to $V$, then $GA$ is strongly $\varphi$-modulated with respect to $|VA|$.
\end{lem}

\begin{proof}
Without loss of generality we assume that $\|V\|_\infty\le 1$ and $\|A\|_\infty\le 1$. Set $p_k=E_V[0,2^{-k}]$ and $q_k=E_{|VA|}[0,2^{-k}]$, $k\ge 0$.
We have
$$\|(1-p_j)Aq_k\|_\infty \le 2^j \|VAq_k\|_\infty \le 2^{j-k}, \quad 
\|Gp_j\|_2 \le c\cdot \sqrt{\varphi(2^{j})},$$
by Lemma~\ref{spec_char}.

Thus,
\begin{align*}
\|GAq_k\|_2 &\le \|Gp_kAq_k\|_2+\sum_{j=1}^k \|G(p_{j-1}-p_j)Aq_k\|_2\\
&\le\|Gp_k\|_2+\sum_{j=1}^k \|Gp_{j-1}\|_2\cdot \|(1-p_j)Aq_k\|_\infty\\
&\le C \cdot \left( \sqrt{\varphi(2^{k})}+ \sum_{j=1}^k \sqrt{\varphi(2^{j-1})}\cdot 2^{j-k}\right)
\end{align*}

Using~\eqref{c1} and Lemma~\ref{regvarresultdiscrete} we obtain
$$\sum_{j=1}^k \sqrt{\varphi(2^{j-1})}\cdot 2^{j-k} \le C 2^{-k}\sum_{j=1}^k \sqrt{\varphi(2^{j})}\cdot 2^{j} = O(\sqrt{\varphi(2^{k})}), \ k\to \infty.$$

Thus,
$$\|GAq_k\|_2 \le C \sqrt{\varphi(2^{k})}, \ k\ge0.$$
By Lemma~\ref{spec_char} the operator $GA$ is strongly $\varphi$-modulated with respect to $|VA|$.
\end{proof}

\begin{rem}
It follows from the definition that the set of operators being strongly/spectrally $\varphi$-modulated with respect to $V$ forms a left ideal in $\mathbb{B}(\mathcal{H})$.
\end{rem}

Next we will use Lemma~\ref{spec_char} to show that every strongly $\varphi$-modulated operators is, in fact, weakly $\varphi$-modulated. To do so, we need an additional assumption.

\begin{dfn}
\label{someavass}
We say that the function $\varphi:[0,\infty)\to (0,\infty)$ satisfies \emph{property (W)} if there exist positive constants $C_1$ and $C_2$ such that
\begin{enumerate}
\item[(W1)]$\displaystyle C_1 t^2 \varphi(t)\le \varphi^{-1}\left(\frac{1}{t}\right), \ t\to \infty;$
\item[(W2)]$\displaystyle \varphi^{-1}\left(\frac{1}{t}\right) \le C_2 t^2 \varphi(t), \ t\to \infty.$
\end{enumerate}
\end{dfn}

We remark that property (W) is solely used for relating different notions of being modulated and to prove symbol properties of certain modulated operators.

\begin{rem}
Property (W) relates heat trace asymptotics and the Weyl law through the Karamata Theorem (see e.g. Example~\ref{fromGS} above and ~\cite[Chapter IV, Theorem 8.1]{Korevaar}).
\end{rem}

\begin{example}
The reader can easily verify that $\varphi(t)=\frac{1}{\e+t}$, satisfies property (W). With some more work, one shows the same for $\varphi(t)=\frac{\log^k(\e+t)}{\e+t}$, $k\in \Z$. The situation is more interesting for the function $\varphi=\Phi'$, $\Phi(t) = e^{\log^\beta t}$: for $0<\beta \le 1/2$ it satisfies (W1) but does not satisfy (W2); for $1/2<\beta<1$ it satisfies (W2) but does not satisfy (W1).
\end{example}

Before proving that strongly modulated operators are weakly modulated, we first need a lemma relating singular values to the spectral modulation norm.

\begin{lem}
\label{singvalueest}
Let $\varphi$ satisfy Condition \eqref{quasi}. For any $V\in \mathcal{L}_\varphi$, there is a constant $C=C_V$ such that for any $S\in \mathcal L_2$,
$$\mu(n,S)\leq n^{-1/2}\|SE_V[0,C\varphi(n)]\|_{\mathcal{L}_2},\quad\forall n\geq 1.$$
\end{lem}

\begin{proof}
We can assume that $\mu(k,V)\leq \varphi(k)$ for all $k$. In this case, $\mathrm{rk}(1-E_V[0,\varphi(n)])\leq n$. We have that 
\begin{align*} 
n\mu(2n,S)^2&\leq \sum_{k=n+1}^{2n}\mu(k,S)^2\leq \sum_{k=n+1}^{\infty}\mu(k,S)^2=\\
&=\inf\{\|S-A\|_{\mathcal{L}_2}^2: \mathrm{rk}(A)\leq n\} \leq \|SE_V[0,\varphi(n)]\|_{\mathcal{L}_2}^2.
\end{align*}
The lemma follows from this inequality and Equation \eqref{quasi}.
\end{proof}

The following result is an extension of~\cite[Lemma 11.2.9]{LSZ}.

\begin{lem}
\label{mod_w-mod}
Assume that $\varphi\in R_{-1}$ has property (W1) (see Definition \ref{someavass}). Let $V \in \Lw$ be strictly positive. Whenever $G$ is spectrally $\varphi$-modulated with respect to $V$, $G$ is weakly $\varphi$-modulated with respect to $V$. 
\end{lem}

Recall that when $\varphi\in R_{-1}$, it is equivalent for an operator to be strongly and spectrally $\varphi$-modulated (see Lemma \ref{spec_char}).

\begin{proof}
The proof follows that of~\cite[Lemma 11.2.9]{LSZ}. We only indicate the differences.

Take $p>2$ and assume that $G$ is spectrally $\varphi$-modulated with respect to $V$. For $2^k \le n < 2^{k+1}$, we estimate
\begin{align}
\nonumber
\|GV^{-1/p}E_V[0, 1/n]\|_2 &\le \sum_{j=k}^\infty 2^\frac{j+1}p\|GE_V(2^{-j-1}, 2^{-j}]\|_2\\
\nonumber
&\le O(1) \sum_{j=k}^\infty 2^\frac{j+1}p \sqrt{\varphi(2^j)} = O(1)\int_{2^k}^\infty z^{\frac1p-1} \sqrt{\varphi(z)} \, dz.
\end{align}

Using Lemma~\ref{regvarresultdiscrete} (ii) with $\alpha=1/p$ and $\beta=1/2$, we obtain
\begin{align}
\label{estimateinggvonep}
\|GV^{-1/p}E_V[0, 1/n]\|_2 = O\left((2^k)^{\frac1p}\sqrt{\varphi(2^{k})}\right)=O\left(n^{\frac1p}\sqrt{\varphi(n)}\right)
=O\left(n^{\frac1p-1}\sqrt{\varphi^{-1}(n^{-1})}\right),
\end{align}
where we in the last identity used property (W1). Combining Lemma \ref{singvalueest} with the estimate \eqref{estimateinggvonep} we obtain that there is a constant $C$ such that
$$ \mu(k, GV^{-1/p}) \le k^{-1/2}\|GV^{-1/p}E_V[0, C \varphi(k)]\|_2 = k^{-1/2}O\left((C\varphi(k))^{1-\frac1p}\sqrt{\varphi^{-1}(C \varphi(k))}\right).$$

Since $\varphi$ is varying regularly with index $-1$, \cite[Theorem 1.5.12]{RegVar} implies that $\varphi^{-1}$ varies regularly with index $-1$, too. Hence,
$$\lim_{s\to 0+} \frac{\varphi^{-1}(s)}{\varphi^{-1}(Cs)}=C.$$

Thus,
$$ \mu(k, GV^{-1/p}) = k^{-1/2}O\left(\varphi(k)^{1-\frac1p}\sqrt{\varphi^{-1}(\varphi(k))}\right)=O(\varphi(k)^{1-\frac1p}).$$

Therefore, $GV^{-1/p} \in \Lw^{\left(\frac{p}{p-1}\right)}$ and so $G$ is weakly $\varphi$-modulated with respect to $V$.
\end{proof}

\begin{prop}
Assume that $\varphi\in R_{-1}$ is decreasing and satisfies property (W2) (see Definition \ref{someavass}). Then any $V\in \Lw$ is spectrally $\varphi$-modulated with respect to itself. In particular, any $G\in \mathbb{B}(\mathcal{H})$ is spectrally $\varphi$-modulated with respect to $V$ whenever it is weakly $\varphi$-modulated with respect to $V$ for $p=1$ (i.e $GV^{-1}$ has a bounded extension).
\end{prop}

\begin{proof}
For $2^k \le t < 2^{k+1}$, we estimate
\begin{align*}
\|VE_V[0, 1/t]\|_2 &\le \sum_{j=k}^\infty 2^{-j}\|E_V(2^{-j-1}, 2^{-j}]\|_2
= \sum_{j=k}^\infty 2^{-j}\sqrt{{\rm Tr}(E_V(2^{-j-1}, 2^{-j}])}\\
&\le \sum_{j=k}^\infty 2^{-j}\sqrt{{\rm Tr}(E_V(2^{-j-1}, \infty])}= \sum_{j=k}^\infty 2^{-j} (n_V(2^{-j-1}))^{1/2}.
\end{align*}

Since $n_V(s) \le t$ if and only if $\mu(t,V)\le s$, it follows that $\mu(n,V)\le \varphi(n)$ implies $n_V(n) \le \varphi^{-1}(n)$. Hence,
$$\|VE_V[0, 1/t]\|_2 \le 2 \sum_{j=k+1}^\infty 2^{-j} \sqrt{\varphi^{-1}(2^{-j})}.$$

Since $\varphi$ varies regularly with index $-1$, \cite[Theorem 1.5.12]{RegVar} implies that $\varphi^{-1}$ varies regularly with index $-1$ at zero, that is
$$\lim_{t\to 0} \frac{\varphi^{-1}(t)}{\varphi^{-1}(2t)}=2.$$
Direct verification shows that the function $f(t)=\frac1{\varphi^{-1}(1/t)}$ belongs to $R_{-1}$.
Thus, Proposition~\ref{regvarresultdiscrete} yields
\begin{align*}
\sum_{j=k+1}^\infty 2^{-j} \sqrt{\varphi^{-1}(2^{-j})}= \sum_{j=k+1}^\infty \frac{2^{-j}}{ \sqrt{f(2^{j})}} = O\left(\frac{2^{-k-1}}{ \sqrt{f(2^{k+1})}}\right) = O\left(\frac{1}{t}\sqrt{\varphi^{-1}(1/t)}\right), \ t\to\infty.
\end{align*}

Using property (W2), we obtain $\|VE_V[0, 1/t]\|_2 = O(\sqrt{\varphi(t)}).$ Thus, $V$ is spectrally $\varphi$-modulated with respect to itself.
\end{proof}

\section{Symbol properties of modulated operators}
\label{subsec:symbprop}

To make the discussion on modulated operators more concrete, we focus on the case $\mathcal{H}=L_2(W)$ for a $d$-dimensional real inner product space $W$. The case of interest to us later, closed manifolds, will be localized to this situation. Associated with the inner product on $W$, there is a Laplace operator $\Delta$. We follow the analyst's sign convention making the Laplacian negative (in form sense), i.e. $\Delta=\sum_{j=1}^d\frac{\partial^2}{\partial x_j^2}$ for an orthonormal choice of linear coordinates $x_1,\ldots, x_d$.

\subsection{Laplacian modulated operators}

\begin{dfn}
An operator $G\in \mathbb{B}(L_2(W))$ is said to be weakly/strongly/spectrally \emph{$\varphi$-Laplacian modulated} if it is weakly/strongly/spectrally $\varphi$-modulated with respect to the operator $\varphi[(1-\Delta)^{d/2}]$, where $d$ is the dimension of $W$.
\end{dfn}

We remark that the operator $\varphi[(1-\Delta)^{d/2}]$ is not in $\Lw$, it is not even a compact operator on $L_2(W)$. The operator is however locally in $\Lw$ by the next proposition. Here we say that $G$ is \emph{locally in $\Lw$} if $\chi G, G\chi\in \Lw$ for all compactly supported $\chi$. The coming three results are strictly speaking not needed to develop the general theory, and Lemma \ref{psiideallem} is proved using machinery for $\varphi$-pseudo-differential operators developed below in Section \ref{psipseudos}, but we nevertheless include them for explanatory purposes.

Although we are interested in functions $\varphi$ which have smooth regular variation of index $-1$ (since this guarantees the existence of Dixmier traces on the corresponding ideals), results in this subsection hold for functions which have smooth regular variation of any index $\rho>0$.

\begin{lem}
\label{psiideallem}
Let $\varphi$ be a decreasing function with smooth regular variation. Then for any $\chi\in C^\infty_c(W)$, we have $\chi \varphi((1-\Delta)^{d/2})\in \Lw(L_2(W))$.
\end{lem}

We shall return to the proof of Lemma \ref{psiideallem} below in Corollary \ref{corofphidifflem}. It seems feasible that much weaker assumptions than $\varphi$ having smooth regular variation are needed from $\varphi$ for Lemma \ref{psiideallem} to hold.

\begin{prop}
\label{locallylwpsi}
Assume that $\varphi$ is decreasing and has smooth regular variation. Then, for any compactly supported $\chi\in L^\infty(W)$ and $p\geq 1$, the operator $\chi\varphi[(1-\Delta)^{d/2}]^{\frac1p}$ belongs to $\Lw^{(p)}(L_2(W))$. 
\end{prop}

\begin{proof}
The proof that $\chi\varphi[(1-\Delta)^{d/2}]\in \Lw(L_2(W))$, for $\chi\in C^\infty_c(W)$ and $\varphi$ having smooth regular variation, follows from Lemma \ref{psiideallem}. Since any compactly supported $\chi\in  L^\infty(W)$ admits a $\chi'\in C^\infty_c(W)$ with $\chi'\chi=\chi$, the statement that $\chi\varphi[(1-\Delta)^{d/2}]^{\frac1p}$ belongs to $\Lw^{(p)}(L_2(W))$ for all $\chi\in C^\infty_c(W)$ is equivalent to the same statement for compactly supported $L^\infty$-functions. This statement is in turn equivalent to $L^\infty(K)V^{1/p}\subseteq \Lw^{(p)}(L_2(W))$ for any compact $K$.

We set $V:=\varphi[(1-\Delta)^{d/2}]$. Define the Banach space $\mathfrak{X}:=\mathbb{B}(L_2(W))$. We consider the partially defined operator $T:\mathfrak{X}\dashrightarrow \mathfrak{X}$ given by $T(A):=AV^{-1}$ with $\mathrm{Dom}(T)=\mathfrak{X}V\subseteq \mathbb{B}(L_2(W))$. The operator $T$ is closed, and admits complex powers $T^\alpha(A):=AV^{-\alpha}$ for $\mathrm{Re}(\alpha)\in [0,1]$. By \cite[Theorem 3]{seeley71}, $\mathrm{Dom}(T^\alpha)=\mathfrak{X}V^\alpha=[\mathfrak{X},\mathrm{Dom}(T)]_\alpha$ when equipping $\mathrm{Dom}(T)$ with its graph norm. 

Fix a compact $K\subseteq W$. From the discussion in the second paragraph, we conclude that $[L^\infty(K),L^\infty(K)V]_\alpha=L^\infty(K)V^\alpha$, where we consider $L^\infty(K)V$ as a closed subset of $\mathrm{Dom}(T)$ in its graph norm. By the discussion in the first paragraph of the proof, the mapping $L^\infty(K)V\to \Lw(L_2(W))$ is a continuous inclusion. By complex interpolation, for $\alpha=\frac1p$,
$$L^\infty(K)V^\alpha\to \Lw^{(p)}(L_2(W))=[\mathbb{B}(L_2(W)),\Lw(L_2(W))]_\alpha,$$
is continuous.
\end{proof}

\begin{dfn}
Let $X$ be a locally compact space and $\mathcal{H}$ a Hilbert space with an action of $C_0(X)$. Let $G\in \mathbb{B}(\mathcal{H})$ be an operator.
\begin{itemize}
\item $G$ is said to be \emph{compactly based} if there is a $\chi\in C_c(X)$ such that $\chi G=G$. 
\item $G$ is said to be \emph{compactly supported} if $G$ and $G^*$ are compactly based. 
\end{itemize}
\end{dfn}

\begin{lem}
Assume that $\varphi$  is decreasing and has smooth regular variation. Whenever $G\in \mathbb{B}(L_2(W))$ is weakly $\varphi$-Laplacian modulated and $G^*$ is compactly based, then $G\in \Lw(L_2(W))$.
\end{lem}

\begin{proof}
Assume that $p\geq 1$ is the number for which $GV^{-1/p}\in \Lw^{(p(p-1)^{-1})}$ and set $\alpha=1/p$. The fact that $G^*$ is compactly based ensures that there is a $\chi\in C^\infty_c(W)$ such that $G\chi=G$. We write
$$G=GV^{-\alpha}V^{\alpha}=GV^{-\alpha}\chi V^{\alpha}+GV^{-\alpha}(1-\chi) V^{\alpha}=GV^{-\alpha}\chi V^{\alpha}-GV^{-\alpha}[\chi, V^{\alpha}]$$
By the assumption on $G$, $GV^{-\alpha}\in  \Lw^{(p(p-1)^{-1})}$ and by the assumption on $\varphi$, and Proposition \ref{locallylwpsi}, $\chi V^{\alpha}\in \Lw^{(p)}$ so $GV^{-\alpha}\chi V^{\alpha}\in \Lw$ by the H\"older inequalities \eqref{holdineq}. 

It remains to show that $GV^{-\alpha}[\chi, V^{\alpha}]=GV^{-\alpha}(V^\alpha\chi-\chi V^\alpha)\in \Lw$. However, by the assumption on $\varphi$, and Proposition \ref{locallylwpsi}, both $\chi V^{\alpha}$ and $V^\alpha\chi=(\chi V^\alpha)^*$ belong to $\Lw^{(p)}$. By assumption on $G$, $GV^{-\alpha}\in  \Lw^{(p(p-1)^{-1})}$, so $GV^{-\alpha}[\chi, V^{\alpha}]\in \Lw$ by the H\"older inequalities \eqref{holdineq}. 
\end{proof}

\subsection{Symbols of Laplacian modulated operators}

We now return to the main line of inquiry in this section. It follows from Proposition \ref{schbelong} that if $G\in \Lw(L_2(W))$, and $\varphi\in R_{-1}$ is decreasing, then there is a function $p_G\in L_2(W\times W^*)$ such that for $f\in L_2(W)$, 
$$Gf(x)=\int_{W^*} p_G(x,\xi)\hat{f}(\xi)\mathrm{e}^{i\langle x,\xi\rangle }\mathrm{d}\xi,$$
where $\hat{f}\in L_2(W^*)$ denotes the Fourier transform of $f$. The function $p_G$ is uniquely determined by $G$ and is called the $L_2$-symbol of $G$. It follows from the definition that $G$ is compactly based if and only if for some compact $K\subseteq W$, $p_G$ is supported in $K\times W^*$.

The following lemma describes strongly and spectrally $\varphi$-Laplacian modulated operators in terms of their Hilbert-Schmidt symbol. We use the standard notation $\langle \xi\rangle :=(1+|\xi|^2)^{1/2}$ for $\xi\in W^*$.

\begin{lem}
\label{p_mod} Let $\varphi\in R_{-1}$ be a decreasing function and 
let $G : L_2 (W) \to L_2 (W)$ be a Hilbert-Schmidt operator with $L_2$-symbol $p_G$. The operator $G$ is spectrally $\varphi$-Laplacian modulated if and only if
$$\sup_{t>0} \frac{1}{\varphi(t)} \int_{\varphi(\langle \xi\rangle^{d})< 1/t} \int_{W} |p_G(x,\xi)|^2 \, \mathrm{d}x \mathrm{d}\xi <\infty.$$
The operator $G$ is strongly $\varphi$-Laplacian modulated if and only if
$$\sup_{t>0} \frac{1}{\varphi(t)} \int_{W\times W^*} \frac{|p_G(x,\xi)|^2}{(1+t\varphi(\langle \xi\rangle^{d}))^2} \, \mathrm{d}x \mathrm{d}\xi <\infty.$$
\end{lem}

The proof is similar to that of~\cite[Lemma 11.3.13]{LSZ} and is omitted. To give an example of a spectrally modulated operator, we consider a localization of $\varphi((1-\Delta)^{d/2})$.

\begin{prop}
\label{psi_mod} 
Let $\varphi\in R_{-1}$ be decreasing and satisfy property (W2) (see Definition~\ref{someavass}). For every function $\chi\in C^\infty_c(W)$ the operator $G:=\chi \varphi[(1-\Delta)^{d/2}]$ is spectrally $\varphi$-Laplacian modulated.
\end{prop}

\begin{proof}
Note that the symbol of $G$ is $p_G(x,\xi) = \chi(x)\varphi(\langle \xi\rangle^d) .$ Let $K$ be a compact set containing the support of the function $\chi$. We obtain 
\begin{align*}
\int_{\varphi(\langle \xi\rangle^{d})< 1/t} &\int_{W} |p_G(x,\xi)|^2\, \mathrm{d}x \mathrm{d}\xi \le c \cdot \mathrm{vol}(K) \int_{\varphi(\langle \xi\rangle^{d})< 1/t} \varphi^2(\langle \xi\rangle^d)\,\mathrm{d}\xi\\ 
&\le  c_1 \int_{\varphi(|\xi|^{d})< 1/t} \varphi^2(| \xi|^d)\,\mathrm{d}\xi = \frac{c_1}d \int_{ \varphi^{-1}(1/t)}^\infty \varphi^2(s)\,\mathrm{d}s.
\end{align*}

It follows from Proposition~\ref{regvarresult}, that  
$$\int_{ \varphi^{-1}(1/t)}^\infty \varphi^2(s)\,\mathrm{d}s =O(\left. s \varphi^2(s)\right|_{s=\varphi^{-1}(1/t)})=O(\frac{\varphi^{-1}(1/t)}{t^2})=O(\varphi(t)), \ t\to \infty,$$
by using property (W2).

Hence,
\begin{align*}
\int_{\varphi(\langle \xi\rangle^{d})< 1/t} \int_{W} |p_G(x,\xi)|^2\, \mathrm{d}x \mathrm{d}\xi\le C \varphi(t), \ t>0.
\end{align*}
Thus, $G$ is spectrally $\varphi$-Laplacian modulated by Lemma~\ref{p_mod}.
\end{proof}

The following definition introduces two characteristics of the symbol of a Hilbert-Schmidt operator.

\begin{dfn}
\label{phidec}
Let $\varphi\in R_{-1}$ be a decreasing function and 
let  $G\in \mathcal{L}_2(L_2(W))$ be an operator with $L_2$-symbol $p_G$. We say that $G$ has \emph{$\varphi$-moderate growth} if 
$$\sup_k \int_{k<\Phi(\langle \xi\rangle^{d})<k+1}\int_W |p_G(x,\xi)| \, \mathrm{d}x \mathrm{d}\xi <\infty.$$
We say that the operator $G$ has \emph{$\varphi$-reasonable decay} if 
$$\int_{W\times W^*} \frac{|p_G(x,\xi)|}{\langle t-\langle\xi\rangle^d\rangle}\, \mathrm{d}x \mathrm{d}\xi =o(\Phi(t)), \quad\mbox{as $t\to \infty$}.$$

\end{dfn}

If $\varphi(t)=\frac1{\e+t}$, then the symbol of a spectrally $\varphi$-Laplacian modulated operators has the above properties~\cite[Proposition 11.3.18]{LSZ}. For the general case it is true only under some additional conditions.

\begin{lem}
\label{summinupp}
Let $G$ be a Hilbert-Schmidt operator on $L_2(W)$ with $\varphi$-moderate growth. Then it holds that 
$$\int_{\langle \xi\rangle^{d}<t}\int_W|p_G(x,\xi)|\mathrm{d}x\mathrm{d}\xi=O(\Phi(t)), \quad\mbox{as $t\to \infty$}.$$
\end{lem}

\begin{proof}
Take $n=\lfloor \Phi(t)\rfloor$. We have that 
\begin{align*}
\int_{\langle \xi\rangle^{d}<t}\int_W|p_G(x,\xi)|\mathrm{d}x\mathrm{d}\xi&\leq \sum_{k=0}^n \int_{k<\Phi(\langle \xi\rangle^{d})<k+1}\int_W |p_G(x,\xi)|\, \mathrm{d}x \mathrm{d}\xi=\\
&=O(1)(n+1)=O(\Phi(t)).
\end{align*}
\end{proof}

\begin{prop}
\label{weirdassonpsi} 
If $\varphi\in R_{-1}$ is decreasing and $\Phi(t):=\int_0^t \varphi(s) \mathrm{d}s$, then there is an $\epsilon>0$ and a constant $C>0$ such that
 $$t-\Phi^{-1}(\Phi(t)-1)\geq Ct^{\frac12+\epsilon} \ \text{for $t$ large}.$$
In particular, $t-\Phi^{-1}(\Phi(t)-1)\to \infty$ as $t\to \infty$.
\end{prop}

\begin{proof}
Since $\varphi\in R_{-1}$, $\Phi$ satisfy~\eqref{c1} and thus $\Phi$ varies slowly (see Definition~\ref{SV} above). A direct verification shows that
$$\Phi(t)= c(t) \cdot \exp\left(\int_a^t f(s) \frac{\mathrm{d}s}s\right),$$
where $c\equiv 1$, $a>0$ is such that $\Phi(a)=1$ and $f(s)=\frac{s \varphi(s)}{\Phi(s)}$. Since $c$ is a constant function, the function $\Phi$ is a normalised slowly varying function (see~\cite[formula (1.3.4)]{RegVar}). In view of~\cite[Theorem 1.5.5]{RegVar} the class of such function coincides with the Zygmund class. In particular,  it means that for every $0< \epsilon<1/2$ the function $t \mapsto t^{-1/2+\epsilon}\Phi(t)$ is decreasing for large values of $t$. Since $\Phi$ is non-negative, it follows that the function $t \mapsto t^{-1/2+\epsilon}\Phi(t)$ is convergent. Moreover, it converges to zero.
Indeed, suppose to the contrary that $t \mapsto t^{-1/2+\epsilon}\Phi(t) \to const \neq0$ as $t\to\infty.$ Then,
$$\frac{(2t)^{-1/2+\epsilon}\Phi(2t)}{t^{-1/2+\epsilon}\Phi(t)}\mathop{\longrightarrow}\limits_{t\to\infty} 1.$$
However, since the function $\Phi$ is slowly varying, direct computation show that the above expression tends to $2^{-1/2+\epsilon}\neq 1$. This contradiction shows that $t^{-1/2+\epsilon}\Phi(t)$ converges to zero.
Thus, $\Phi(t)\le t^{1/2-\epsilon}$ for large values of $t$.

Since both functions $t^{1/2-\epsilon}$ and $\Phi$ are concave, it follows that 
\begin{equation}
\label{convace}
\Phi(t_1)-\Phi(t_2)\le t_1^{1/2-\epsilon}-t_2^{1/2-\epsilon} \quad \mbox{for large $t_1, t_2$.}
\end{equation}

The functions $t^{1/2-\epsilon}$ and $\Phi$ are continuous and increasing to $+\infty$, hence it follows that for every $t>0$ such that $\Phi(t)\ge1$ ($t^{1/2-\epsilon}\ge1$, resp.) there exist $\alpha_t>0$ ($\beta_t>0$, resp.) such that $\Phi(t)-\Phi(t-\alpha_t)=1$ ($t^{1/2-\epsilon}-(t-\beta_t)^{1/2-\epsilon}=1$, resp.). From the latter equality we find that
$$\beta_t= t-\left(t^{\frac12-\epsilon}-1\right)^{\frac{2}{1-2\epsilon}}= t\left(1-\left(1-t^{-\frac12+\epsilon}\right)^{\frac{2}{1-2\epsilon}}\right)=c_1t^{\frac12+\epsilon}+O(t^{2\epsilon}) \ \text{as} \ t\to \infty,$$
where for the latter equality we used the Taylor expansion of the function $s\mapsto (1-s)^a$.
We conclude that for a suitable $C>0$, $\beta_t\geq Ct^{\frac12+\epsilon}$. Moreover, the estimate \eqref{convace} implies that for large $t$ we have
$$\Phi(t)-\Phi(t-\beta_t)\le t^{1/2}- (t-\beta_t)^{1/2}=1 =\Phi(t)-\Phi(t-\alpha_t).$$
Thus, $\Phi(t-\beta_t)\ge \Phi(t-\alpha_t)$ and $\alpha_t\ge\beta_t\geq Ct^{\frac12+\epsilon}$. Hence, $\alpha_t\to \infty$ as $t\to \infty$.

Finally, we obtain
$$t-\Phi^{-1}(\Phi(t)-1) = t-\Phi^{-1}(\Phi(t-\alpha_t))=\alpha_t\geq Ct^{\frac12+\epsilon}\quad\mbox{for $t$ large},$$
as required.
\end{proof}

\begin{lem}
\label{growthtodecay}
Let $\varphi\in R_{-1}$ be decreasing. If the compactly based Hilbert-Schmidt operator $G$ on $L_2(W)$ has $\varphi$-moderate growth and is spectrally $\varphi$-Laplacian modulated, then $G$ has $\varphi$-reasonable decay.
\end{lem}

\begin{proof}
We divide up the integral $\int_{W\times W^*} \frac{|p_G(x,\xi)|}{\langle t-\langle\xi\rangle^d\rangle}\, \mathrm{d}x \mathrm{d}\xi$ 
into three terms corresponding to the three regions $\Phi(\langle\xi\rangle^d)\leq \Phi(t)-1$, $\Phi(t)-1< \Phi(\langle\xi\rangle^d)\leq \Phi(t)+1$ and $\Phi(t)+1< \Phi(\langle\xi\rangle^d)$. 

If $\Phi(\langle\xi\rangle^d)\leq \Phi(t)-1$, we have $\langle\xi\rangle^d\leq \Phi^{-1}(\Phi(t)-1)$. As such, $\langle t-\langle\xi\rangle^d\rangle^{-1}\leq \langle t-\Phi^{-1}(\Phi(t)-1)\rangle$ in the region $\Phi(\langle\xi\rangle^d)\leq \Phi(t)-1$. Combining this with Proposition \ref{summinupp}, we arrive at the estimate 
\begin{align*}
\int_{\Phi(\langle\xi\rangle^d)\leq \Phi(t)-1}\int_W \frac{|p_G(x,\xi)|}{\langle t-\langle\xi\rangle^d\rangle}\, \mathrm{d}x \mathrm{d}\xi = O\left(\frac{\Phi(t)}{\langle t-\Phi^{-1}(\Phi(t)-1)\rangle}\right)=o(\Phi(t)).
\end{align*}
In the last equality, we used Lemma~\ref{weirdassonpsi}.

In the region $\Phi(t)-1< \Phi(\langle\xi\rangle^d)\leq \Phi(t)+1$, we use that $\langle t-\langle\xi\rangle^d\rangle\geq1$ and the fact that $G$ has $\varphi$-moderate growth to estimate that 
\begin{align*}
\int_{\Phi(t)-1< \Phi(\langle\xi\rangle^d)\leq \Phi(t)+1}\int_W &\frac{|p_G(x,\xi)|}{\langle t-\langle\xi\rangle^d\rangle}\, \mathrm{d}x \mathrm{d}\xi \leq\\
&\leq  \int_{\Phi(t)-1< \Phi(\langle\xi\rangle^d)\leq \Phi(t)+1}\int_W |p_G(x,\xi)|\, \mathrm{d}x \mathrm{d}\xi=O(1)=o(\Phi(t)).
\end{align*}

Let $K$ be a compact set for which $p_G$ is supported in $K\times W^*$. By Proposition \ref{weirdassonpsi}, the estimate $\langle t-\langle\xi\rangle^d\rangle\geq \langle \xi\rangle^{d/2+\epsilon}$ holds in the region $ \Phi(t)+1< \Phi(\langle\xi\rangle^d)$ and we use the Cauchy-Schwarz inequality to estimate
\begin{align*}
\int_{\Phi(t)+1< \Phi(\langle\xi\rangle^d)}\int_W &\frac{|p_G(x,\xi)|}{\langle t-\langle\xi\rangle^d\rangle}\, \mathrm{d}x \mathrm{d}\xi \leq\\
&\leq  \int_{\Phi(t)+1< \Phi(\langle\xi\rangle^d)} \int_K \frac{|p_G(x,\xi)|}{\langle\xi\rangle^{d/2+\epsilon}}\, \mathrm{d}x \mathrm{d}\xi\leq\\
&\leq  \left(\mathrm{vol}(K)\int_{\langle \xi\rangle^d\geq t}\langle \xi\rangle^{-d-2\epsilon}\mathrm{d}\xi \int_{\langle \xi\rangle^d\geq t}|p_G(x,\xi)|^2\, \mathrm{d}x \mathrm{d}\xi\right)^{1/2}\leq \\
&\le O(1) \left(t^{-2\epsilon}\int_{\varphi(\langle \xi\rangle^d)\leq \frac{1}{\varphi(t)^{-1}}}|p_G(x,\xi)|^2\, \mathrm{d}x \mathrm{d}\xi\right)^{1/2}\leq\\
&\leq  O\left(\sqrt{t^{-2\epsilon}\varphi\left(\frac{1}{\varphi(t)}\right)}\right)\|G\|_{\varphi, \varphi((1-\Delta)^{d/2}),spec}=O(t^{-\epsilon})=o(\Phi(t)).
\end{align*}
In the second last equality, we used that $\varphi$ is bounded.
\end{proof}

\begin{prop}
\label{psi_growth} 
Let $\varphi\in R_{-1}$ be decreasing and satisfying property (W2) (see Definition~\ref{someavass}). For every function $\chi\in C^\infty_c(W)$ the operator $G:=\chi \varphi[(1-\Delta)^{d/2}]$ has $\varphi$-moderate growth and $\varphi$-reasonable decay.
\end{prop}

\begin{proof}
Note that the symbol of $G$ is $p_G(x,\xi) = \chi(x)\varphi(\langle \xi\rangle^d)$. Let $K$ denote the support of the function $\chi$. For every $k\in \N$ we obtain 
\begin{align*}
\int_{k<\Phi(\langle \xi\rangle^{d})<k+1}\int_W |p_G(x,\xi)| \, \mathrm{d}x \mathrm{d}\xi&\le c \cdot \mathrm{vol}(K) \int_{k<\Phi(s)<k+1} \varphi(s)\,\mathrm{d}s\\ 
&\le  c_1 \left.\Phi(s)\right\vert_{\Phi(s)=k}^{\Phi(s)=k+1} =c_1.
\end{align*}

Thus, $G$ has $\varphi$-moderate growth. Since $G$ is spectrally $\varphi$-Laplacian modulated by Proposition~\ref{psi_mod}, it follows from Lemma~\ref{growthtodecay} that $G$ has $\varphi$-reasonable decay.
\end{proof}

The following result shows that under additional assumptions on $\varphi$ the symbol of any spectrally $\varphi$-Laplacian modulated operator have nice properties.

\begin{lem}
\label{phimodphireas}
Let $\varphi\in R_{-1}$ be decreasing and satisfy that $\varphi(t)=O(t^{-1})$ as $t\to \infty$. Then compactly based spectrally $\varphi$-Laplacian modulated Hilbert-Schmidt operators on $L_2(W)$ have $\varphi$-moderate growth, and in particular $\varphi$-reasonable decay.
\end{lem}

\begin{proof}
Let $G$ be a compactly based spectrally $\varphi$-Laplacian modulated Hilbert-Schmidt operators on $L_2(W)$ with $L_2$-symbol $p_G$. Take a compact set $K$ such that $p_G$ is supported in $K\times W^*$. We start by making the following estimate:
\begin{equation}
\label{phasevolest}
\int_{k<\Phi(\langle\xi\rangle^d)<k+1}\mathrm{d}\xi\leq C\int_{\Phi^{-1}(k)}^{\Phi^{-1}(k+1)} \mathrm{d} s=C\int_{k}^{k+1} \frac{\mathrm{d}u}{\varphi(\Phi^{-1}(u))}\leq C\frac{1}{\varphi(\Phi^{-1}(k))}.
\end{equation}
where we make the change of variables $u=\Phi(s)$ and use that $\varphi=\Phi'$ is decreasing. Here $C$ is a constant depending only on $\Phi$ and the dimension $d$.

Using the Cauchy-Schwarz inequality and the estimate \eqref{phasevolest}, we estimate that 
\begin{align*}
\int_{k<\Phi(\langle\xi\rangle^d)<k+1}&\int_W |p_G(x,\xi)| \, \mathrm{d}x \mathrm{d}\xi \leq \\
&\leq \left(C\frac{\mathrm{vol}(K)}{\varphi(\Phi^{-1}(k))}\int_{k<\Phi(\langle\xi\rangle^d)<k+1}\int_W |p_G(x,\xi)|^2 \, \mathrm{d}x \mathrm{d}\xi\right)^{1/2}\\
&\leq C'\left(t\int_{\varphi(\langle\xi\rangle^d)<1/t}\int_W |p_G(x,\xi)|^2 \, \mathrm{d}x \mathrm{d}\xi\right)^{1/2}
\end{align*}
where $t=1/\varphi(\Phi^{-1}(k))$. Using that $\varphi(t)=O(t^{-1})$, we conclude that 
\begin{align*}
\int_{k<\Phi(\langle\xi\rangle^d)<k+1}\int_W |p_G(x,\xi)| \, \mathrm{d}x \mathrm{d}\xi \leq C''\sup_{t>0}\frac{1}{\sqrt{\varphi(t)}}\|GE_{\varphi((1-\Delta)^{d/2})}[0,t^{-1}]\|_{\mathcal{L}_2}.
\end{align*}
This argument shows that compactly based spectrally $\varphi$-Laplacian modulated Hilbert-Schmidt operators on $L_2(W)$ have $\varphi$-moderate growth, and by Proposition \ref{growthtodecay} they also have $\varphi$-reasonable decay.
\end{proof}

\begin{rem}
Lemma \ref{phimodphireas} is crucial to proving Connes' trace formula for strongly Laplacian modulated operators in $\mathcal{L}_{1,\infty}$, i.e. when $\varphi(t)=\frac{1}{\e+t}$, as is best seen in \cite[Chapter 11]{LSZ}. By Lemma \ref{phimodphireas}, the assumption of being strongly Laplacian modulated makes the property of having $\varphi$-reasonable decay superfluous when $\varphi(t)=O(t^{-1})$. Since we are interested in situations when $\varphi(t)=O(t^{-1})$ fails, we will separate the operator theoretical condition of being modulated (in a suitable way) from the symbol property of having $\varphi$-reasonable decay; both conditions are needed in our approach to proving Connes' trace formula.
\end{rem}

\section{Connes' trace formula in the local model}
\label{cttsection}

As above, we let $W$ denote a $d$-dimensional inner product space with negative Laplacian $\Delta$. Inner product spaces locally model Riemannian manifolds via the tangent space.

The following result is an extension of Connes' trace formula to $\varphi$-Laplacian modulated operators. 

\begin{thm}
\label{CTT_R}
Let $\varphi\in SR_{-1}$ be decreasing. Assume that $G\in \Lw(L_2(W))$ is a compactly supported weakly $\varphi$-Laplacian modulated operator with $\varphi$-reasonable decay and $L_2$-symbol $p_G$. Then, as $n\to \infty$,
$$\int_{W} \int_{\langle \xi\rangle \le n^{1/d}} p_G(x,\xi)\,\mathrm{d}\xi \mathrm{d}x=O(\Phi(n)),$$
and for every extended limit $\omega$ on $\ell_\infty$ we have
$${\rm Tr}_\omega(G) = \omega \left(\frac1{\Phi (n+1)} \int_{W} \int_{\langle \xi\rangle\le n^{1/d}} p_G(x,\xi)\,\mathrm{d}\xi \mathrm{d}x\right).$$
\end{thm}

Having Theorem~\ref{weaklymodform}, the key point of the proof is to find a relation between expectation values of an operator and the integral of its symbol. To this end, we make some preparatory remarks and lemmas.

\begin{rem}
The reader can compare Theorem \ref{CTT_R} to the statement \cite[Theorem 11.5.1]{LSZ} which considers the case $\varphi(t)=\frac{1}{\e+t}$ (when $\Lw=\mathcal{L}_{1,\infty}$). The result \cite[Theorem 11.5.1]{LSZ} is stated for compactly supported strongly $\varphi$-modulated operators for $\varphi(t)=\frac{1}{\e+t}$. At this point, we are dealing with a noncompact space  and the results from Section \ref{subsec:stronglymod} and Section \ref{subsec:symbprop} are merely indicative of the implication that strong $\varphi$-Laplacian modulation implies weak $\varphi$-Laplacian modulation and $\varphi$-reasonable decay. For $\varphi$-pseudo-differential operators on a closed manifold (which we introduce below in Section \ref{psipseudos}), this implication is true by Proposition \ref{growthtodecay} and Theorem \ref{factpsps}. 
\end{rem}

\begin{rem}
If $G$ is an operator satisfying the assumptions of Theorem \ref{CTT_R} and additionally has an $L_2$-symbol of $\varphi$-moderate growth, Lemma \ref{summinupp} implies that not only do we have $\int_{W} \int_{\langle \xi\rangle \le n^{1/d}} p_G(x,\xi)\,\mathrm{d}\xi \mathrm{d}x=O(\Phi(n))$ but in fact $\int_{W} \int_{\langle \xi\rangle \le n^{1/d}} |p_G(x,\xi)|\,\mathrm{d}\xi \mathrm{d}x=O(\Phi(n))$
\end{rem}

The following definition introduces an analogue of Sobolev spaces. We will discuss them in more details (and in more general situations) in Section~\ref{sec:localizing}.

\begin{dfn}
\label{sobsparn}
For $\varphi:[0,\infty)\to (0,\infty)$, $W$ a $d$-dimensional inner product space, and $s\in \R$, we define the Hilbert space
\begin{align*}
H^s_\varphi(W):=&\varphi((1-\Delta)^{d/2})^{s/d}L_2(W), \quad\mbox{with the inner product}\\
& \langle f_1,f_2\rangle_{H^s_\varphi(W)}:=\langle \varphi((1-\Delta)^{d/2})^{-s/d}f_1,\varphi((1-\Delta)^{d/2})^{-s/d}f_2\rangle_{L_2(W)}.
\end{align*}
Let $\mathbb{T}^d:=\R^d/\Z^d$ be the $d$-torus equipped with its flat metric and $\Delta_{\mathbb{T}^d}$ the associated Laplacian. For $s\in \R$, we define the Hilbert space
\begin{align*}
\tilde{H}^s_\varphi(\mathbb{T}^d):=&\varphi((1-\Delta_{\mathbb{T}^d})^{d/2})^{s/d}L_2(\mathbb{T}^d), \quad\mbox{with the inner product}\\
& \langle f_1,f_2\rangle_{\tilde{H}^s_\varphi(\mathbb{T}^d)}:=\langle \varphi((1-\Delta_{\mathbb{T}^d})^{d/2})^{-s/d}f_1,\varphi((1-\Delta_{\mathbb{T}^d})^{d/2})^{-s/d}f_2\rangle_{L_2(\mathbb{T}^d)}.
\end{align*}
\end{dfn}

After choosing an ON-basis, we can identify $W=\R^d$, $\mathbb{T}^d$ with a quotient of $W$ by a lattice and the cube $(0,1)^d$ as a Lipschitz fundamental domain in $W$. We identify function spaces on $\mathbb{T}^d$ with $\Z^d$-invariant function spaces on $W$. 

\begin{lem}
\label{incluofhs}
Let $\varphi$ have smooth regular variation. Then for any $\chi\in C^\infty_c((0,1)^d)$ and $s\in \R$,
$$\chi \tilde{H}^s_\varphi(\mathbb{T}^d)=\chi H^s_\varphi(W),$$
with equivalent norms.
\end{lem}

The proof of this Lemma requires some heavier machinery, and we return to its proof below in Corollary \ref{compsobspaceoncube}.

\begin{lem}
\label{dirwemod}
Assume that $\varphi$ is a decreasing function with smooth regular variation. If $G\in \mathcal{L}_2(L_2(W))$ is compactly supported in $(0,1)^d$ and weakly $\varphi$-Laplacian modulated, then $G$ is weakly $\varphi$-modulated with respect to $\varphi((1-\Delta_{\mathbb{T}^d})^{d/2})\in \Lw$.

\end{lem}

\begin{proof}
The idea of the proof is to apply Lemma \ref{incluofhs}. Set $V:=\varphi((1-\Delta)^{d/2})$ and $V_{\mathbb{T}^d}:=\varphi((1-\Delta_{\mathbb{T}^d})^{d/2})$. If $G$ is compactly supported in $(0,1)^d$ and weakly $\varphi$-Laplacian modulated we can for some $s\in (0,d]$ factor $GV^{-s}$ as an operator 
\begin{equation}
\label{factorizinggvs}
L_2(W)\xrightarrow{V^{-s}}H^{-s}_\varphi(W)\xrightarrow{\chi}\chi H^{-s}_\varphi(W)\xrightarrow{G}L_2((0,1)^d)=L_2(\mathbbm{T}^d),
\end{equation}
where $\chi \in C^\infty_c((0,1)^d)$ satisfies that $G=G\chi$. In fact, that $G$ is weakly $\varphi$-Laplacian modulated is in this case equivalent to $G$ extending to an operator in $\Lw^{(\frac{d}{d-s})}(H^{-s}_\varphi(W),L_2(W))$, see Proposition \ref{weakregandsob}.

Assuming that $G$ is compactly supported in $K$, it is by definition equivalent for $G$ to be weakly $\varphi$-modulated with respect to $V_{\mathbb{T}^d}$ and $G$ extending to an operator in $\Lw^{(\frac{d}{d-s})}(\tilde{H}^{-s}_\varphi(\mathbb{T}^d),L_2(\mathbb{T}^d))$. Using Lemma \ref{incluofhs} and \eqref{factorizinggvs}, we can factor $G$ as 
$$\tilde{H}^{-s}_\varphi(\mathbb{T}^d)\xrightarrow{\chi}\chi\tilde{H}^{-s}_\varphi(\mathbb{T}^d)=\chi H^{-s}_\varphi(W)\xrightarrow{\chi' G}\chi'L_2(W)\subseteq L_2(\mathbb{T}^d),$$
for a $\chi'\in C^\infty_c((0,1)^d)$ such that $\chi'G=G$. Since $G\in \Lw^{(\frac{d}{d-s})}(H^{-s}_\varphi(W),L_2(W))$ it follows that $G\in \Lw^{(\frac{d}{d-s})}(\tilde{H}^{-s}_\varphi(\mathbb{T}^d),L_2(\mathbb{T}^d))$. We conclude that $GV_{\mathbb{T}^d}^{-s}\in \Lw^{(\frac{d}{d-s})}(L_2(\mathbb{T}^d))$, that is $G$ is weakly $\varphi$-modulated with respect to $V_{\mathbb{T}^d}$.
\end{proof}

We will need the following result (Lemma 11.4.4 in \cite{LSZ}). Using the above mentioned ON-basis, we can identify $\Z^d$ with a lattice in $W^*$.

\begin{lem}
For a Schwartz function $\phi\in \mathcal{S}(W)$ with $\phi=1$ on $[0,1]^d$ it holds that
\begin{equation}
\label{comparingfouriertocut}
\sum_{\mathbbm{k}\in \Z^d: \,\langle\mathbbm{k}\rangle^d\leq t} \mathrm{e}^{2\pi i\langle u,\xi-\mathbbm{k}\rangle}\hat{\phi}(\xi-\mathbbm{k})=\chi_{[0,t]}(\langle\xi\rangle^d)+O\left(\langle t-\langle\xi\rangle^d\rangle^{-1}\right),\quad t>0, \xi\in W^*,
\end{equation}
uniformly in $u\in [0,1]^d$.
\end{lem}

The following lemma is the key technical result of this section.

\begin{lem}
\label{comparingcondval}
Suppose that $G\in \mathcal{L}_2(L_2(W))$ is compactly supported in $(0,1)^d$ and has $\varphi$-reasonable decay. Then 
$$\sum_{\langle \mathbbm{k}\rangle^d\leq t}\langle G\mathrm{e}_\mathbbm{k},\mathrm{e}_\mathbbm{k}\rangle_{L_2((0,1)^d)}-\int_{\langle\xi\rangle^d\leq t}\int_W p_G(x,\xi)\mathrm{d}x\mathrm{d}\xi=o(\Phi(t)).$$
Here $\mathrm{e}_\mathbbm{k}\in L_2((0,1)^d)=L_2(\mathbb{T}^d)$ denotes the Fourier basis $\mathrm{e}_\mathbbm{k}(u):=\mathrm{e}^{2\pi i\langle \mathbbm{k},u\rangle}$, $\mathbbm{k}\in \Z^d$, which is an eigenbasis for the flat Laplacian $\Delta_{\mathbb{T}^d}$ on $\mathbb{T}^d$.
\end{lem}

The proof of Lemma \ref{comparingcondval} is identical to that of \cite[Lemma 11.4.6]{LSZ}. The reader should beware of the fact that $\e_\mathbbm{k}$ is used incorrectly as an eigenbasis for the Dirichlet-Laplacian on $(0,1)^d$ in \cite{LSZ}. We remark that Lemma \ref{comparingcondval} is \emph{the reason} for making the assumption that $G$ has $\varphi$-reasonable decay in the statement of Theorem \ref{CTT_R}. To add slightly more detail, we point towards the error term in Equation \eqref{comparingfouriertocut} as justification for the definition of $\varphi$-reasonable decay.

\begin{proof}[Proof of Theorem \ref{CTT_R}]
We can assume that $G$ is compactly supported in $(0,1)^d$. By Lemma \ref{dirwemod}, $G$ is weakly $\varphi$-modulated with respect to $\varphi((1-\Delta_{\mathbb{T}^d})^{d/2})$. By Lemma \ref{l_exp} and Lemma \ref{comparingcondval}, 
$$\sum_{k=0}^n (\lambda(k,\Re G)-i\lambda(k,\Re (iG))) - \int_{\langle\xi\rangle^d\leq n}\int_W p_G(x,\xi)\mathrm{d}x\mathrm{d}\xi=o(\Phi(n))\ \text{as} \ n\to \infty.$$ 
Since $G\in \Lw$, it follows that $\Re G, \Re (iG)\in \Lw$ and thus the sum in the expression above is $O(\Phi(n))$. Hence,
$$\int_{\langle\xi\rangle^d\leq n}\int_W p_G(x,\xi)\mathrm{d}x\mathrm{d}\xi=O(\Phi(n))\ \text{as} \ n\to \infty.$$ 

By Theorem \ref{weaklymodform} and Lemma \ref{comparingcondval} we compute that 
$${\rm Tr}_\omega(G) = \omega \left(\frac1{\Phi(n+1)} \sum_{\langle \mathbbm{k}\rangle^d\leq n} \langle G\mathrm{e}_\mathbbm{k},\mathrm{e}_\mathbbm{k}\rangle_{L_2((0,1)^d)}\right)=\omega \left(\frac1{\Phi (n+1)} \int_{W} \int_{\langle\xi\rangle^{d}\le n} p_G(x,\xi)\,\mathrm{d}\xi \mathrm{d}x\right).$$
\end{proof}

\section{$\varphi$-Pseudo-differential operators}
\label{psipseudos}

In this section we consider a special class of pseudo-differential operators. They are inspired by and generalizes Lesch's construction of $\log$-classical pseudo-differential operators. The goal of this section is to provide a general machinery for studying pseudo-differential operators belonging to a principal ideal $\mathcal{L}_\varphi$ and the computation of their Dixmier traces. Much of the behaviour found in the class of pseudo-differential that we study resembles that in the class of ordinary pseudo-differential operators, of which it is a subclass. We will briskly recall the basic constructions of pseudo-differential operators following \cite{hormanderIII,Taylorpsidobook} after which we proceed to define the so called $\varphi$-pseudo-differential operators and consider their applications to more general situations.

Throughout this section, the minimal assumption we impose on $\varphi$ is that it has smooth regular variation of any index $\rho>0$. The relevance of this assumption on pseudo-differential calculi stems from the following lemma.

\begin{lem}
\label{A14}
If a function $\varphi$ has smooth regular variation, one has
\begin{enumerate}
\item[(i)] for any $k,m\in \N$, there are constants $C_{k,m}\geq0$ such that
$$\left|\partial_t^m\left( \frac{\partial_t^k\varphi}{\varphi}\right)\right|(t)\leq C_{k,m}\langle t\rangle^{-k-m};$$

\item[(ii)] there exists a constant $m_\varphi\in \R$ such that
$|\varphi(t)|\leq C\langle t\rangle^{m_\varphi}. $
\end{enumerate}
\end{lem}

\begin{proof} (i) 
For $m=0$ the estimate follows from the definition of smoothly regularly varying function and the boundedness of $\varphi$. Suppose that the estimate holds for every $0\le m \le n-1$ and every $k\in \N$. Since
$$\partial_t\left( \frac{\partial_t^k\varphi}{\varphi}\right)= \frac{\partial_t^{k+1}\varphi}{\varphi}-\frac{\partial_t^k\varphi}{\varphi}\cdot \frac{\partial_t\varphi}{\varphi},$$
it follows that
$$\partial_t^{n}\left( \frac{\partial_t^k\varphi}{\varphi}\right)= \partial_t^{n-1}\left( \frac{\partial_t^{k+1}\varphi}{\varphi}\right)-\sum_{j=0}^{n-1} \begin{pmatrix}
m-1\\j
\end{pmatrix}  \partial_t^{n-1-j}\left( \frac{\partial_t^k\varphi}{\varphi}\right)\cdot \partial_t^{j}\left( \frac{\partial_t\varphi}{\varphi}\right).$$
Using the assumption of induction we obtain
\begin{align*}
\left|\partial_t^n\left( \frac{\partial_t^k\varphi}{\varphi}\right)\right|&\leq C_{k+1,n-1}\langle t\rangle^{-k-1-n+1}+\sum_{j=0}^{n-1} \begin{pmatrix}
m-1\\j
\end{pmatrix} C_{k,n-1-j}\langle t\rangle^{-k-n+1+j}C_{1,j}\langle t\rangle^{-1-j}\\
&=C_{k,n}\langle t\rangle^{-k-n}.
\end{align*}
This proves part (i). The statement in part (ii) follows from that for $\varphi\in SR_\rho$, $\varphi(t)t^{-\rho}\in SR_0$ is slowly varying. Slowly varying functions are polynomially bounded. 
\end{proof}

Recall that $W$ denotes a $d$-dimensional inner product space and $\langle\xi\rangle:=(1+|\xi|^2)^{1/2}$, $\xi\in W^*$. If $\varphi$ has smooth regular variation, the function 
$$\varphi_0(\xi):=\varphi(\langle \xi\rangle^d), \quad \xi\in W^*,$$
is a H\"ormander symbol of order $m_\varphi d$ and type $(1,0)$. Here $m_\varphi$ is as in Lemma  \ref{A14}. This fact follows from Lemma \ref{A14} and that for suitable coefficients $(c_{\beta,k,l})$, 
\begin{equation}
\label{derivativesofpsizero}
\partial_\xi^\alpha \varphi_0(\xi)=\sum_{|\alpha|+|\beta|=k+l} c_{\beta,k,l}\xi^\beta \langle\xi\rangle^{k(d-1)-l}[\partial_t^k\varphi](\langle\xi\rangle^d).
\end{equation}

The examples we have in mind for $\varphi$:s in this section are of the form $\varphi(t)=\langle t\rangle^m\log^k(\e+t)$ for $m\in \R$ and $k\in \Z$. These functions provide examples of elements in $SR_m$. In this case, $m_\varphi$ is an arbitrary number $>m$ if $k>0$ and $m_\varphi=m$ if $k\leq 0$.

\subsection{The local model using $\varphi$-symbols}

The local model for $\varphi$-pseudo-differential operators are operators whose symbols behave like $\varphi$ at infinity. For an open subset $U\subseteq W$, we write $S^m(U)$ for the space of all symbols of order $m$ and H\"ormander type $(1,0)$ on $U$, see \cite[Definition 18.1.]{hormanderIII}. To be precise, $S^m(U)\subseteq C^\infty(U\times W^*)$ and a function $a\in C^\infty(U\times W^*)$ belongs to $S^m(U)$ if for any $\alpha,\beta\in \N^d$ and any compact $K\subseteq U$, there are constants $C_{\alpha,\beta,K}>0$ such that 
$$|\partial_x^\alpha\partial_\xi^\beta a(x,\xi)|\leq C_{\alpha,\beta,K}\langle\xi\rangle^{m-|\beta|},$$
for $(x,\xi)\in K\times W^*$. We write $S^{-\infty}(U):=\cap_{m\in \R}S^m(U)$, see \cite[Definition 18.1.1]{hormanderIII}. If $a_1,a_2\in S^m(U)$ we write $a_1\sim a_2$ if $a_1-a_2\in S^{-\infty}(U)$. The space of symbols $S^m(U)$ are asymptotically complete in the following sense. By \cite[Proposition 18.1.3]{hormanderIII}, if $(a_j)_{j\in \N}\subseteq S^m(U)$ is a sequence such that $a_j\in S^{m_j}(U)$ for a sequence $m_j\to -\infty$, then there is a symbol $a\in S^m(U)$ such that for any $N$, there is a $k>0$ such that $a-\sum_{j<k}a_j\in S^{-N}(U)$. We write $a\sim \sum_j a_j$. The reader should beware that $\sum_j a_j$ rarely exists as a pointwise defined sum.

For a symbol $a\in S^m(U)$, the associated operator $Op(a):=C^\infty_c(U)\to C^\infty(U)$ is defined by
\begin{equation}
\label{opdef}
Op(a)f(x):=\frac{1}{(2\pi)^d}\int_{W^*} a(x,\xi)\hat{f}(\xi)\mathrm{e}^{ix\cdot \xi}\mathrm{d}\xi,
\end{equation}
where $\hat{f}$ denotes the Fourier transform of $f$. In \cite[Theorem 18.1.6]{hormanderIII}, the operator $Op(a)$ is denoted by $a(x,D)$. The space $L^{-\infty}(U):=C^\infty(U\times U)$ acts as operators $C^\infty(U)'\to C^\infty(U)$, we call these operators smoothing operators. We define $L^m(U):=Op(S^m(U))+L^{-\infty}(U)$ which is a subspace of the space of operators $C^\infty_c(U)\to C^\infty(U)$ by \cite[Theorem 18.1.6]{hormanderIII}. An element of $L^m(U)$ is called a pseudo-differential operator of order $m$ on $U$. For two pseudo-differential operators $P_1$ and $P_2$, we write $P_1\sim P_2$ if $P_1-P_2\in L^{-\infty}(U)$. The quantization in Equation \eqref{opdef} induces an isomorphism $Op:S^m(U)/S^{-\infty}(U)\to L^m(U)/L^{-\infty}(U)$ by \cite[Proposition 18.1.19]{hormanderIII}. 

A pseudo-differential operator $P \in L^m(U)$ on an open subset $U\subseteq W$ is said to be properly supported if the Schwartz kernel $K_P$ satisfies that the two projections $\mathrm{supp}(K_P)\to U$ are proper mappings, see \cite[Definition 18.1.21]{hormanderIII}. We write $L^{m,{\rm prop}}(U)\subseteq L^m(U)$ for the space of properly supported pseudo-differential operators on $U$ of order $m$. In fact, properly supported pseudo-differential operators preserve compact support and we can consider $L^{m,{\rm prop}}(U)$ as a subalgebra of the space of linear operators on $C^\infty_c(U)$, see \cite[Definition 18.1.21, Theorem 18.1.23]{hormanderIII}. Similarly, operators from $L^{m,{\rm prop}}(U)$ naturally extend to operators on $C^\infty(U)$. The filtered spaces $(L^{m,{\rm prop}}(U))_{m\in \R}$ and $(L^{m}(U))_{m\in \R}$ come equipped with products that coincide with the composition of operators
$$L^{m}(U)\times L^{m',{\rm prop}}(U)\to L^{m+m'}(U)\quad\mbox{and}\quad L^{m,{\rm prop}}(U)\times L^{m'}(U)\to L^{m+m'}(U).$$

By \cite[Proposition 18.1.22]{hormanderIII} it holds that 
$$L^m(U)/L^{-\infty}(U)=L^{m,{\rm prop}}(U)/(L^{m,{\rm prop}}(U)\cap L^{-\infty}(U))=S^m(U)/S^{-\infty}(U).$$
We tacitly assume that the composition of operators is well defined. Since our operators arise from closed manifolds this assumption will not restrict us. If $P_i=Op(a_i)+S_i$ for $a_i\in S^{m_i}(U)$ and $S_i\in L^{-\infty}(U)$, $i=1,2$, then $P_1P_2=Op(b)+S_3$ where $S_3\in L^{-\infty}(U)$ and $b\in S^{m_1+m_2}(U)$ is given by 
\begin{equation}
\label{productform}
b(x,\xi)\sim \sum_{\alpha}\frac{1}{\alpha!} D^\alpha_\xi a_1(x,\xi)\partial^\alpha_xa_2(x,\xi).
\end{equation}
Here $D^\alpha_\xi:= (i \partial_\xi)^\alpha$, see \cite[Theorem 18.1.8]{hormanderIII}.

\begin{dfn}
\label{defnofsmphi}
Let $U\subseteq W$ be an open subset and $m\in \R$. Define the space of \emph{$\varphi$-symbols} on $U$ to be 
$$S_\varphi^m(U):=\varphi_0 S^m(U)+S^{-\infty}(U).$$ 
If $\varphi$ has smooth regular variation, we define 
$$L_\varphi^m(U):=Op(S_\varphi^m(U))+L^{-\infty}(U).$$
We also define the subspace $L_\varphi^{m,{\rm prop}}(U):=L_\varphi^m(U)\cap L^{m+m_\varphi d,{\rm prop}}(U)$ of properly supported operators. 
\end{dfn}

Note here that if $\varphi$ has smooth regular variation, $S_\varphi^m(U)\subseteq  S^{m+m_\varphi d}(U)$ by Lemma \ref{A14} and $Op(S_\varphi^m(U))\subseteq L^{m+m_\varphi d}(U)$ is well defined. We call an element of $L^m_\varphi(U)$ a \emph{$\varphi$-pseudo-differential operator} of order $m$ on $U$.

\begin{prop}
\label{asymptcomplete}
The space of symbols $S^m_\varphi(U)$ is asymptotically complete, that is, if $(a_j)_{j\in \N}\subseteq S^m_\varphi(U)$ is a sequence such that $a_j\in S^{m_j}_\varphi(U)$ for a sequence $m_j\to -\infty$, then there is a symbol $a\in S^m_\varphi(U)$ such that for any $N$, there is a $k>0$ such that $a-\sum_{j<k}a_j\in S^{-N}_\varphi(U)$. 
\end{prop}

We write $a\sim \sum_j a_j$ when in the situation of Proposition \ref{asymptcomplete}. The proof of Proposition \ref{asymptcomplete} is immediate from the asymptotic completeness of $S^m(U)$ (see discussion above or \cite[Proposition 18.1.3]{hormanderIII}) and is therefore omitted.

\begin{prop}
\label{productdecomp}
Assume that $\varphi$ has smooth regular variation. Then $T\in L_\varphi^m(U)$ if and only if there is a $T_0\in L^m(U)$ and a $S\in L^{-\infty}(U)$ such that $T=T_0Op(\varphi_0)+S$. 
\end{prop}

\begin{proof}
Since $L^m(U)/L^{-\infty}(U)=S^m(U)/S^{-\infty}(U)$, $L_\varphi^m(U)/L^{-\infty}(U)=S^m_\varphi(U)/S^{-\infty}(U)$. Indeed, $T\in L_\varphi^m(U)$ if and only if there is an $a\in S^m_\varphi(U)$ such that $T\sim Op(a)$. Finally, $a\in S^m_\varphi(U)$ if and only if $a=a_0\varphi_0+s$ for symbols $a_0\in S^m(U)$ and $s\in S^{-\infty}(U)$. Seeing that $\varphi_0$ only depends on $\xi$, we have the identity $Op(a_0\varphi_0)=Op(a_0)Op(\varphi_0)$.
\end{proof}

\begin{prop}
\label{products}
Assume that $\varphi$ has smooth regular variation. The subspace $S_\varphi^m(U)\subseteq S^{m+m_\varphi d}(U)$ is closed under pointwise multiplication by $S^0(U)$ and the operator product defines products
$$L^{m}(U)\times L^{m',{\rm prop}}_\varphi(U)\to L^{m+m'}_\varphi(U)\quad\mbox{and}\quad L^{m,{\rm prop}}(U)\times L^{m'}_\varphi(U)\to L^{m+m'}_\varphi(U).$$
In particular, the space $L_\varphi^{0,{\rm prop}}(U)$ forms a $*$-algebra if $m_\varphi\leq 0$. 
\end{prop}

\begin{proof}
By Proposition \ref{productdecomp}, it suffices to show that if $a_1\in S^m(U)$ and $a_2=a_0\varphi_0\in S^{m'}_\varphi(U)$, then $Op(a_1)Op(a_2)\in L^{m+m'}_\varphi(U)$. Equation \eqref{productform} implies that $Op(a_1)Op(a_2)=Op(b)$ where $b\in S^{m+m'}(U)$ takes the form 
\begin{equation}
\label{prodformpsi}
b(x,\xi)\sim \sum_{\alpha}\frac{1}{\alpha!} D^\alpha_\xi a_1(x,\xi)\partial^\alpha_xa_2(x,\xi)=\sum_{\alpha}\frac{1}{\alpha!} D^\alpha_\xi a_1(x,\xi)\partial^\alpha_xa_0(x,\xi)\varphi_0(\xi).
\end{equation}
Therefore, $Op(b)\sim Op(a_1)Op(a_0)Op(\varphi_0)$ which by Proposition \ref{productdecomp} belongs to $L^{m+m'}_\varphi(U)$.
\end{proof}

The action of an operator $T\in L^{0,{\rm prop}}(U)$ on $C^\infty_c(U)$ is locally bounded in the $L_2$-norm by \cite[Theorem 18.1.11]{hormanderIII}. By density, we can consider $L^{0,{\rm prop}}(U)$ as an algebra of bounded operators on $L_{2,loc}(U)$.

\subsection{Coordinate changes and additional assumptions on $\varphi$}

For a smoothly regularly varying $\varphi$, we now ensure that the property for an operator to have a symbol in $S_\varphi(U)$ is coordinate independent. With these properties at hand, we will be able to define $\varphi$-pseudo-differential operators on manifolds and prove that these operators are Laplacian modulated.

\begin{rem}
\label{higherdsecond}
Using Equation \eqref{derivativesofpsizero} and Lemma~\ref{A14}, it is easily seen that smooth regular variation implies that for any dimension $d$, the function $\varphi_0(\xi):=\varphi(\langle \xi\rangle^d)$ satisfies the following property for any multi-index $\alpha$,
$$\partial^\alpha_\xi \varphi_0\in S^{-|\alpha|}_\varphi.$$
\end{rem}

\begin{prop}
\label{psiidealprop}
If $\varphi$ is a decreasing function with smooth regular variation then for any $\chi,\chi'\in C^\infty_c(U)$ it holds that $\chi L_\varphi^{0,{\rm prop}}(U)\chi'=\chi L_\varphi^0(U)\chi'\subseteq \mathcal{L}_\varphi(L_2(U))$ when represented as operators on $L_2(U)$. 

\end{prop}

\begin{proof}
It is clear that $\chi L_\varphi^{0,{\rm prop}}(U)\chi'\subseteq \chi L_\varphi^0(U)\chi'$ and the converse inclusion follows from the fact that the Schwartz kernel of an operator in $\chi L_\varphi^0(U)\chi'$ is compactly supported in $U\times U$. Clearly, if $T_1\sim T_2$ in $L_\varphi^0(U)$, then $\chi T_1\chi'- \chi T_2\chi'$ smoothing and compactly supported, so it belongs to $\Lw(L_2(U))$. 

For $T\in \chi L_\varphi^{0,{\rm prop}}(U)$ there is a symbol $a_0\in S^0(U)$ and a $\tilde{\chi}\in C^\infty_c(U)$ such that 
$$T\sim \chi Op(a_0)\tilde{\chi}Op(\varphi_0).$$
Since $\chi Op(a_0)$ is compactly based and of order $0$, it acts as a bounded operator. It therefore suffices to show that $\tilde{\chi}Op(\varphi_0)\in \Lw(L_2(U))$ for any compactly supported $\tilde{\chi}\in C^\infty_c(U)$. We can assume that $U=W$ and the claim follows from Lemma \ref{psiideallem}.
\end{proof}

The reader wary of circular proofs should note that we are not using Proposition \ref{psiidealprop} to prove Corollary \ref{corofphidifflem} -- the result that proves the yet unproven Lemma \ref{psiideallem}.

\begin{prop}
\label{sigmaforquad}
Let $\varphi$ have smooth regular variation. Then for any $k\in \N$ and $b>0$, there is a constant $C_k>0$ such that,
$$\left|\partial_t^k\left(\frac{\varphi(t)}{\varphi(bt)}\right)\right|\leq C_k\langle t\rangle^{-k}.$$
In particular, for any positive definite quadratic form $g$ on $W$, the function
$$\sigma_g(\xi):=\frac{\varphi(\langle \xi\rangle^d)}{\varphi\left((1+|\xi|_g^2)^{d/2}\right)},$$ 
is a symbol $\sigma_g\in S^0(W)$ and H\"ormander type $(1,0)$ depending smoothly on $g$.
\end{prop}

\begin{proof}
Define the function $\sigma_b(t):=\frac{\varphi(t)}{\varphi(bt)}$. For suitable constants $c_{\alpha,l,j}\geq 0$, we can for any $k$ write 
$$\partial^k_t\sigma_b(t)=\sum_{l=0}^k\sum_{j=0}^l\sum_{\alpha\in \N^j, \, |\alpha|=j} c_{\alpha,l,j} \frac{b^l\partial_t^{k-l}\varphi(t)\prod_{p=1}^j \partial^{\alpha_p}_t\varphi(bt)}{\varphi(bt)^{j+1}}.$$
We set $a_m:=\frac{\partial_t^m\varphi}{\varphi}$, which is a symbol of order $-m$ and H\"ormander type $(1,0)$ by Lemma \ref{A14}. Using that $|\sigma_b(t)|$ is bounded (by~\eqref{quasi}), we estimate 
$$|\partial^k_t\sigma_b(t)|\leq C\sum_{l=0}^k\sum_{j=0}^l\sum_{\alpha\in \N^j, \, |\alpha|=l} c_{\alpha,l,j} |a_{k-l}(t)|\prod_{p=1}^j |a_{\alpha_p}(bt)|\leq C_k\langle t\rangle^{-k},$$
where we in the last step used that $a_{k-l}(t)\prod_{p=1}^j a_{\alpha_p}$ is a symbol of order $-(k-l)-|\alpha|=-k$ and H\"ormander type $(1,0)$. The verification of the statements about $\sigma_g$ goes as in the $1$-dimensional case $g=b$ modulo tedious computations, and will be omitted. 
\end{proof}

\begin{prop}
\label{consofsigmaforquad}
Let $\varphi$ have smooth regular variation and $g$ be a metric on an open subset $U\subseteq W$ and define $\sigma_g(x,\xi):=\sigma_{g(x)}(\xi)$, $(x,\xi)\in U\times W^*$. Then $\sigma_g\in S^0(U)$ is elliptic.
\end{prop}

\begin{proof}
By Proposition \ref{sigmaforquad}, $\sigma_g\in S^0(U)$ is a smooth symbol of order $0$. To verify that $\sigma_g$ is elliptic, we assume that $g_1$ and $g_2$ are two metrics on $U$ and consider the function
\begin{equation}
\label{sigmag1g2}
\sigma_{g_1,g_2}(x,\xi):=\frac{\sigma_{g_1}(x,\xi)}{\sigma_{g_2}(x,\xi)}=\frac{\varphi\left((1+|\xi|_{g_2(x)}^2)^{d/2}\right)}{\varphi\left((1+|\xi|_{g_1(x)}^2)^{d/2}\right)}
\end{equation}
The same proof as that of Proposition \ref{sigmaforquad} shows that $\sigma_{g_1,g_2}\in S^0(U)$. Since $\sigma_{g_1,g_2}=\sigma_{g_2,g_1}^{-1}$ the proof is complete.

\end{proof}

\begin{rem}
It should be remarked that we expect that the main results Theorem \ref{CTT} and Corollary \ref{CTT_psipsido} in this paper hold when relaxing the condition for $\varphi$ to have smooth regular variation to the expression in Equation \eqref{liminitsrrho} only being bounded for finitely many $k$ below some dimensionally dependent cut-off. The pseudo-differential techniques can in this case not be used in our cavalier way, but controlling the relevant remainders ought to only require finitely many derivatives of $\varphi_0$ as in \cite{toolspde}.
\end{rem}

\begin{prop}
\label{closedunderder}
Suppose that $\varphi$ has smooth regular variation. If $a\in S^m_\varphi(U)$, then $\partial_x^\alpha\partial_\xi^\beta a\in S^{m-|\beta|}_\varphi(U)$.
\end{prop}

\begin{proof}
Assume that $a=a_0\varphi_0$ for $a_0\in S^m(U)$. By the product rule, 
$$\partial_x^\alpha\partial_\xi^\beta a=\sum_{\alpha'\leq \alpha, \, \beta'\leq \beta} \partial_x^{\alpha'}\partial_\xi^{\beta'} a_0\partial_x^{\alpha-\alpha'}\partial_\xi^{\beta-\beta'} \varphi_0=\sum_{ \beta'\leq \beta} \partial_x^{\alpha}\partial_\xi^{\beta'} a_0\partial_\xi^{\beta-\beta'} \varphi_0.$$
The proposition follows using Remark \ref{higherdsecond}.
\end{proof}

\begin{prop}
\label{coordinatecha}
Let $\kappa:U\to U'$ be a diffeomorphism of two open sets and $\kappa^*:C^\infty_c(U')\to C^\infty_c(U)$ the associated pull back operator. Consider a pseudo-differential operator $T\in L^m_\varphi(U)$ with $T\sim Op(a)$ for a $\varphi$-symbol $a=a(x,\xi)\in S^m_\varphi(U)$. Then the operator $T_\kappa:=(\kappa^{-1})^*T\kappa^*:C^\infty_c(U')\to C^\infty(U')$ is a pseudo-differential operator of order $m$ on $U'$ and $T_\kappa\sim Op(a_\kappa)$ where 
\begin{equation}
\label{coordtransf}
a_\kappa(y,\eta)\sim \sum_\alpha \frac{1}{\alpha!}[\partial^\alpha_\xi a](\kappa^{-1}(y),(D\kappa)^T(\kappa^{-1}(y))\eta)D^\alpha_z\mathrm{e}^{i\langle\rho_y(z),\eta\rangle}|_{\kappa(z)=y},
\end{equation}
where $\rho_y(z):= \kappa(z)-y-D\kappa(\kappa^{-1}(y))(z-\kappa^{-1}(y))$. In particular, if $\varphi$ has smooth regular variation, then $a_\kappa\in S^m_\varphi(U')$ whenever $a\in S^m_\varphi(U)$ and $T_\kappa\in L^m_\varphi(U')$ whenever $T\in L^m_\varphi(U)$.
\end{prop}

The first statement of Proposition \ref{coordinatecha} and the formula in Equation \eqref{coordtransf} follows immediately from \cite[Theorem 18.1.17]{hormanderIII}. The statement about $L^m_\varphi(U)$ being coordinate invariantly defined follows from Equation \eqref{coordtransf}, Proposition \ref{closedunderder} and the asymptotic completeness of $S^m_\varphi(U)$ (see Proposition \ref{asymptcomplete}).

\subsection{$\varphi$-pseudo-differential operators on closed manifolds}

We now turn to closed manifolds.

\begin{dfn}
\label{lmpsi}
Let $\varphi$ have smooth regular variation, $m\in \mathbb{R}$ and $M$ a closed manifold. Define $L^m_\varphi(M)$ as the space of linear operators $T:C^\infty(M)\to C^\infty(M)$ such that for any $\chi,\chi'\in C^\infty(M)$ being compactly supported in an arbitrary coordinate chart $\kappa:U\to U'\subseteq \R^d$, there is an operator $A\in L^m_\varphi(\R^d)$ such that $\kappa^*\circ (\chi T\chi')\circ (\kappa^{-1})^*\sim A$.
\end{dfn}

For a closed manifold $M$, we can define the symbols $S^m(M)\subseteq C^\infty(T^*M)$ on $M$ of order $m$ to consist of those functions $a\in C^\infty(T^*M)$ such that in any coordinate chart $\kappa:U\to U'\subseteq \R^d$, $(\kappa^{-1})^*a\in S^m(U')$. Choose a metric $g$ on $M$ and define $\varphi_g\in C^\infty(T^*M)$ by
$$\varphi_g(x,\xi):=\varphi((1+|\xi|_g^2)^{d/2}).$$ 
If $\varphi$ has smooth regular variation then $\varphi_g\in S^0(M)$. We define $S^m_\varphi(M)\subseteq C^\infty(T^*M)$ as in the local case by 
\begin{equation}
\label{defnofsmphionm}
S^m_\varphi(M):=\varphi_g S^m(M)+S^{-\infty}(M).
\end{equation}

\begin{prop}
Let $\varphi$ be a function with smooth regular variation, $m\in \mathbb{R}$ and $M$ a closed manifold. The space $S^m_\varphi(M)$ is independent of choice of metric $g$. 
\end{prop}

\begin{proof}
The proposition is an immediate corollary of Proposition \ref{consofsigmaforquad}. Indeed, Proposition \ref{consofsigmaforquad} implies that $a\in S^m_\varphi(M)$ if and only if for any coordinate chart $\kappa:U\to U'\subseteq \R^d$, $(\kappa^{-1})^*a\in S^m_\varphi(U')$. 

\end{proof}

\begin{dfn}
Let $\varphi$ be a function having smooth regular variation, $m\in \mathbb{R}$ and $M$ a closed manifold. Fix a covering $(U_j)_{j=1}^N$ of $M$, coordinate charts $\kappa_j:U_j\to \R^d$ and a partition of unity $(\chi_j)_{j=1}^N$ subordinate to  $(U_j)_{j=1}^N$ (i.e. $\sum\chi_j^2=1$ and $\chi_j\in C^\infty_c(U_j)$). We define $Op:S^m_\varphi(M)\to L^m_\varphi(M)$ by
$$Op(a)f:=\sum_j \chi_j\kappa_j^*Op(\kappa_j^*a)[(\kappa_j^{-1})^*(\chi_j f)],$$ 
for $a\in S^0_\varphi(M)$. Here $\kappa_j^*a$ is supported in $\R^d$, and $Op(\kappa_j^*a)$ is defined as in~\eqref{opdef}.
\end{dfn}

\begin{lem}
\label{philapuptolot}
If $\varphi$ has smooth regular variation, the mapping $Op$ is well defined and modulo lower order terms independent of choices. In particular, 
$$Op(\varphi_g)-\varphi((1-\Delta_g)^{d/2})\in L^{-1}_\varphi(M).$$
\end{lem}

We note that it is not necessary for $M$ to be closed for this lemma to hold, although one would need to use a locally finite countable covering of coordinate charts rather than a finite one.

\begin{proof}
By Proposition \ref{coordinatecha}, each summand $\chi_j\kappa_j^*\circ Op(\kappa_j^*a)(\kappa_j^{-1})^*\circ \chi_j$ in $Op(a)$ belongs to $L^m_\varphi(M)$. A partition of unity argument and the product formula \eqref{prodformpsi} shows that the property of $Op$ being independent of choices up to lower order terms is local. In local coordinates, it is clear that $Op$ is independent of choices up to lower order terms from the coordinate covariance up to lower order terms of the symbol, see Proposition \ref{coordinatecha}. 

It remains to prove that $Op(\varphi_g)-\varphi((1-\Delta_g)^{d/2})\in L^{-1}_\varphi(M)$. It follows from \cite[Theorem 1.3, Chapter XII, Section 1]{Taylorpsidobook} that $\varphi((1-\Delta_g)^{d/2})\in L^0(M)$. In fact, the argument in \cite{Taylorpsidobook} shows that the symbol $q=q(x,\xi)$ of $\varphi((1-\Delta_g)^{d/2})$ coincides with $\varphi_g$ up to a term from $S^{-1}_\varphi(M)$. This shows that $q-\varphi_g\in S^{-1}_\varphi(M)$ and as such the argument in the paragraph above shows that  
$$Op(\varphi_g)-\varphi((1-\Delta_g)^{d/2})=Op(\varphi_g)-Op(q)+Op(q)-\varphi((1-\Delta_g)^{d/2})\in L^{-1}_\varphi(M)+L^{-\infty}(M).$$

\end{proof}

\begin{thm}
\label{factpsps}
Let $\varphi$, $\varphi_1$ and $\varphi_2$ be functions of smooth regular variation. The filtration of spaces $(L^m_\varphi(M))_{m\in \mathbb{R}}$ satisfies
\begin{enumerate}
\item[a)] For any $m\in \R$, $L^m_\varphi(M)\subseteq L^{m+m_\varphi d}(M)$.
\item[b)] The composition of operators on $C^\infty(M)$ defines a product $L^m_{\varphi_1}(M)\times L^{m'}_{\varphi_2}(M)\to L^{m+m'}_{\varphi_1\varphi_2}(M)$. 
\end{enumerate}
Furthermore, if $g$ is a metric on $M$ and $T\in L^{m+m_\varphi d}(M)$, the following statements are equivalent:
\begin{enumerate}
\item[(i)] It holds that $T\in L^m_\varphi(M)$.
\item[(ii)] There exists a $T_0\in L^m(M)$ such that 
$$T-T_0\varphi((1-\Delta_g)^{d/2})\in L^{m-1}_\varphi(M).$$ 
\item[(iii)] For any coordinate chart $\kappa:U\to U'\subseteq \R^d$, and $\chi, \chi'\in C^\infty_c(U)$, there is an operator $T_{00}\in L^m(U')$ such that 
$$\kappa^*\circ (\chi T\chi')\circ (\kappa^{-1})^*-T_{00}Op(\varphi_0)\in L^{m-1}_\varphi(U').$$
\end{enumerate}
In fact, $L^m_\varphi(M)=L^m(M)\varphi((1-\Delta_g)^{-d/2})=\varphi((1-\Delta_g)^{-d/2})L^m(M)$.
\end{thm}

\begin{proof}
Part (a) follows directly from the fact that $S_\varphi^m(M)\subseteq S^m(M)$. Part b) is seen directly from the product formula \eqref{prodformpsi} and from the fact that the function $\varphi_1\varphi_2$ have smooth regular variation.

Note that (ii), is by Lemma \ref{philapuptolot} and (b) equivalent to there existing $T_0$ with $T-T_0Op(\varphi_g)\in L^{m-1}_\varphi(M)$. We further note that (iii) is by Proposition \ref{consofsigmaforquad} equivalent to there for any metric $g'$ and cutoffs $\chi$, $\chi'$ in a coordinate chart existing a $T_{00}$ such that $\kappa^*\circ (\chi T\chi')\circ (\kappa^{-1})^*-T_{00}Op(\varphi_{g'})\in L^{m-1}_\varphi(U')$. As such, the equivalence of the statements (i)--(iii) follows from that $Op$ is local, canonically defined and multiplicative\footnote{Cf. the product formula \eqref{productform}.} up to lower order terms combined with the definition \eqref{defnofsmphionm} (see also Definition \ref{defnofsmphi}). The final statement of the theorem follows from asymptotic completeness of $L^m_\varphi$ and 2).
\end{proof}

\begin{cor}
\label{corofphidifflem}
Let $(M,g)$ be a $d$-dimensional complete Riemannian manifold and $\varphi$ a decreasing function with smooth regular variation. Then $\chi\varphi((1-\Delta_g)^{d/2})\in \Lw(L_2(M))$ for any $\chi\in C^\infty_c(M)$.
\end{cor}

This proves Lemma \ref{psiideallem} and finalizes the proof of Proposition \ref{psiidealprop}.

\begin{proof}
We fix $\chi$ and assume it is positive. It is clear that $\chi\varphi((1-\Delta_g)^{d/2})\in \Lw$ if and only if $\chi\varphi((1-\Delta_g)^{d/2})^2\chi\in \mathcal{L}_{\varphi^2}$. We note that $\varphi^2$ has smooth regular variation whenever $\varphi$ does. We pick a closed Riemannian manifold $(M',g')$ such that for a smooth domain $U\subseteq M'$ there is a smooth isometric mapping $\kappa:U\to M$ which is a diffeomorphism onto its image and assumed to contain the support of $\chi$.

By Lemma \ref{philapuptolot}, we have that $\chi\varphi((1-\Delta_g)^{d/2})^2\chi$ is a compactly supported element of $L^0_{\varphi^2}(M)$. By using Theorem \ref{factpsps}, an induction step shows that for any $N$ there is a $T_N\in 1+L^{-1}_{\varphi^2}(M')$ compactly supported in $U$ and $\chi'\in C^\infty_c(U)$ such that 
$$\chi\varphi((1-\Delta_g)^{d/2})^2\chi-(\kappa^{-1})^*\circ (\chi'T_N\varphi((1-\Delta_{g'})^{d/2}))\circ \kappa^* \circ [\chi'\circ \kappa^{-1}]\in L^{-N}_{\varphi^2}(M).$$
Since all operators have compact support, the Weyl law on the compact manifold $M'$ guarantees that 
$$\mu(n,\chi\varphi((1-\Delta_g)^{d/2})^2\chi)= c_{d,\chi} \varphi(n)^2+O(n^{-1/d}\varphi(n)^2),$$ 
for a suitable constant $c_{d,\chi}$. It follows that $\chi\varphi((1-\Delta_g)^{d/2})^2\chi\in \mathcal{L}_{\varphi^2}$ and $\chi\varphi((1-\Delta_g)^{d/2})\in \Lw$.

\end{proof}

\begin{cor}
If $\varphi$ is a decreasing function with smooth regular variation, the representation $L^0(M)\to \mathbb{B}(L_2(M))$ restricts to a representation $L^0_\varphi(M)\to \mathcal{L}_\varphi(L_2(M))$.
\end{cor}

\begin{proof}
Pick a metric $g$ on $M$. Using induction, we can by Theorem \ref{factpsps} write any $T\in L^0_\varphi(M)$ in the form $T=T_N\varphi((1-\Delta_g)^{d/2})+L^{-N}_\varphi(M)$ for a $T_N\in L^0(M)$. Since $M$ is closed, $L^0(M)$ acts as bounded operators and the Weyl law for $\Delta_g$ guarantees that 
$$\mu(n,T)\leq \|T_N\|_{\mathbb{B}(L_2(M))}\varphi(n)+o(\varphi(n))=O(\varphi(n)).$$

\end{proof}

\section{Localizing $\varphi$-modulated operators on manifolds}
\label{sec:localizing}

In this section we will use the theory of the previous section to study more general operators on $L_2(M)$ for a closed manifold $M$. We are interested in computing Dixmier traces, and can therefore to a large extent work modulo the kernel of all traces. A key tool is to localize to operators that propagate supports small distances. Along the way we prove Lemma \ref{incluofhs} which allowed us to localize the property of being weakly $\varphi$-Laplacian modulated in an inner product space. To extend the computations for $\varphi$-Laplacian modulated operators from Section \ref{cttsection} to a general manifold, we localize our operators to coordinate neighbourhoods.

\begin{dfn}
Let $X$ be a proper metric space and $\mathcal{H}$ a Hilbert space with an action of $C_0(X)$. An operator $G\in \mathbb{B}(\mathcal{H})$ is said to have \emph{propagation speed} $\epsilon$, if $\chi G\chi'=0$ whenever $\chi,\chi'\in C_c(X)$ satisfy $\mathrm{d}(\mathrm{supp}(\chi),\mathrm{supp}(\chi'))> \epsilon$.

For $\epsilon>0$, the subspace $\mathcal{L}_{2,\epsilon}(\mathcal{H})\subseteq \mathcal{L}_{2}(\mathcal{H})$ is defined as the Hilbert-Schmidt operators of propagation speed $\epsilon$.
\end{dfn}

\begin{rem}
We note that $\mathcal{L}_{2,\epsilon}(\mathcal{H})$ is not closed under multiplication, but $\mathcal{L}_{2,\epsilon}(\mathcal{H})\mathcal{L}_{2,\epsilon}(\mathcal{H})\subseteq \mathcal{L}_{2,2\epsilon}(\mathcal{H})$.
\end{rem}

The situation of interest in this paper is $L_2(M)$ for a closed manifold $M$. 

\begin{dfn}
\label{coveringdata}
For a closed manifold $M$, \emph{covering data} is a collection
$$\mathcal{U}:=(U_j,\kappa_j,\chi_{j})_{j=1}^N.$$
where $(U_j)_{j=1}^N$ is a cover of $M$ with coordinate charts $\kappa_j:U_j\to U_j'\subseteq \R^d$ and $(\chi_j)_{j=1}^N$ a subordinate smooth subordinate partition of unity. 

If $M$ is Riemannian, and $\mathrm{diam}(U_j)<\epsilon/2$ we say that $\mathcal{U}$ is a set of $\epsilon$-covering data for $M$.
\end{dfn}

It is easily seen that we can find $\epsilon$-covering data for $M$ for any $\epsilon$ strictly smaller than half the injectivity radius of $M$. We tacitly assume that $\epsilon$ is small enough for $\epsilon$-covering data for $M$ to exist.

Assuming that $\kappa_j$ extends smoothly to a neighborhood of $\overline{U}_j$, we obtain a bounded invertible operator $\kappa_j^*:L_2(U_j')\to L_2(U_j)$. Let $Z_j:L_2(U_j')\to L_2(U_j)$ denote the unitary associated with $\kappa_j^*$ through polar decomposition, in fact $Z_j^*\kappa_j^*$ and $\kappa_j^*Z_j^*$ are multiplication operators. Moreover, $Z_j(af)=\kappa_j^*(a)Z_j(f)$ for $f\in L_2(U_j')$ and $a\in C_c(U_j')$. We define the following operators
\begin{align}
\nonumber
u_{\mathcal{U},0}:L_2(M)&\to \bigoplus_{j=1}^N L_2(U_j), \quad f\mapsto (\chi_j f)_{j=1}^N,\\
\nonumber
Z:=\oplus_{j=1}^NZ_j:\bigoplus_{j=1}^N L_2(U_j')&\to \bigoplus_{j=1}^N L_2(U_j), \quad\mbox{and}\\
\label{uudef}
u_{\mathcal{U}}:=Z^*u_{\mathcal{U},0}:L_2(M)&\to \bigoplus_{j=1}^N L_2(U_j')\subseteq \bigoplus_{j=1}^N L_2(\R^d).
\end{align}
The operators $u_{\mathcal{U},0}$ and $u_{\mathcal{U}}$ are isometries and $Z$ is unitary. We note that for an operator $G\in \mathbbm{B}(L_2(M))$, 
$$u_{\mathcal{U},0}Gu_{\mathcal{U},0}^*=(G_{jk})_{j,k=1}^N,$$
where $G_{jk}:=\chi_j G\chi_k:L_2(U_k)\to L_2(U_j)$. As such 
$$u_{\mathcal{U},0}^*(G_{jk})_{j,k=1}^N u_{\mathcal{U},0}=\sum_{j,k} \chi_jG_{jk}\chi_k=\sum_{j,k} \chi_j^2G\chi_k^2=G.$$

\begin{prop}
\label{localizprop}
Let $\varphi\in R_{-1}$ be decreasing and $\mathcal{U}$ be $\epsilon$-covering data for $M$ as in Definition \ref{coveringdata}. Define the \emph{localization mapping} $\ell_\mathcal{U}:\mathcal{L}_2(L_2(M))\to \mathcal{L}_{2,\epsilon}(L_2(M))$ by 
$$\ell_\mathcal{U}(G):=\sum_{j}G_{jj}, \quad \mbox{where} \quad G_{jj}:=\chi_jG\chi_j.$$
It holds that $\ell_\mathcal{U}:\Lw(L_2(M))\to \Lw(L_2(M))\cap \mathcal{L}_{2,\epsilon}(L_2(M))$ and 
$$\mathrm{Tr}_\omega(G)=\mathrm{Tr}_\omega(\ell_\mathcal{U}(G)),$$
for any $G\in \Lw$ and any Dixmier trace. In particular, up to terms of vanishing Dixmier traces, any element of $\Lw(L_2(M))$ is a finite sum of operators with propagation $\epsilon$, compactly supported in a coordinate chart and belonging to $\Lw(L_2(M))$.
\end{prop}

\begin{proof}
The ideal property for Hilbert-Schmidt operators and the fact that the diameter of $U_j$ is bounded by $\epsilon$ implies that $G_{jj}\in \mathcal{L}_{2,\epsilon}(L_2(M))$ for any $j$. The ideal property for $\Lw$ implies that $G_{jj}\in \Lw$ if $G\in \Lw$. Finally, we compute that 
$$\mathrm{Tr}_\omega(G)=\mathrm{Tr}_\omega(u_{\mathcal{U},0}Gu_{\mathcal{U},0}^*)=\sum_{j=1}^N \mathrm{Tr}_\omega(G_{jj})=\mathrm{Tr}_\omega(\ell_\mathcal{U}(G)).$$
\end{proof}

\begin{dfn}
\label{locallyweakly}
Let $G\in \mathbbm{B}(L_2(M))$ be an operator. We say that $G$ is {\em locally} strongly (or spectrally, or weakly) $\varphi$-Laplacian modulated if for any coordinate chart $\kappa:U\to U'\subseteq \R^d$, and $\chi, \chi'\in C^\infty_c(U)$, the compactly supported operator $\kappa^*\circ (\chi G\chi')\circ (\kappa^{-1})^*$ on $L_2(\R^d)$ is strongly (or spectrally, or weakly) $\varphi$-Laplacian modulated.
\end{dfn}

The following is an immediate consequence of the definition of local modulation.

\begin{prop}
\label{decomposingt}
Let $\varphi$ be a decreasing function with smooth regular variation and $\mathcal{U}$ be $\epsilon$-covering data for $M$ as in Definition \ref{coveringdata}. Suppose that $G\in \mathbbm{B}(L_2(M))$ is locally strongly (or spectrally, or weakly) $\varphi$-Laplacian modulated. Then each of the operators 
$$G_{jk}:=\chi_jG\chi_k,$$
are locally strongly (or spectrally, or weakly) $\varphi$-Laplacian modulated. In particular, for any locally strongly (or spectrally, or weakly) $\varphi$-Laplacian modulated operator $G$, the localization $\ell_\mathcal{U}(G)$ is a finite sum of locally strongly (or spectrally, or weakly) $\varphi$-Laplacian modulated operators with propagation $\epsilon$, compactly supported in a coordinate chart and belongs to $\Lw(L_2(M))$.
\end{prop}

\begin{proof}
It follows from the construction that each $G_{jk}$ is locally strongly  (or spectrally, or weakly) $\varphi$-Laplacian modulated. The last statement follows from Proposition \ref{localizprop} and the fact that compactly supported strongly (or spectrally, or weakly) $\varphi$-Laplacian modulated operators on $\R^d$ belong to $\Lw(L_2(M))$ by Lemma \ref{l_exp}, Lemma \ref{VAmod}, and Lemma \ref{mod_w-mod} (see Lemma \ref{l_exp} and Corollary \ref{corofphidifflem} for the weakly modulated case).
\end{proof}

\subsection{$\varphi$-Sobolev spaces}

An important tool will be that of \emph{$\varphi$-Sobolev spaces}. These spaces were introduced in Definition \ref{sobsparn} for inner product spaces. We shall define their analogues for manifolds.

\begin{dfn}
Let $M$ be a closed $d$-dimensional Riemannian manifold with metric $g$. For $s\in \R$, we define $H^s_\varphi(M)$ as the space of distributions $f\in \mathcal{D}'(M)$ such that for any coordinate chart $\kappa:U\to U'\subseteq \R^d$, and $\chi\in C^\infty_c(U)$, $(\kappa^{-1})^*(\chi f)\in H^s_\varphi(\R^d)$.

We topologize $H^s_\varphi(M)$ as a Banach space by picking covering data $\mathcal{U}$ (as in Definition \ref{coveringdata}) and declaring the operator $u_{\mathcal{U}}$ from Equation \eqref{uudef} to be an isometry
$$u_{\mathcal{U}}:H^s_\varphi(M)\to \bigoplus_{j=1}^N H^s_\varphi(\R^d).$$
\end{dfn}

The careful reader notes that $u_{\mathcal{U}}$ is well defined while for any $f$, $u_{\mathcal{U}}(f)=(Z_j^*(\chi_jf))_{j=1}^N$ and by definition, $Z_j^*(\chi_jf)\in H^s_{\varphi}(\R^d)$ whenever $f\in H^s_\varphi(M)$. It is at this point not immediate that $H^s_\varphi(M)$ is well behaved. The next theorem gives a coordinate free definition of $H^s_\varphi(M)$ inducing a Hilbert space structure related to the ``$\varphi$-Laplacian" $\varphi((1-\Delta_g)^{d/2})$.

\begin{thm}
\label{psisobol}
Let $\varphi$ be a smoothly regularly varying function and $M$ a closed manifold. The space $H^s_\varphi(M)$ is well defined and 
$$H^s_\varphi(M)=\varphi((1-\Delta_g)^{d/2})^{s/d}L_2(M).$$
\end{thm}

\begin{rem}
The space $H^s_\varphi(M)$ is by Theorem \ref{psisobol} a Hilbert space in the inner product
$$\langle f_1,f_2\rangle_{H^s_\varphi(M)}:=\langle \varphi((1-\Delta_g)^{d/2})^{-s/d}f_1,\varphi((1-\Delta_g)^{d/2})^{-s/d}f_2\rangle_{L_2(M)}.$$
We also note that 
$$H^s_{\frac{1}{\varphi}}(M)=H^{-s}_\varphi(M).$$
\end{rem}

To prove Theorem \ref{psisobol}, we require some lemmas.

\begin{lem}
\label{actingwithlzero}
Let $M$ be a closed manifold and $\varphi$ be a function of smooth regular variation. The action of $L^0(M)$ on $L_2(M)$ defines a continuous action of $L^0(M)$ on $H^s_\varphi(M)$ for all $s\in \R$. 
\end{lem}

\begin{proof}
By a partition of unity argument, and Theorem \ref{factpsps}, for $k\in d\N$, $f\in H^k_\varphi(M)$ if and only if there is an $f_0\in L_2(M)$ such that $f-Op(\varphi_g)^{k/d}f_0\in H^{k-d}_\varphi(M)$. By an induction argument over $k\in d\N$, and the fact that $[A,B]\in L^{m-1}_\varphi(M)$ for $A\in L^0(M)$ and $B\in L^m_\varphi(M)$, by Proposition \ref{products} and the product formula \eqref{prodformpsi}, we see that the action of $L^0(M)$ on $H^k_\varphi(M)$ is well defined and continuous for $k\in d\N$. 

The isometry $u_{\mathcal{U}}:f \mapsto \left(Z_j^*(\chi_jf)\right)_{j=1}^N$ is split by the mapping $u^*:(f_j)_{j=1}^N\mapsto \sum_{j=1}^N \chi_jZ_j(f_j)$. We can deduce that $(H^s_\varphi(M))^*=H^{-s}_\varphi(M)$, with equivalent norms, via the $L_2$-pairing. Moreover, we can for $s_0,s_1\in \R$ do complex interpolation 
$$[H^{s_0}_\varphi(M),H^{s_1}_\varphi(M)]_{\theta}=H^{s_\theta}_\varphi(M), \quad s_\theta=(1-\theta)s_0+\theta s_1.$$
It now follows by interpolation and duality that the action of $L^0(M)$ on $H^s_\varphi(M)$ is well defined and continuous for $s\in \R$.
\end{proof}

\begin{lem}
\label{comparingpsis}
Let $M$ be a complete manifold and $\varphi$ have smooth regular variation. For any two metrics $g_1,g_2$ on $M$ there is an elliptic operator $\Lambda_{g_1,g_2}\in L^0(M)$ such that $\Lambda_{g_1,g_2}\varphi((1-\Delta_{g_1})^{d/2})-\varphi((1-\Delta_{g_2})^{d/2})$ is a smoothing operator.
\end{lem}

\begin{proof}
Consider the elliptic symbol $\sigma_{g_1,g_2}\in S^0(M)$ defined from Equation \eqref{sigmag1g2}, it is elliptic by Proposition \ref{consofsigmaforquad}. The product formula \eqref{productform} implies that $B_0:=Op(\sigma_{g_1,g_2})\in L^0(M)$ satisfies that $A_0:=B_0\varphi((1-\Delta_{g_1})^{d/2})-\varphi((1-\Delta_{g_2})^{d/2})\in L^{-1}_\varphi(M)$. Using Theorem \ref{factpsps}, we can find a $C_0\in L^{-1}(M)$ such that $A_0-C_0\varphi((1-\Delta_{g_2})^{d/2})\in L^{-2}_\varphi(M)$. Let $Q_0$ be a parametrix to $1+C_0$ and $B_1:=Q_0B_0$. We note that $B_1-B_0\in L^{-1}(M)$ and $A_1:=B_1\varphi((1-\Delta_{g_1})^{d/2})-\varphi((1-\Delta_{g_2})^{d/2})\in L^{-2}_\varphi(M)$. 

Again using Theorem \ref{factpsps}, we can find a $C_1\in L^{-2}(M)$ such that $A_1-C_1\varphi((1-\Delta_{g_2})^{d/2})\in L^{-3}_\varphi(M)$. Let $Q_1$ be a parametrix to $1+C_1$ and $B_2:=Q_1B_1$. We note that $B_2-B_1\in L^{-2}(M)$ and $A_2:=B_2\varphi((1-\Delta_{g_1})^{d/2})-\varphi((1-\Delta_{g_2})^{d/2})\in L^{-3}_\varphi(M)$. 

Proceeding by induction, we find elliptic $B_0,\ldots, B_N\in L^0(M)$ such that $B_k-B_{k+1}\in L^{-k-1}(M)$ and $A_N:=B_N\varphi((1-\Delta_{g_1})^{d/2})-\varphi((1-\Delta_{g_2})^{d/2})\in L^{-N-1}_\varphi(M)$. By asymptotic completeness, we can define the elliptic operator $\Lambda_{g_1,g_2}\sim B_0+\sum_{k=0}^\infty (B_{k+1}-B_k)\in L^0(M)$; this operator will satisfy that $\Lambda_{g_1,g_2}\varphi((1-\Delta_{g_1})^{d/2})-\varphi((1-\Delta_{g_2})^{d/2})$ is a smoothing operator.
\end{proof}

\begin{proof}[Proof of Theorem \ref{psisobol}]
We need only to consider the case $s=d$, the other cases follow from induction using Lemma \ref{actingwithlzero}, interpolation and duality. We need only to prove that $f\in \varphi((1-\Delta_g)^{d/2})L_2(M)$ if and only if $f\in H^d_\varphi(M)$ for $f$ supported in one coordinate chart. In this case, we can clearly reduce to showing that $f\in \varphi((1-\Delta_g)^{d/2})L_2(M)$ if and only if $f\in \varphi((1-\Delta_{\tilde{g}})^{d/2})L_2(M)$ for a metric $\tilde{g}$ which is the pullback of the Euclidean metric to the support of $f$. But $f\in \varphi((1-\Delta_g)^{d/2})L_2(M)$ is equivalent to $f\in \varphi((1-\Delta_{\tilde{g}})^{d/2})L_2(M)$ by elliptic regularity and Lemma \ref{comparingpsis}.
\end{proof}

\begin{cor}
\label{coronmod}
Let $\varphi$ be a decreasing function of smooth regular variation and pick $\epsilon$-covering data $\mathcal{U}$ of $M$ (see Definition \ref{coveringdata}). For $G\in \mathcal{L}_2(L_2(M))$, let $\ell_\mathcal{U}(G)$ be its localization as in Proposition \ref{localizprop}. Then the following folds
\begin{enumerate}
\item $\ell_\mathcal{U}(G)$ is locally weakly $\varphi$-Laplacian modulated if and only if $\ell_\mathcal{U}(G)\in \Lw$ is weakly $\varphi$-modulated with respect to $\varphi((1-\Delta_g)^{d/2})$.
\item If $G$ is weakly $\varphi$-modulated with respect to $\varphi((1-\Delta_g)^{d/2})$, then $G$ is locally weakly $\varphi$-Laplacian modulated. 
\item If $G$ is locally weakly $\varphi$-Laplacian modulated, then $\ell_\mathcal{U}(G)$ is locally weakly $\varphi$-Laplacian modulated. 
\end{enumerate}
\end{cor}

\begin{proof}
By Proposition \ref{weakregandsob} and Theorem \ref{psisobol}, we can assume that $G$ is compactly supported in $M=\R^d$ and $g$ being the Euclidean metric. By Lemma \ref{comparingpsis}, it is equivalent for compactly supported operators in $\R^d$ to be locally weakly $\varphi$-Laplacian modulated and to be weakly $\varphi$-modulated with respect to $\varphi((1-\Delta)^{d/2})$. This proves 1) and 2). Part 3) was already proven in Proposition \ref{decomposingt}.
\end{proof}

\begin{cor}
Let $\varphi$ be a decreasing function with smooth regular variation. Consider a closed Riemannian manifold $(M,g)$. The set of strongly $\varphi$-modulated operators with respect to $\varphi((1-\Delta_g)^{d/2})$ is closed under left and right multiplication by $L^0(M)$. Moreover, an operator is strongly $\varphi$-modulated operators with respect to $\varphi((1-\Delta_g)^{d/2})$ if and only if it is strongly $\varphi$-modulated operators with respect to $\varphi((1-\Delta_{g'})^{d/2})$ for any metric $g'$.
\end{cor} 

The proof goes as in \cite[Lemma 11.6.2 and 11.6.2]{LSZ} using Theorem \ref{factpsps} and Lemma \ref{comparingpsis}.

We now turn to proving Lemma \ref{incluofhs} by means of a corollary to Theorem \ref{psisobol}. Recall the definition of $ \tilde{H}^s_\varphi(\mathbb{T}^d)$ from Definition \ref{sobsparn}. By Theorem \ref{psisobol},  $\tilde{H}^s_\varphi(\mathbb{T}^d)=H^s_\varphi(\mathbb{T}^d)$ and the following corollary holds.

\begin{cor}
\label{compsobspaceoncube}
Let $\chi\in C^\infty_c((0,1)^d)$. Via the inclusion $(0,1)^d\hookrightarrow \mathbb{T}^d:=\R^/\Z^d$ of the fundamental domain and $(0,1)^d\hookrightarrow \R^d$, we have the equality
$$\chi \tilde{H}^s_\varphi(\mathbb{T}^d)=\chi H^s_\varphi(\R^d).$$
\end{cor}

\subsection{Symbols and symbol properties on arbitrary closed manifolds}

In this subsection, we will define abstract symbols of general Hilbert-Schmidt operators on closed manifolds. The construction of an abstract symbol will depend on covering data (see Definition \ref{coveringdata}) and say little about the operator except for its Dixmier traces when it is weakly $\varphi$-Laplacian modulated. Connes' trace formula on closed manifolds will be the means through which the Dixmier trace computations are done, and the abstract symbol the main player. For operators of pseudo-differential type, the abstract symbol and the pseudo-differential symbol provide the same formulas for Dixmier traces.

For a closed Riemannian manifold $M$, we denote
$$\mathrm{Diag}_\epsilon:=\{(x,y)\in M\times M: \mathrm{d}(x,y)<\epsilon\}.$$
For $\epsilon>0$ small enough, $\mathrm{Diag}_\epsilon$ is a tubular neighborhood of the diagonal in $M\times M$ and we can therefore choose a diffeomorphism $\phi:\mathrm{Diag}_{2\epsilon}\to TM$ such that $\pi_{TM}\circ \phi(x,y)=x$. Here $\pi_{TM}:TM\to M$ denotes the tangent bundle. Under the identification of $\mathcal{L}_2(L_2(M))$ with $L_2(M\times M)$, the space $\mathcal{L}_{2,\epsilon}(L_2(M))$ corresponds to $L_2(\mathrm{Diag}_\epsilon)$. 

\begin{dfn}
\label{abssymbol}
The abstract symbol $\sigma_{\epsilon,\phi}:\mathcal{L}_{2,\epsilon}(L_2(M))\to L_2(T^*M)$ is defined as the composition 
$$\mathcal{L}_{2,\epsilon}(L_2(M))\cong L_2(\mathrm{Diag}_\epsilon)\xrightarrow{(\phi^{-1})^*}L_2(TM)\xrightarrow{\mathcal{F}} L_2(T^*M),$$
where $\mathcal{F}:L_2(TM)\xrightarrow{\sim} L_2(T^*M)$ is the fiberwise Fourier transform.

For a choice of covering data $\mathcal{U}$ (see Definition \ref{coveringdata}), we define the localized symbol 
$$\sigma_{\epsilon,\phi,\mathcal{U}}:\mathcal{L}_{2}(L_2(M))\to  L_2(T^*M), \quad \sigma_{\epsilon,\phi,\mathcal{U}}:=\sigma_{\epsilon,\phi}\circ \ell_{\mathcal{U}}.$$
\end{dfn}

\begin{rem}
The reader should note that we are using a diffeomorphism $\phi:\mathrm{Diag}_{2\epsilon}\to TM$ since we do not wish to concern ourselves with the boundary behaviour at the boundary of the tubular neighborhood. This choice will not affect values of traces that only depend on the behaviour near the diagonal. This obstructs the abstract symbol  $\sigma_{\epsilon,\phi}:\mathcal{L}_{2,\epsilon}(L_2(M))\to L_2(T^*M)$ being surjective, but it is nevertheless injective.
\end{rem}

\begin{prop}
For $G\in \Lw(L_2(M))$, the value of $\mathrm{Tr}_\omega(G)$ depends only on the localized abstract symbol $\sigma_{\epsilon,\phi,\mathcal{U}}(G)\in L_2(T^*M)$.
\end{prop}

\begin{proof}
This fact follows trivially from Proposition \ref{localizprop} because $\sigma_{\epsilon,\phi}$ is injective.
\end{proof}

For $G\in \mathcal{L}_2(L_2(M))$, and a choice of auxiliary data $(\epsilon,\phi,\mathcal{U})$ as above, we often write 
\begin{equation}
p_G:=\sigma_{\epsilon,\phi,\mathcal{U}}(G).
\end{equation}
We call $p_G$ \emph{a localized abstract symbol} of $G$ when we wish to suppress the dependence on the auxiliary data. The arbitrariness in the construction of a localized abstract symbol is less so, yet still present, for $\varphi$-pseudo-differential operators as the next proposition shows. 

\begin{prop}
Let $\varphi$ be a decreasing function of smooth regular variation, $M$ a closed manifold, $(\epsilon,\phi)$ as in the paragraph proceeding Definition \ref{abssymbol} and $\mathcal{U}$ covering data. If $G\in L^0_\varphi(M)$ has pseudo-differential symbol $p\in S^0_\varphi(M)$, then $\sigma_{\epsilon,\phi,\mathcal{U}}(G)\in S^0_\varphi(M)$ and
$$p-\sigma_{\epsilon,\phi,\mathcal{U}}(G)\in S^{-1}_\varphi(M).$$
In particular, $\mathrm{Tr}_\omega(G)$ depends only on the pseudo-differential symbol $[p]\in S^0_\varphi(M)/S^{-1}_\varphi(M)$.
\end{prop}

\begin{proof}
The product formula \eqref{prodformpsi} implies that $G-\ell_\mathcal{U}(G)\in L^{-1}_\varphi(M)$. Therefore, we have that $\ell_\mathcal{U}(G)-Op(p)\in L^{-1}_\varphi(M)$. Proposition \ref{philapuptolot} and injectivity of $\sigma_{\epsilon,\phi}$ implies that $p-\sigma_{\epsilon,\phi,\mathcal{U}}(G)\in S^{-1}_\varphi(M)$. Since $p\in S^0_\varphi(M)$ we have $\sigma_{\epsilon,\phi,\mathcal{U}}(G)\in S^{0}_\varphi(M)$. We have that $\mathrm{Tr}_\omega(G)$ only depends on the class $[p]\in S^0_\varphi(M)/S^{-1}_\varphi(M)$ because Dixmier traces vanishes on $L^{-1}_\varphi(M)=L^{0}_\varphi(M)L^{-1}(M)$ as it is a subalgbra of the norm closure of the finite rank operators in $\Lw$.
\end{proof}

Prior to the next definition, note that on a closed Riemannian manifold, the function $\langle t-\langle\xi\rangle^d\rangle^{-1}$ belongs to $L_2(T^*M)$ for any $t>0$ As such, the function $\frac{p_G(x,\xi)}{\langle t-\langle\xi\rangle^d\rangle}$ is integrable for any $p\in L_2(T^*M)$.

\begin{dfn}
\label{phidecmnfd}
Let $M$ be a closed Riemannian manifold. We say that $p\in L_2(T^*M)$ has $\varphi$-reasonable decay if 
$$\int_{T^*M} \frac{|p(x,\xi)|}{\langle t-\langle\xi\rangle^d\rangle}\, \mathrm{d}x \mathrm{d}\xi =o(\Phi(t)), \quad\mbox{as $t\to \infty$}.$$
\end{dfn}

\begin{rem}
It is easily seen that for $p\in L_2(T^*M)$, the property of having $\varphi$-reasonable decay is independent of metric and is equivalent to $(D\kappa^{-1})^*(\pi_{T^*M}^*(\chi) p)\in L_2(\R^{2d})$ having  $\varphi$-reasonable decay (in the sense of Definition \ref{phidec}) for any coordinate chart $\kappa:U\to \R^d$ and $\chi\in C^\infty_c(U)$. 
\end{rem}

\begin{prop}
\label{reasondecayofphipsido}
Let $\varphi\in SR_{-1}$ be a decreasing function satisfyin (W2) (see Definition \ref{someavass} on page \pageref{someavass}) and $M$ a closed Riemannian manifold. Any $p\in S^0_\varphi(M)$ has $\varphi$-reasonable decay. 
\end{prop}

\begin{proof}
By definition, if $p\in S^0_\varphi(M)$ there is a $C>0$ such that $|p(x,\xi)|\leq C\varphi_g(\xi)$. The proposition now follows from Proposition \ref{psi_growth}.
\end{proof}

\begin{prop}
\label{strongandreas}
Let $\varphi\in SR_{-1}$ be a decreasing function, $M$ a closed Riemannian manifold and $G\in \mathbb{B}(L_2(M))$ an operator which is strongly $\varphi$-modulated with respect to $\varphi((1-\Delta_g)^{d/2})$. Then $G$ is locally strongly $\varphi$-Laplacian modulated. 

In particular, if $\varphi\in SR_{-1}$ and $\varphi(t)=O(t^{-1})$ then any localized abstract symbol of an operator which is strongly $\varphi$-modulated with respect to $\varphi((1-\Delta_g)^{d/2})$ has $\varphi$-reasonable decay.
\end{prop}

\begin{proof}
The proof that $G$ is locally strongly $\varphi$-Laplacian modulated whenever it is strongly $\varphi$-modulated with respect to $\varphi((1-\Delta_g)^{d/2})$ goes as in \cite[Propostion 11.6.7]{LSZ} (using Lemma \ref{VAmod} and Theorem \ref{factpsps}) and is omitted. To prove the final statement, fix the auxiliary data $(\epsilon,\phi,\mathcal{U})$ needed to define the localized abstract symbol. We note that $\ell_\mathcal{U}(G)$ is by Proposition \ref{decomposingt} a finite sum of locally strongly $\varphi$-Laplacian modulated operators compactly supported in a coordinate chart. As such, the localized abstract symbol $\sigma_{\epsilon,\phi,\mathcal{U}}(G)$ has $\varphi$-reasonable decay by Lemma \ref{phimodphireas}.

\end{proof}

\section{Connes' trace formula on closed manifolds}
\label{sec:CTTonmanifolds}
Now we prove Connes' trace formula for $\varphi$-Laplacian modulated operators on manifolds.

\begin{thm}
\label{CTT}
Let $\varphi\in SR_{-1}$ be a decreasing function. Assume that $M$ is a $d$-dimensional Riemannian closed manifold. Consider an operator $G\in \Lw(L_2(M))$ which is 
\begin{enumerate}
\item locally weakly $\varphi$-Laplacian modulated (see Definition \ref{locallyweakly}),
\item having a localized abstract symbol $p_G$ with $\varphi$-reasonable decay (see Definition \ref{phidecmnfd}).
\end{enumerate}
Then for every extended limit $\omega$ on $\ell_\infty$ we have
$${\rm Tr}_\omega(G) = \omega \left(\frac1{\Phi (n+1)} \int_{M} \int_{|\xi|\le n^{1/d}} p_G(x,\xi)\,\mathrm{d}\xi \mathrm{d}x\right).$$
\end{thm}

\begin{proof}
We can by Corollary \ref{coronmod} assume that $G$ is supported in a coordinate chart $U$ in which it is defined from an $L_2$-symbol $p_G$ having $\varphi$-reasonable decay. Theorem \ref{CTT_R} implies that 
$${\rm Tr}_\omega(G) = \omega \left(\frac1{\Phi (n+1)} \int_{U} \int_{|\xi|\le n^{1/d}} p_G(x,\xi)\,\mathrm{d}\xi \mathrm{d}x\right)= \omega \left(\frac1{\Phi (n+1)} \int_{M} \int_{|\xi|\le n^{1/d}} p_G(x,\xi)\,\mathrm{d}\xi \mathrm{d}x\right).$$

\end{proof}

\begin{rem}
By Corollary \ref{coronmod}, Theorem \ref{CTT} applies if $G$ is weakly $\varphi$-modulated with respect to $\varphi((1-\Delta_g)^{d/2})$ and has a localized abstract symbol of $\varphi$-reasonable decay. Moreover, the reader is encouraged to recall Lemma \ref{mod_w-mod} stating that strongly $\varphi$-modulated operators are weakly $\varphi$-modulated if $\varphi$ has property (W1).

If $\varphi(t)=O(t^{-1})$ then having symbols of $\varphi$-reasonable decay is automatic for any strongly $\varphi$-modulated operator by Proposition \ref{strongandreas}.  

\end{rem}

\begin{cor}
\label{CTT_psipsido}
Let $\varphi\in SR_{-1}$ be a decreasing function. Assume that $M$ is a $d$-dimensional Riemannian closed manifold. For $G\in L^0_\varphi(M)$, with $\varphi$-symbol $p\in S^0_\varphi(M)$, and for every extended limit $\omega$ on $\ell_\infty$ we have
$${\rm Tr}_\omega(G) = \omega \left(\frac1{\Phi (n+1)} \int_{M} \int_{|\xi|\le n^{1/d}} p(x,\xi)\,\mathrm{d}\xi \mathrm{d}x\right).$$
\end{cor}

\begin{proof}
We note that $G-Op(p)\in L^{-1}_\varphi(M)$ so ${\rm Tr}_\omega(G)= {\rm Tr}_\omega(Op(p))$ because all Dixmier traces vanish on $L^{-1}_\varphi(M)$. To prove the corollary, we can assume that $G=Op(p)$ and in this case verify the assumptions of Theorem \ref{CTT}. It follows from Theorem \ref{factpsps} that any $G\in L^0_\varphi(M)$ is locally weakly $\varphi$-Laplacian modulated, in fact, Theorem \ref{factpsps} shows that any $G\in L^0_\varphi(M)$ is  weakly $\varphi$-modulated with respect to $\varphi((1-\Delta_g)^{d/2})$. By Proposition \ref{reasondecayofphipsido}, any symbol $p\in S^0_\varphi(M)$ has $\varphi$-reasonable decay.
\end{proof}

\subsection{$\log$-classical pseudo-differential operators}

In~\cite{Lesch}, Lesch considered classes of pseudo-differential operators $CL^{m,k}(M)$, $m\in \Z$, $k\in \N$, consisting of log-polyhomogeneous pseudo-differential operators on a $d$-dimensional manifold $M$. These are operators given by~\eqref{opdef} with symbols of the following form:
$$a(x,\xi) \sim \sum_{j=0}^\infty \sum_{i=0}^k a_{m-j,i}(x,\xi) \log^i |\xi|,$$
where functions $a_{m-j,i}$ are homogeneous of degree $m-j$ in the second argument.

The classes $CL^{-d,k}(M)$ are subclasses of $L^0_{\varphi_k}(M)$ for $\varphi_k(t) = \frac{\log^k (\e+t)}{\e+t}$, $k\in \N$. For the classes $CL^{m,k}(M)$, Lesch constructed a Wodzicki-type non-commutative residue ${\rm Res}_k$ and established the following expression for expression for ${\rm Res}_k$ in terms of the symbol of an operator:

If $M$ is a $d$-dimensional manifold and $A\in CL^{-d,k}(M)$, then 
\begin{equation}
\label{Res_k}
\mathrm{Res}_k(A) = \frac{(k+1)!}{(2\pi)^n} \int_{S^*M}  a_{-d,k}(x, \xi) \ \mathrm{d}\xi \mathrm{d}x.
\end{equation}

\begin{thm}
\label{cor_logcl} 
Let $M$ be a compact Riemannian $d$-dimensional manifold and let $A\in CL^{-d,k}(M)$, $k\in \N$, then

(i) $A \in \mL_{\varphi_k}$ for $\varphi_k(t) = \frac{\log^k (\e+t)}{\e+t}$;

(ii) For every Dixmier trace ${\rm Tr}_\omega$ on $\mL_{\varphi_k}$ we have
\begin{equation}
{\rm Tr}_\omega (A) = \frac{(2\pi)^n}{(k+1)! \cdot d^{k+1}} \  \mathrm{Res}_k(A).
\end{equation}
\end{thm}

\begin{proof}
(i) Since $CL^{-d,k}(M) \subset L^0_{\varphi_k}(M)$, it follows from Proposition~\ref{psi_mod} and Lemma~\ref{l_exp} that $A \in \mL_{\varphi_k}$.

(ii) The leading symbol of operator $A$ is $a_{-d,k}(x,\xi) \log^k |\xi|$ where $a_{-d,k}$ is a homogeneous function of degree $-d$ in the second argument. Hence,
\begin{align*}
\int_{M} \int_{|\xi|\le n^{1/d}} a_{-d,k}(x,\xi) \log^k |\xi|\,\mathrm{d}\xi \mathrm{d}x &=
\int_{M} \int_{1\le|\xi|\le n^{1/d}} a_{-d,k}(x,\xi) \log^k |\xi|\,\mathrm{d}\xi \mathrm{d}x +O(1).
\end{align*}

Setting $\xi=|\xi|\rho$ with $\rho \in \mathbb S^{d-1}$, we have $a_{-d,k}(x,\xi) = |\xi|^{-d}a_{-d,k}(x,\rho)$ and $\mathrm{d}\xi=|\xi|^{d-1}\mathrm{d}|\xi| \mathrm{d}\rho$. Thus,
\begin{align*}
\int_{M} \int_{|\xi|\le n^{1/d}} &a_{-d,k}(x,\xi) \log^k |\xi|\,\mathrm{d}\xi \mathrm{d}x \\
&=\int_{M} \int_{\mathbb S^{d-1}} \int_{1}^{n^{1/d}} |\xi|^{-d}a_{-d,k}(x,\rho) \log^k |\xi| |\xi|^{d-1}\,\mathrm{d}|\xi| \mathrm{d}\rho \mathrm{d}x +O(1)\\
&=\int_{M} \int_{\mathbb S^{d-1}} a_{-d,k}(x,\rho)\mathrm{d}\rho \mathrm{d}x \cdot \int_{1}^{n^{1/d}} |\xi|^{-1} \log^k |\xi| \,\mathrm{d}|\xi|  +O(1)\\
&=\frac{(2\pi)^n}{(k+1)! } \  \mathrm{Res}_k(A) \cdot \frac{\log^{k+1} n^{1/d}}{k+1}+O(1)\\
&=\frac{(2\pi)^n}{(k+1)! \cdot d^{k+1}} \  \mathrm{Res}_k(A) \cdot \Phi_k( n+1)+O(1),
\end{align*}
where $\Phi_k$ is a primitive function of $\varphi_k$.

Therefore, for the complete symbol $p_A$ of $A$ we have
$$\int_{M} \int_{|\xi|\le n^{1/d}} p_A(x,\xi) \,\mathrm{d}\xi \mathrm{d}x = \frac{(2\pi)^n}{(k+1)! \cdot d^{k+1}} \  \mathrm{Res}_k(A) \cdot \Phi_k( n+1)+o(\Phi_k( n+1))$$
and
$$\frac1{\Phi_k( n+1)}\int_{M} \int_{|\xi|\le n^{1/d}} p_A(x,\xi) \,\mathrm{d}\xi \mathrm{d}x = \frac{(2\pi)^n}{(k+1)! \cdot d^{k+1}} \  \mathrm{Res}_k(A) +o(1).$$

The assertion follows from Theorem~\ref{CTT}.
\end{proof}

\subsection{Non-commutative residues and zeta-functions}

The residue $\mathrm{Res}_k$ from~\eqref{Res_k} can also be expressed in terms of the residue of a $\zeta$-function. If $M$ is a $d$-dimensional manifold, $P\in CL^{m,0}(M)$ and $A\in CL^{-d,k}(M)$, then a $\zeta$-function ${\rm Tr} (AP^{-s})$ has a meromorphic continuation to $\C$ with poles in $\{-j/m : j\in \N\}$ of order $k+1$ (see the discussion after Theorem 3.7 in~\cite{Lesch}). Moreover,
\begin{equation}
\label{resdeforeq}
\mathrm{Res}_k(A) = m^{k+1} \mathrm{res}_{k+1} {\rm Tr} (AP^{-s})\vert_{s=0},
\end{equation}
independently of $P$~\cite[Theorem 4.4]{Lesch}. Here, $\mathrm{res}_{k+1}$ is the usual $(k+1)$-st residue of a meromorphic function. In fact, in \cite{Lesch} the formula \eqref{resdeforeq} was used as a definition and~\eqref{Res_k} was established as a theorem.

Let $\varphi$ have smooth regular variation and let $G\in \Lw(L_2(M))$ be a positive weakly $\varphi$-Laplacian modulated operator with an $L_2$-symbol $p_G$ of $\varphi$-reasonable decay. Following~\cite[Definition 11.3.19]{LSZ} we define the vector-valued residue of $G$ as follows:
$$\mathrm{Res}(G) :=\left[ \left\{\frac1{\Phi (n+1)} \int_{M} \int_{|\xi|\le n^{1/d}} p_G(x,\xi)\,\mathrm{d}\xi \mathrm{d}x\right\}_{n\in\N}\right],$$
where $[\cdot]$ denotes an equivalence class in $l_\infty/c_0$.

We shall show how $\mathrm{Res}(G)$ can be computed in terms of a $\zeta$-function of $G$.

It is proved in~\cite[Theorem 3.3]{GS} that for $\varphi(t) = \frac{\log^k(\e+t)}{\e+t}$, every positive $G\in \Lw$ and every exponentiation invariant extended limit $\omega$ the following formula holds:
$${\rm Tr}_\omega (G) = \frac1{(k+1)!} \omega \left( \frac{{\rm Tr} (G^{1+1/\log n})}{\Phi(n)} \right).$$
If $\mathrm{Res}(G)$ is scalar-valued (that is, if the sequence $\mathrm{Res}(G)$ converges), then it follows from Theorem~\ref{CTT}, properties of extended limits and the preceding formula that
$$\mathrm{Res}(G) = \frac1{(k+1)!} \lim_{n\to\infty} \frac{{\rm Tr} (G^{1+1/\log n})}{\Phi(n)}
=\frac1{(k+1)!} \lim_{s\to1+} (s-1)^{k+1}{\rm Tr} (G^s),$$
since $\Phi(t)=\log^{k+1}(\e+t).$

\begin{rem}
\label{remfromDima}
Although the residue can be evaluated via the asymptotic of the $\zeta$-function $\zeta_G(s):={\rm Tr} (G^s)$ for a positive $G\in \Lw$, it is not guaranteed that $\zeta_G$ is meromorphic at $s=1$. In particular, if $G$ satisfies that $\lambda_n(G)=\frac{\log^k(2+n)}{2+n}$, $n\ge0$, then $\zeta_G(s)$ is \emph{not} analytic in any punctured neighbourhood of $s=1$. We shall demonstrate this for $k=1$. So, the question is whether the function
$$s\to\sum_{n=1}^{\infty}\left(\frac{\log(n)}{n}\right)^s$$
admits an analytic continuation to a punctured neighbourhood of $s=1$ and has a pole of second order in $s=1$.

Consider the analytic functions $\alpha_n$ defined near $1$ by the formula
$$\alpha_n(s)=\left(\frac{\log(n)}{n}\right)^s-\int_n^{n+1}\left(\frac{\log(t)}{t}\right)^s\,\mathrm{d}t.$$
It is easy to see that
$$|\alpha_n(s)|\leq {\rm const}\cdot\frac{\log^{\Re(s)}(n)}{n^{1+\Re(s)}}.$$
Hence, the series $\sum_{n=1}^{\infty}\alpha_n(s)$ converges uniformly and is, therefore, analytic near $s=1.$

It, therefore, suffices to consider the behaviour of the function
$$s\to\sum_{n=1}^{\infty}\int_n^{n+1}\left(\frac{\log(t)}{t}\right)^s\,\mathrm{d}t=\int_1^{\infty}\left(\frac{\log(t)}{t}\right)^s\,\mathrm{d}t.$$
Making the substitution $t=e^u,$ we infer that
$$\int_1^{\infty}\left(\frac{\log(t)}{t}\right)^s\,\mathrm{d}t=\int_0^{\infty}u^se^{-u(s-1)}du=(s-1)^{-s-1}\Gamma(1+s).$$
The latter function is analytic at $\mathbb{C}\backslash(-\infty,1].$ It cannot be extended to an analytic function in a punctured neighbourhood of $s=1,$ so formally the residue is not defined. However, the limit
$$\lim_{s\to1+}(s-1)^2\sum_{n=1}^{\infty}\left(\frac{\log(n)}{n}\right)^s$$
exists and equals $1.$
\end{rem}

\section{Examples from noncommutative geometry.}
\label{sec:examples}

One of the motivations to consider more general weak ideals and to study rather general pseudo-differential operators as in Section \ref{psipseudos}, comes from recent constructions in noncommutative geometry. We review these constructions here, how they fit into the machinery of Section \ref{psipseudos} and some peculiar consequences of Corollary \ref{CTT_psipsido} in this context. 

The fundamental objects of study in noncommutative geometry are spectral triples. A spectral triple $(\mathcal{A},\mathcal{H},D)$ consists of a Hilbert space $\mathcal{H}$, a $*$-subalgebra $\mathcal{A}\subseteq \mathbb{B}(\mathcal{H})$ and a self-adjoint operator $D$, densely defined on $\mathcal{H}$, such that $a(i\pm D)^{-1}\in \mathbb{K}(\mathcal{H})$ for $a\in \mathcal{A}$ and $\mathcal{A}\subseteq \mathrm{Lip}(D)$. Here
$$\mathrm{Lip}(D):=\{T\in \mathbb{B}(\mathcal{H}): \; T\mathrm{Dom}(D)\subseteq \mathrm{Dom}(D)\;\mbox{ and }\; [D,T] \; \mbox{ is norm bounded}\}. $$
The condition $a(i\pm D)^{-1}\in \mathbb{K}(\mathcal{H})$ is an ellipticity type condition, guaranteeing that $D$ has ``locally compact" resolvent relative to $\mathcal{A}$. The fact that $[D,a]$ is bounded for all $a\in \mathcal{A}$ is a ``differentiability" condition, indeed by \cite[Theorem 2.4]{ChristCom}, $\mathrm{Lip}(D)$ consists of operators $T$ such that $t\mapsto \mathrm{e}^{itD}T\mathrm{e}^{-itD}$ is weakly differentiable at $t=0$.

The prototypical example of a spectral triple is on a complete Riemannian manifold $M$ for a choice of Dirac operator $\slashed{D}$ acting on a Clifford bundle $S\to M$, the triple $(C^\infty_c(M),L_2(M,S),\slashed{D})$ is a spectral triple. In several applications, it is of interest to understand the dynamics of a (semi)group action. This is done by incorporating the (semi)group action $\Gamma\curvearrowright M$ into the spectral triple by considering the crossed product $C(M)\rtimes \Gamma$. It is problematic to incorporate the (semi)group action yet retaining bounded commutators with $\slashed{D}$ unless the action is isometric. In \cite{CM95} this problem is solved by leaving $M$ and going to a frame bundle. Recent works \cite{DGMW16,GM16,MS16} show that upon applying functional calculus by a logarithm to $\slashed{D}$, a spectral triple pertaining several properties of $M$ can be constructed on $C(M)\rtimes \Gamma$. In this section we show how  the pseudo-differential techniques developed in this paper are relevant to the study of such noncommutative geometries. Some results of this section appeared in~\cite{DGMW16}. Here we use different techniques simplifying several proofs.

\subsection{Setup for the Dirac operator} 

We consider a Dirac operator $\slashed{D}$ acting on a Clifford bundle $S\to M$ over a complete Riemannian manifold $M$. The operator $\slashed{D}$ is essentially self-adjoint for the core $C^\infty_c(M,S)$ by \cite{CheSA}, and we use $\slashed{D}$ as to denote the self-adjoint operator obtained from the graph closure of $C^\infty_c(M,S)$. Consider the globally Lipschitz function 
\begin{equation}
\label{functionh}
h:\R\to \R, \quad h(s):=\mathrm{sign}(s)\log(1+s^2).
\end{equation}
We introduce the notation 
\begin{equation}
\label{varphinotation}
\varphi_{m,k}(t)=(\mathrm{e}+t)^m(\log(\mathrm{e}+t))^k, \quad m\in \mathbb{R}, \; k\in \mathbb{Z},
\end{equation}
and note that $L^0_{\varphi_{m/d,k}}(M)=L^m_{\varphi_{0,k}}(M)$ for any  $m\in \mathbb{R}$ and $k\in \mathbb{Z}$. For a vector bundle $S\to M$ we can define classes $L^m_\varphi(M,S)$ of $\varphi$-pseudo-differential operators acting $C^\infty_c(M,S)\to C^\infty(M,S)$, in the same way as in Section \ref{psipseudos}.

We use $c:T^*M\to \mathrm{End}(S)$ to denote Clifford multiplication, so $c_x(\xi)\in \mathrm{End}(S_x)$ is the endomorphism defined by Clifford multiplying by $\xi\in T^*_xM$.

\begin{prop}
\label{propforD}
The operator 
$$D:=h(\slashed{D}),$$ 
satisfies the following 
\begin{enumerate}
\item $D\in L^0_{\varphi_{0,1}}(M)$ and $D-Op(p)\in L^{-1}_{\varphi_{0,1}}(M,S)$ where 
$$p(x,\xi):=c_x(\xi)|\xi|_x^{-1}\log(1+|\xi|^2_x), \ |\xi|>1.$$
\item As an operator on $L_2(M,S)$, $D$ is self-adjoint, $(i\pm D)^{-1}\in L^0_{\varphi_{0,-1}}(M)$ and $(i\pm D)^{-1}-Op(p^{-1})\in L^{-1}_{\varphi_{0,-1}}(M,S)$ for $p$ as in (1).
\item $a(i\pm D)^{-1}\in \mathcal{L}_{\varphi_{0,-1}}(L_2(M,S))$ for any $a\in C^\infty_c(M)$.
\end{enumerate}
\end{prop}

\begin{proof}
For the Laplacian $\Delta=-\slashed{D}^2$ and the operator $F:=\slashed{D}|\slashed{D}|^{-1}\in L^0(M,S)$, we have that $D=F\log(1-\Delta)$. Here $\slashed{D}|\slashed{D}|^{-1}$ is defined from functional calculus using the sign function defined as $\mathrm{sign}(x):=x|x|^{-1}$ for $x\neq 0$ and $\mathrm{sign}(0):=0$, the sign function is continuous on the spectrum of $\slashed{D}$ because the spectrum is discrete. It follows from Lemma \ref{philapuptolot} (see page \pageref{philapuptolot}) that $D\in L^0_{\varphi_{0,1}}(M)$ is of the form $Op(p)+L^{-1}_{\varphi_{0,1}}(M,S)$. Therefore (1) follows. Since $D$ is constructed by functional calculus applied to the self-adjoint $\slashed{D}$, (2) follows from the composition formula \eqref{productform} (see page \pageref{productform}). From (2) and Corollary \ref{corofphidifflem} (see page \pageref{corofphidifflem}) we deduce (3).
\end{proof}

\begin{rem}
The function $h$ in Equation \eqref{functionh} is not smooth at $s= 0$ but satisfies the estimates 
$$|\partial_s^kh(s)|\le c_k \langle s\rangle^{-k},$$
for some constant $c_k$, any $k>0$ and $|s|>1$. Therefore, $h$ differs from a H\"ormander symbol $h_0$ of any order $m>0$ by a compactly supported Lipschitz function. The difference $D-h_0(\slashed{D})$ is therefore smoothing.
\end{rem}

\begin{prop}
\label{thepresepc}
Assume that $M$ is a closed manifold and that $D$ is as above. Then $(C^\infty(M),L_2(M,S),D)$ is a spectral triple satisfying the following:
\begin{enumerate}
\item  $(C^\infty(M),L_2(M,S),D)$ is $\theta$-summable, i.e. $\mathrm{e}^{-tD^2}\in \mathcal{L}_1(L_2(M,S))$ for any $t>0$;
\item $[D,a]\in L_{\varphi_{-1/d,1}}^0(M,S)\subseteq \mathcal{L}_{\varphi_{-1/d,1}}(L_2(M,S))$ for any $a\in C^\infty(M)$;
\item It holds that 
$$[(C^\infty(M),L_2(M,S),D)]=[(C^\infty(M),L_2(M,S),\slashed{D})],$$
in the $K$-homology group $K^*(C(M))$.
\end{enumerate}
\end{prop}

\begin{proof}
If $(C^\infty(M),L_2(M,S),D)$ is a well defined spectral triple, (3) is immediate from the fact that $D|D|^{-1}=\slashed{D}|\slashed{D}|^{-1}$. By Proposition \ref{propforD}, $D\in L^0_{\varphi_{0,1}}(M,S)\subseteq \cap_{m>0}L^m(M,S)$ so $[D,a]\in L_{\varphi_{-1/d,1}}^0(M,S)\subseteq \mathcal{L}_{\varphi_{-1/d,1}}(L_2(M,S))$ for any $a\in C^\infty(M)$ by the composition formula \eqref{productform} (see page \pageref{productform}). Therefore, $(C^\infty(M),L_2(M,S),D)$ is a well defined spectral triple as soon as $\mathrm{e}^{-tD^2}\in \mathcal{L}_1(L_2(M,S))$ which implies that $D$ has compact resolvent.

The Weyl law guarantees that for a constant $C$, the eigenvalues of $D$ satisfies $|\lambda_k(D)|\geq C\log(1+k^2)$. We therefore have that $\mathrm{Tr}\,(\mathrm{e}^{-tD^2})\leq \sum_{k=0}^\infty \mathrm{e}^{-tC^2\log^2(1+k^2)}$ which is finite for all $t>0$ by an elementary integral estimate.
\end{proof}

This proposition shows that $D$ defines a spectral triple on $C^\infty(M)$ whose topological features are the same as $\slashed{D}$. The spectral features of $D$ sees $(C^\infty(M),L_2(M,S),D)$ as infinite-dimensional, through its $\theta$-summability, but the finite-dimensionality is remembered in the commutators $[D,a]\in \mathcal{L}_{\varphi_{-1/d,1}}(L_2(M,S))$ for $a\in C^\infty(M)$. We call $D$ the \emph{logarithmic dampening} of $\slashed{D}$.

\subsection{Setup for the group action}

Interesting features arise when we consider an action on $M$. Let $f:M\to M$ be a smooth mapping. For now, we do {\bf not} assume $f$ to be a diffeomorphism, or even a homeomorphism. We shall however assume the following properties.

\begin{dfn}
Let $(M,g)$ be a Riemannian manifold, $S\to M$ a Clifford bundle and $f:M\to M$ smooth.
\begin{enumerate}
\item We say that $f$ acts conformally if there is a $c_f\in C^\infty(M,\R_{>0})$ such that 
$$f^*g=c_fg.$$
\item If $f$ acts conformally, we say that $f$ lifts to $S$ if there is a unitary isomorphism of Clifford bundles
$$u_f:f^*S\to S.$$
\end{enumerate}
\end{dfn}

We remark at this point that if $f$ acts conformally, the inverse mapping theorem guarantees that $f$ is a local diffeomorphism. We could possibly relax the assumption for $f$ to be conformal in defining the notion of $f$ lifting to $S$, but we will only use them in combination. 

For a closed manifold $M$ and a mapping $f:M\to M$ that acts conformally and lifts to $S$, we define the operator
$$V_f:L_2(M,S)\to L_2(M,S), \quad V_f\xi:=c_f^{d/4}n^{-1/2}u_f(\xi\circ f).$$
Here $n\in C^\infty(M,\N)$ is the function defined by $n(x):=\#f^{-1}(\{x\})$; this number is well defined since $M$ is closed and is locally constant since $f$ is a local diffeomorphism.

\begin{prop}
\label{theisometryandd}
The operator $V_f$ is an isometry preserving $\mathrm{Dom}(D)$ such that 
$$V_fa=(a\circ f) V_f, \ a \in C(M).$$
Moreover, $V_f^*DV_f-D\in L^0(M,S)$ so $[D,V_f]$ is bounded in the $L_2$-norm.
\end{prop}

We omit the proof of this result here and refer the reader to \cite[Proposition 8.3 and 8.12]{DGMW16}. The gist of the proof is that $D$ and $V_f$ have bounded commutators because the symbol $p$ from Proposition \ref{propforD} satisfies the property 
$$p(x,\xi)-p(f(x),(Df)^T\xi)\in S^0(M).$$
The fact that $V_f$ preserves $\mathrm{Dom}(D)$ follows from the fact that $V_f$ preserves the core $C^\infty(M,S)$ for $D$ and the boundedness of the commutator $[D,V_f]|_{C^\infty(M,S)}$.

\subsection{The spectral triple}

We can now construct the relevant spectral triple on an algebra generated by functions on the closed manifold $M$ and by translation operators such as $V_f$. 

\begin{thm}
\label{thmforspec}
Let $M$ be closed manifold and $S\to M$ a Clifford bundle with $D$ as above. Assume that $\mathcal{A}\subseteq \mathbb{B}(L_2(M,S))$ is the unital $*$-algebra generated by $C^\infty(M)$ and a collection $\{V_{f_i}\}_{i\in I}$ for a collection of smooth mappings $f_i:M\to M$, $i\in I$, that act conformally and lift to $S$. Let $A$ denote the norm closure of $\mathcal{A}$. The data $(\mathcal{A},L_2(M,S),D)$ is a spectral triple such that 
\begin{enumerate}
\item  $(\mathcal{A},L_2(M,S),D)$ is $\theta$-summable, i.e. $\mathrm{e}^{-tD^2}\in \mathcal{L}_1(L_2(M,S))$ for any $t>0$;
\item $[D,a]\in L_{\varphi_{-1/d,1}}^0(M,S)\subseteq \mathcal{L}_{\varphi_{-1/d,1}}(L_2(M,S))$ in the subalgebra $a\in C^\infty(M)\subseteq \mathcal{A}$, and $V_{f_i}^*DV_{f_i}-D\in L^0(M,S)$ for any $i\in I$;
\item If $\iota:C(M)\to A$ denotes the inclusion, then 
$$\iota^*[(\mathcal{A},L_2(M,S),D)]=[(C^\infty(M),L_2(M,S),\slashed{D})],$$
in the $K$-homology group $K^*(C(M))$.
\end{enumerate}
\end{thm} 

The theorem follows directly from the Propositions \ref{thepresepc} and \ref{theisometryandd}.

\begin{rem}
We remark here that the $C^*$-algebra $A$ can be rather exotic. For example, many (iterated) Cuntz-Pimsner algebras have the form that the $C^*$-algebras in Theorem \ref{thmforspec} have. For more examples, see \cite{MS16, DGMW16}. An important property of these $C^*$-algebras is that they can be purely infinite, and would not allow {\bf any} finitely summable spectral triple by \cite{C89}. In particular, the lack of tracial states on $A$ obstructs $\slashed{D}$ having bounded commutators with $V_f$ (under suitable assumptions on $f$).
\end{rem}

\subsection{Computations with singular traces}

For a closed manifold $M$, Connes' trace formula for $\varphi$-pseudo-differential operators, from Corollary \ref{CTT_psipsido}, has several interesting consequences for the operator $D:=h(\slashed{D})\in L^0_{\varphi_{0,1}}(M,S)$, where $\varphi_{0,1}(t)=\log(\e+t)$ as above. The reader is encouraged to recall the notation from Equation \eqref{varphinotation} (see page \pageref{varphinotation}). 

In theorem below the notation ${\rm Tr}_{\omega, \varphi}$ stands for the Dixmier trace on $\Lw$ constructed via the extended limit $\omega$.

\begin{thm}
\label{compsngtrace}
Let $M$ be a closed $d$-dimensional manifold and $D$ as above acting on sections of the Clifford bundle $S\to M$. For a classical $T\in CL^0(M,S)$ and $a_1, a_2,\ldots, a_d\in C^\infty(M)$, 
\begin{align*}
T[\slashed{D},a_1]\cdots [\slashed{D},a_d](1+\slashed{D}^2)^{-d/2}&\in L^{-d}(M,S)=L^{0}_{\varphi_{-1,0}}(M,S)\subseteq \mathcal{L}_{1,\infty}(L_2(M,S))\quad\mbox{and}\\ 
T[D,a_1]\cdots [D,a_d]&\in L^{0}_{\varphi_{-1,d}}(M,S)\subseteq \mathcal{L}_{\varphi_{-1,d}}(L_2(M,S)).
\end{align*}
The corresponding Dixmier traces of these operators are proportional and can be computed using the principal symbol $\sigma_T\in C^\infty(S^*M,\mathrm{End}(S))$ of $T$ by
\begin{align*}
{\rm Tr}_{\omega,\varphi_{-1,0}}&(T[\slashed{D},a_1]\cdots [\slashed{D},a_d](1+\slashed{D}^2)^{-d/2})=d^{d}\cdot {\rm Tr}_{\omega,\varphi_{-1,d}}(T[D,a_1]\cdots [D,a_d])=\\
&=\frac1d \int_{S^*M} {\rm Tr}_S\left(\sigma_T(x,\xi)c_x(\mathrm{d} a_1)\cdots c_x(\mathrm{d} a_d)\right)\mathrm{d}\xi \mathrm{d}x.
\end{align*}
\end{thm}

\begin{proof}
We first compute the symbol of $[D,a]$ for $a\in C^\infty(M)$ modulo lower order terms. By Proposition \ref{propforD}, we can replace $D$ by $Op(p)$ where $p(x,\xi):=c_x(\xi)|\xi|_x^{-1}\log(1+|\xi|^2_x)$. Using the product formula \eqref{productform} (see page \pageref{productform}), we can write $[Op(p),a]\sim Op(b)$ where 
$$b=i\{p,a\} \mod  S^{-1}_{\varphi_{0,1}}(M,S),$$
and $\{\cdot,\cdot\}$ denotes the Poisson bracket on $C^\infty(T^*M)$. One computes that 
$$i\{p,a\} =c_x(\mathrm{d} a)|\xi|^{-1}_x\log(1+|\xi|_x^2)\mod  S^{-1}_{\varphi_{0,1}}(M,S).$$
Therefore, the symbol of $T[D,a_1]\cdots [D,a_d]$ is 
$$\sigma_T(x,\xi)c_x(\mathrm{d} a_1)\cdots c_x(\mathrm{d} a_d)\frac{\log^d(1+|\xi|^2)}{(1+|\xi|^2)^{d/2}} \mod S^{-1}_{\varphi_{-1,d}}(M,S).$$
The theorem now follows from Theorem \ref{cor_logcl} using that the symbol of $T[\slashed{D},a_1]\cdots [\slashed{D},a_d](1+\slashed{D}^2)^{-d/2}$ is 
$$\sigma_T(x,\xi)c_x(\mathrm{d} a_1)\cdots c_x(\mathrm{d} a_d)(1+|\xi|^2)^{-d/2} \mod S^{-d-1}(M,S)=S^{-1}_{\varphi_{-1,0}}(M,S).$$
\end{proof}

It is not clear to the authors if the identities of Theorem \ref{compsngtrace} have any deeper meaning and how they extend to the algebra $\mathcal{A}$ of Theorem \ref{thmforspec}. In this context it would be interesting to ask for a reasonable ``local" approach to index theory for logarithmically dampended operators. The $\varphi$-pseudo-differential operators are pseudo-differential operators and their index theory is clear from the Atiyah-Singer index theorem, but in the context of the algebra $\mathcal{A}$ from Theorem \ref{thmforspec}, or even more exotic situation, it is not clear how to proceed. Here we mean ``local" in the sense of \cite{CM95}, i.e. computable by residues and singular traces.

\end{document}